\documentclass{article}
\usepackage[english]{babel}
\usepackage[latin1]{inputenc}
\usepackage{latexsym}
\usepackage{amssymb,amsmath,dsfont}
\usepackage{subfigure}
\usepackage{graphicx,float}
\usepackage{color}
\usepackage{url}
\usepackage{authblk}
 
\usepackage{array}    
\usepackage[table]{xcolor}
\usepackage{amsthm}
\usepackage{a4}
\usepackage{fullpage}
\usepackage[unicode]{hyperref}
\hypersetup{
    bookmarks=true,         % show bookmarks bar?
    unicode=true,          % non-Latin characters in Acrobat bookmarks
    pdftoolbar=true,        % show Acrobat toolbar?
    pdfmenubar=true,        % show Acrobatt menu?
    pdffitwindow=true,      % page fit to window when opened
    pdftitle={FKA},    % title
    pdfauthor={C. Bourdarias, M. Ersoy and S. Gerbi},     % author
    pdfsubject={},   % subject of the document
    pdfcreator={},   % creator of the document
    pdfproducer={}, % producer of the document
    pdfkeywords={}, % list of keywords
    pdfnewwindow=true,      % links in new window
    colorlinks=true,       % false: boxed links; true: colored links
    linkcolor=blue,  
    citecolor=blue,        % color of links to bibliography
    filecolor=magenta,      % color of file links
    urlcolor=cyan           % color of external links
}
\newtheorem{thm}{Theorem}[section]
\newtheorem{rque}{Remark}[section]
\newtheorem{prop}{Proposition}[section]

\def\ds{\displaystyle}
\def\abs#1{\vert #1 \vert}

\newcommand{\R}{\mathbb{R}}

\newcommand{\N}{\mathbb{N}}
\newcommand{\F}{\boldsymbol{F}}
\newcommand{\UU}{\boldsymbol{U}}
\newcommand{\VV}{\boldsymbol{V}}
\newcommand{\WW}{\boldsymbol{W}}

\newcommand{\FKAc}{{\sc FKA }}
\newcommand{\FKAl}{{\sc F}ull {\sc K}inetic {\sc A}pproach}
\def\ut{\widetilde{u}}
\def\At{\widetilde{A}}
\def\ct{c(\widetilde \VV)}

\def\M{\mathcal M}
\def\dsp{\displaystyle}
\def\abs#1{\vert #1 \vert}

\def\vecdeux#1#2{\left(\begin{array}{cc}
           #1\\
           #2
          \end{array}
          \right)
         }
\title{Unsteady mixed flows in non uniform closed water pipes: a  \FKAl.}
\author[1]{{\sc{C. Bourdarias}} \thanks{Christian.Bourdarias@univ-savoie.fr}}
\author[2,3]{{\sc{M. Ersoy }}\thanks{Mehmet.Ersoy@univ-tln.fr}}
\author[1]{{\sc{S. Gerbi}}\thanks{Stephane.Gerbi@univ-savoie.fr}}
\affil[1]{\small{Laboratoire de Math\'ematiques, UMR 5127 - CNRS and Universit\'e de Savoie, 73376 Le Bourget-du-Lac Cedex, France.}}
\affil[2]{\small{BCAM--Basque Center for Applied Mathematics, Bizkaia Technology Park 500, 48160, Derio, Basque Country, Spain.}}
\affil[3]{\small{\emph{Present address:} Université de Toulon, IMATH, EA 2134, 83957 La Garde, France.}}

\begin{document}
\date{}
\maketitle
% \tableofcontents
\begin{abstract} 
We recall the \textbf{PFS} (\textbf{P}ressurized  and  \textbf{F}ree  \textbf{S}urface) model constructed for the modeling of unsteady mixed flows in closed water pipes where
transition points between the free surface and pressurized flow are treated as a free boundary associated to a discontinuity of the gradient of pressure.
Then we present a numerical kinetic scheme  for the  computations of unsteady mixed flows in closed water pipes. 
This kinetic method that we call \FKAc  for ``\FKAl''  is an easy and mathematically elegant way to deal with multiple transition points when the changes of state
between free surface and pressurized flow occur. We use two approaches namely the ``ghost waves approach'' and the ``\FKAl''  to treat these transition points.
We show that this kinetic numerical scheme has the following properties: it is wet area conservative, under a CFL condition it preserves the wet area positive,
it treats ``naturally'' the  flooding zones and most of all it is very easy to implement it. 
Finally numerical experiments versus laboratory experiments are presented and the scheme produces results that are in a very good agreement. 
We also present a numerical comparison with analytic solutions for free surface flows in non uniform pipes: the numerical scheme has a very good behavior.
A code to code comparison for pressurized flows is also conducted and leads to a very good agreement. 
We also perform a numerical experiment  when flooding and drying flows may occur and  finally make a numerical study of the order of the kinetic method.
\end{abstract}

\vspace*{0.5cm}

%%%%%%%%%%%%%%%%%%%%%%%%%%%
{\bf Notations concerning geometrical quantities} \\
%%%%%%%%%%%%%%%%%%%%%%%%%%
\begin{tabular}{p{0.1\linewidth}p{0.8\textwidth}}
%%%%%%%%%%%%%%%%%
$\theta(x)$ & angle of the inclination of the main pipe axis $z = Z(x)$ at position $x$ \tabularnewline
$\mathbf{Z}(t,x)$ & dynamic topography \tabularnewline
$\Omega(x)$ & cross-section area of the pipe orthogonal to the axis $z=Z(x)$\tabularnewline
$S(x)$ &area of $\Omega(x)$\tabularnewline
$R(x)$& radius of the cross-section $\Omega(x)$\tabularnewline
$\sigma(x,z)$ & width of the cross-section $\Omega(x)$ at  altitude $z$
\end{tabular}
\newpage
%%%%%%%%%%%%%%%%%
{\bf Notations concerning the \textbf{PFS} model} \\
%%%%%%%%%%%%%%%%%
\begin{tabular}{p{0.1\linewidth}p{0.8\textwidth}}
$p(t,x,y,z)$& pressure \tabularnewline
$\rho_0$& density of the water at atmospheric pressure $p_0$ \tabularnewline
$\rho(t,x,y,z)$ & density of the water at the current pressure\tabularnewline
 $\dsp\overline{\rho}(t,x)$ & $\dsp\overline{\rho}(t,x)= \frac{1}{S(x)}\int_{\Omega(x)} \rho(t,x,y,z)\,dy\,dz$ 
 is the mean value of $\rho$ over $\Omega(x)$ (press. flows)\tabularnewline
$c$ & sonic speed\tabularnewline
$S_{w}(t,x)$& wet area i.e. part of the cross-section area in contact with water. $S_{w} =S(x)$ if the flow is pressurized\tabularnewline
$\dsp A(t,x)$ &  $\dsp A(t,x)=  \frac{\overline{\rho}(t,x)}{\rho_0} S_{w}(t,x) $ is the ``equivalent wet area''. $S_{w}=A(t,x)$ if the
flow is free surface\tabularnewline
 $u(t,x)$ & velocity\tabularnewline
$Q(t,x)$& $Q(t,x) = A(t,x) u(t,x)$ is the discharge\tabularnewline
$E$ & state indicator. $E=0$ if the flow is free surface, $E=1$ otherwise\tabularnewline
$\mathcal{H}(S_{w})$& the $Z$-coordinate of the water level equal to  $\mathcal{H}(S_{w})=h(t,x)$  if the state is free surface, $R(x)$ otherwise \tabularnewline
$p(x,A,E)$& mean pressure over $\Omega$\tabularnewline
$K_s>0$ & Strickler coefficient depending on the material \tabularnewline
$P_m(A)$& wet perimeter of $A$ (length of the part of the channel section in contact with the water) \tabularnewline
$R_h(A)$& $R_h(A) = \dsp \frac{A}{P_m(A)}$ is hydraulic radius
\end{tabular}

Bold characters are used for vectors, except for $\boldsymbol{Z}$, the dynamic topography defined later.
%%%%%%%%%%%%%%%%%
% SECTION	Introduction
%%%%%%%%%%%%%%%%%
\section{Introduction}
The presented work takes place in a more general framework: the modeling of unsteady mixed flows in any kind of closed water pipes taking into 
account  the cavitation problem and air entrapment. We are interested in flows occurring in closed pipes with non uniform sections, where some parts of the flow
can be free surface (it means that only a part of the pipe is filled) and other parts are pressurized (it means that the pipe is full). 
The transition phenomenon between the two types of flows occurs in many situations such as storm sewers, waste or supply pipes in hydroelectric
installations. It can be induced by sudden changes in the boundary conditions as failure pumping.
During this process, the pressure can  reach severe values and  may cause damages.
The simulation of such a phenomenon is thus a major challenge and a great amount of works was devoted to it these last
years (see \cite{CSZ97,N90,R85,WS93}, and references therein).

The classical shallow water equations are commonly used to describe  free surface flows in open channels. 
They are also used in the study of mixed flows  using the Preissman slot artefact (see for example \cite{CSZ97,WS93}). 
However, this technic does not take into account  the  subatmospheric pressurized flows (viewed as a free surface flow) which occur
during a water hammer. In recent works,  \cite{KAEDP09,KAEDP11_1,KAEDP11_2}, a model for mixed flows in closed water pipes has been developed at 
University of Li\`{e}ge, where they use the artifact of the Preismann slot for supatmospheric pressurized flow and 
by introducing the concept of ``negative Preissman slot'' for subatmospheric pressurized flow.
They proposed also a numerical scheme to compute  the stationary flow, as well as the unsteady flow.

On the other hand  the Allievi equations, commonly used to describe pressurized flows, 
are written in a non-conservative form which is not well adapted to a natural coupling with the shallow water equations.

A model for the unsteady mixed water flows in closed pipes, the \textbf{PFS} model,  and a finite volume
discretisation have been proposed by the authors in \cite{BEG09} and its mathematical derivation from the Euler incompressible equations
(for the free surface part of the flow) and from the Euler isentropic compressible equations (for the pressurized part of the flow) is proposed in \cite{BEG12}.
This model and the finite volume scheme extend the model studied by two of the authors for uniform pipes \cite{BG07}.
In \cite{BG09} two of the authors has constructed a kinetic numerical scheme to compute  pressurized flows in uniform pipes. 
For the case of a non uniform  closed pipe and for pressurized flow, the authors has extended the previous kinetic numerical scheme with reflections, 
see \cite{BEG09_1}. Let us also mention that the construction of a kinetic numerical scheme with a correct treatment of all the source terms has been 
published recently \cite{BEG11_2}.

The paper is organized as follows. In  the second section, we recall the \textbf{PFS} model  and 
focus on the  continuous flux whose gradient is discontinuous at the interface between free surface and pressurized flow.
The source terms are also highlight: the conservative ones, the non conservative ones and the source term which is neither conservative nor conservative.
We use the definition of the DLM theory \cite{DLM95} to define the non-conservative products.
We state in this section  the theoretical properties of the system that must be preserved by the numerical scheme.

Section 3 is devoted to the kinetic interpretation of the \textbf{PFS} model thanks to the classical kinetic interpretation of the system (see \cite{PS01} for instance). 

In section 4, we construct the kinetic scheme for the \textbf{PFS} model. The particular treatment of the friction term which is neither conservative nor non-conservative will be upwinded using the notion of \textbf{\textit{ the dynamic topography}}, already introduced by the authors in recent works \cite{BEG09,BEG11_2}. 
Firstly, we use the same kinetic scheme with reflections that we have constructed in 
\cite{BG09,BEG09_1,BEG11_2} to treat the part of the flow where no transition points are present.
Then we treat the transition points by two ways: 
\begin{itemize}
\item as in \cite{BEG09}, the ``ghost waves approach'' is used. We make an assumption on the speed of the discontinuity
between free surface and pressurized flow, to compute the macroscopic states at the right hand side and the left hand side of this discontinuity.
For this sake, we treat the transition points at the macroscopic level.
\item a new approach that we called the ``\FKAl'' is then used to treat these transition points. We stay at the microscopic level to build the macroscopic states
at the right hand side and the left hand side of this discontinuity.
\end{itemize}
The particular treatment of the boundaries of the pipes is treated. Let us emphasize that the novelty in this work comes from the fact that the numerical scheme treats the
transition points as well as the boundary conditions at the microscopic level so that a uniform approach is made possible. 

In the last section, we present numerical experiments: the first one is the so-called Wiggert's test where we have experimental data to compare with.
A very good agreement is shown. Then we perform a code to code comparison for pressurized flow: we compare the results of the \verb+belier+ code used
by the engineers of  Electricit\'e de France, Centre d'Ing\'enierie Hydraulique, Chamb\'ery, (EDF-CIH)to compute a numerical solution of the Allievi equations by the 
characteristics method with the one we implemented, called \verb+FlowMix+, for the same engineers  for the computation of mixed flows.

Then we focus our attention in the numerical computations of steady states. This is due to the fact that we used of a very simple ``maxwellian''
function so that every computations of microscopic quantities are available exactly. 
This conducts to a very easy implementation of the numerical code which has the ambition to be exploited in an industrial way.
Unfortunately, the use of such a function does not permit to prove mathematically that the numerical scheme permits the computations of steady states. 
Nethertheless, we compare the behavior of the numerical scheme towards the analytic transcritical steady solution of free surface flows in non uniform pipe: 
the results are in a very good agreement. Then a mixed ``numerical'' steady state is computed and again the numerical scheme shows a very good agreement. 
Moreover, let us say that this numerical scheme seems very robust
since it is used ``everyday'' in an industrial way by the engineers of EDF-CIH in a lot of different configurations and they are confident in the numerical results.
We will also test the robustness of the code  \verb+FlowMix+ on a drying and flooding flow.
The finite volume version of the method we have presented in \cite{BEG09} could not treat this type of flow unless by the introduction of a cut-off function that will produce a lack of conservation of mass. Finally we perform a numerical study of the order of the method on a unsteady mixed flow (computed by the VFRoe
solver that we have constructed and validated in \cite{BEG09}) which will converge to a steady mixed flow.

In a similar framework, in \cite{KLT08}, Euler equations for compressible fluids in a nozzle with variable  discontinuous cross-section are considered. 
Regarding these equations as a nonconservative hyperbolic system,  weak solutions in the sense of Dal Maso, LeFloch and Murat \cite{DLM95} are investigated. 
A  fully conservative  entropy equality is derived and the authors construct well-balanced numerical scheme preserving the minimum entropy principle (see also \cite{LT11,KT05}). 

For the sake of simplicity, we do not deal with the deformation of the domain induced by the change of pressure. 
We will consider only an infinitely rigid pipe (see \cite{BG08} for unsteady pressurized flows in deformable closed pipe).

%%%%%%%%%%%%%%%%%%%%%%%%%%%%%%%%%%%%%%%%%%%%%%%%%%%%%%%%%%%%%%
\section{A model for  unsteady water flows in closed water pipe}\label{SectionAModelForUnsteadyWaterFlowsPipes}
%%%%%%%%%%%%%%%%%%%%%%%%%%%%%%%%%%%%%%%%%%%%%%%%%%%%%%%%%%%%%%
Although, in recent works (see \cite{BEG09, BEG12}), we have derived and studied a model for mixed flows in closed water pipes that we
called  the \textbf{PFS} model, for the sake of completeness of the present work, we briefly recall this model and its mathematical properties.

The \textbf{PFS} model (see \cite{BEG09, BEG12, TheseErsoy})   is a mixed model of a pressurized (compressible) and free surface  
(incompressible) flow in a one dimensional  rigid pipe with variable cross-section. The pressurized parts of the flow correspond 
to  a full pipe whereas  the section is not completely filled for the free surface flow. 

The free surface part of the model is derived  by writing the $3$D Euler
incompressible equations and by averaging over orthogonal sections to the privileged axis of the flow, the pressure being the hydrostatic pressure defined by:
\begin{equation}\label{pressure_FS}
P(t,x,z) = \rho_{0} g (h(t,x) - z(x)) \cos \theta(x) ,
\end{equation}
where $g$ is the gravity constant, $\theta(x)$  the inclination of the pipe, $\rho_{0}$ is  the density of the water at normal atmospheric conditions,
$h(t,x)$ is the water height of the free surface whereas  $z(x)$ is the altitude of the bottom of the pipe.

In the same spirit, by writing the Euler isentropic and compressible equations with the linearized pressure law 
\begin{equation}\label{pressure_P}
\dsp P(t,x,y,z)=p_0 + c^2(\rho(t,x,y,z) - \rho_0) ,
\end{equation}
where $c$ the sonic speed of the water (assumed to be constant), $\rho$ is  the density of the water,
we obtain a Saint-Venant like system of equations in the ``FS-equivalent'' variables 
$\dsp A(t,x) = \frac{\overline{\rho}(t,x)}{\rho_0} S(x)$, $Q(t,x)=A(t,x) u(t,x)$ 
which takes into account the compressible effects (for a detailed derivation, see \cite{BEG09, BEG12, TheseErsoy}).

These variables are suitable  to study mixed flows by setting:
$$A(t,x) = \frac{\overline{\rho}(t,x)}{\rho_0} S_{w}(t,x),\quad Q(t,x)=A(t,x) u(t,x),$$
where $S_{w}$ is the \emph{physical wet area}, i.e. the part of the cross-section area in contact with
water.

In order to deal with the transition points (that is, when a change of state occurs), we introduce a
state indicator variable $E$ which is equal to $1$
if the state is pressurized and to $0$ if the state is free surface.\\ 
Notice that  $S_{w}$ is $(A,E)$ dependent via the relations: 
$$
S_{w} = S_{w}(A,E) = \left\{
\begin{array}{lll}
S & \textrm{ if } & E = 1 ,\\
A & \textrm{ if } & E = 0 \ .
\end{array}
\right.
$$
After taking the mean value of the pressure term in the Euler equations over the wetted cross-section,
we get the pressure law as a  mixed ``hydrostatic'' (for the free surface part of the flow)  and ``acoustic''  type (for the pressurized part of the flow) as follows:
\begin{equation}\label{PFSPressureLaw}
 \dsp p(x,A,E) =  c^2(A-S_{w}) + g I_1(x,S_{w}) \cos\theta\, .
\end{equation}
%Let us notice that the terms $-c^2 S_{w}$ and $\dsp \frac{A}{S_{w}}$ come from a unified ``hydrostatic pressure'' given by \eqref{pressure_FS} where $\rho = \rho_{0}$ in \eqref{pressure_FS} for a free surface flow while $\rho$ is given by the acoustic law \eqref{pressure_P} for a pressurized flow.
Thus the continuity is obtained by the ``artificial'' addition in the pressure law of the term $-c^2 S_{w}$.
% which will appear in the source term as $-c^{2} S'(x)$ when the flow is a pressurized one.
This form of the pressure insures the continuity of it at transition points.

The term $I_1$ is the classical hydrostatic pressure: 
$$\dsp I_1(x,S_{w}) =
\int_{-R}^{\mathcal{H}(S_{w})}(\mathcal{H}(S_{w})-z) \sigma \,dz,$$ where $\sigma(x,z)$ is the width of the
cross-section, $R=R(x)$ the radius of the cross-section and $\mathcal{H}(S_{w})$ is the $z$-coordinate of the
free surface over the main  axis $Z(x)$ (see figure \ref{OxOz} and figure \ref{OyOz}).
%%%%%%%%%%%%%%%%%%%%%%%%%%
\begin{figure}[H]
 \begin{center}
 \includegraphics[height=5cm]{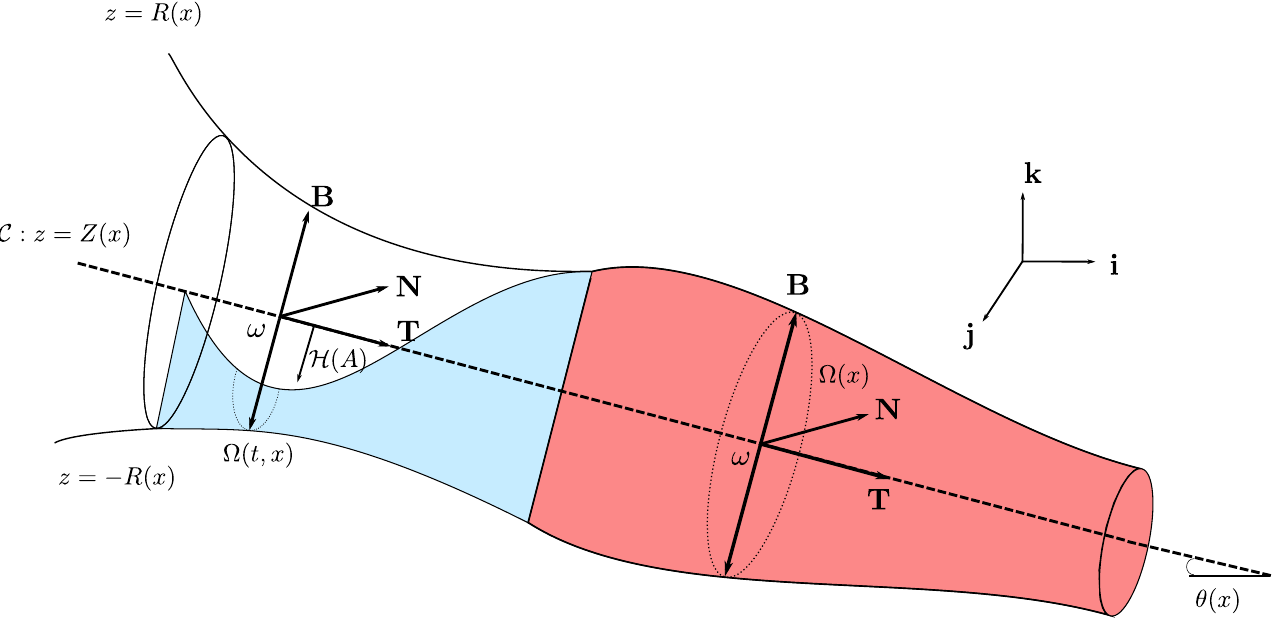}
  \caption{Geometric characteristics of the domain:}
  \small{free surface and pressurized flow. }
  \label{OxOz}
 \end{center}
\end{figure}
\begin{figure}[H]
 \begin{center}
 \includegraphics[height=5cm]{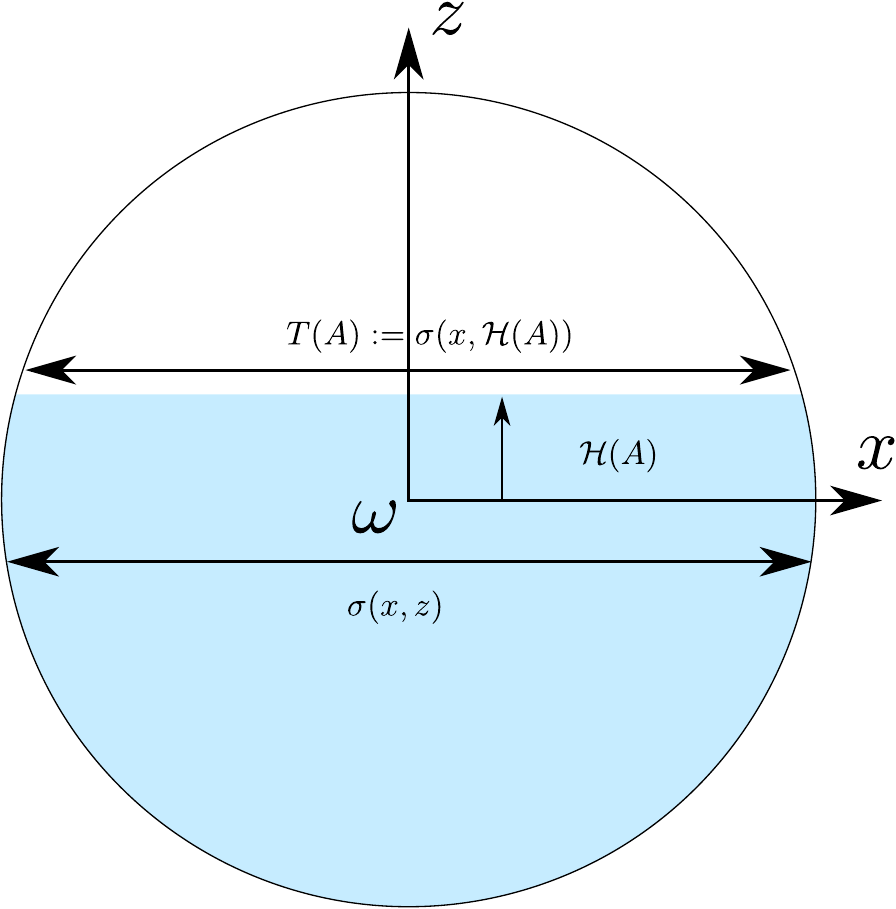}
  \caption{Cross-section $\Omega$.}
  \label{OyOz}
 \end{center}
\end{figure}
%%%%%%%%%%%%%%%%%%%%%%%%%%

\begin{rque}
We can also regard $I_1/A =   \overline{y}$ as the distance separating the free surface to the center of the mass of the wet section (see figure \ref{ybar}).
\end{rque}
\begin{figure}[!ht]
 \begin{center}
 \includegraphics[height=5cm]{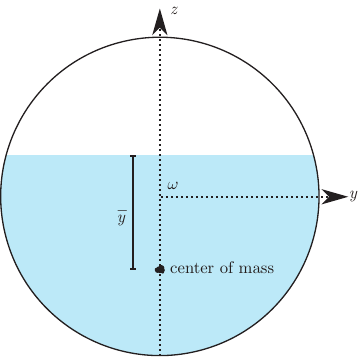}
  \caption{The distance separating the free surface to the center of the mass of the wet section.}
  \label{ybar}
 \end{center}
\end{figure}

The pressure  defined by Equation \eqref{PFSPressureLaw} is  continuous throughout the
transition points and we define the  \textbf{PFS} model by:
\begin{equation}\label{PFS}
\left\{
\begin{array}{lll}
\partial_{t}(A) + \partial_{x}(Q) &=&0, \\
\partial_{t}(Q) + \partial_{x} \left(\dsp \frac{Q^2}{A}+p(x,A,E)\right) &=&\dsp-g\, A\,Z' + Pr(x,A,E) \\
 & &\dsp-G(x,A,E) \\
 & &\dsp - K(x,A,E) \dsp\frac{Q|Q|}{A} \ .
\end{array}
\right.\,
\end{equation}
where $x \in (0,L)$, $L$ being the length of the pipe, $z=Z(x)$ is the altitude of the
main pipe axis. The terms $Pr$, $G$ and $K$ denote respectively
the pressure source term, a curvature term   and the friction:
$$\begin{array}{lll}
Pr(x,A,E) &=& \dsp c^2\left(A- S_{w}\right)\dsp \frac{S'}{S_{w}}+ g \,I_2(x,S_{w})\cos\theta , \\
 G(x,A,E) &=& \dsp g\,A\,  \overline{Z}(x,S_{w}) (\cos\theta)' = \dsp g\,A\,  \left(\mathcal H(S_{w})-I_1(x,S_{w})/S_{w}\right) \dsp
(\cos\theta)',
\\
K(x,A,E) &=& \dsp \frac{1}{K_s^{2} R_h(S_{w})^{4/3}}\ ,
  \end{array}
$$
where we have used the notation $f'$ to denote the  derivative with respect to the space variable $x$ 
of any function $f(x)$. 
The term $I_2$ is the hydrostatic pressure source term defined by:  
$$\dsp I_2(x,S_{w}) =
\int_{-R}^{\mathcal{H}(S_{w})}(\mathcal{H}(S_{w})-z) \partial_x\sigma \,dz \,. $$ 
The term   
$K_s>0$ is the Strickler coefficient depending on the material and $R_h(S_{w})$ is the hydraulic radius. 
\begin{rque}\label{I2}
Let us remark that whenever $S$ is constant on a sub-interval of $(0,L)$, that is $\sigma$ is constant, then $I_{2}(x,S_{w}) = 0$. This fact will be used in the
kinetic interpretation of the  \textbf{PFS} equations.
\end{rque}
%%%%% rem : flux discontinu
\begin{rque}\label{rempression} ~
For the sake of clarity, let us detail the different terms of equation \eqref{PFS} for a free surface flow and a pressurized one.
\begin{itemize}
\item For a free surface flow, we have:
$$
\begin{array}{lll}
 p(x,A,0) &=  &g  I_1(x,A) \cos\theta , \\
Pr(x,A,0) &=& \dsp  g \,I_2(x,A)\cos\theta , \\
 G(x,A,0) &=& \dsp g\,A\,  \overline{Z}(x,A)  \dsp(\cos\theta)', \\
K(x,A,0) &=& \dsp \frac{1}{K_s^{2} R_h(A)^{4/3}}\ . 
  \end{array}
$$
\item For a pressurized flow, we have:
$$
\begin{array}{lll}
 p(x,A,1) &=  & c^2(A-S(x)) + \dsp g I_1(x,S(x)) \cos\theta\, . \\
Pr(x,A,1) &=& \dsp  \dsp c^2\Big(A -S(x)\Big)\dsp \frac{S'(x)}{S(x)}+ g \,I_2(x,S(x))\cos\theta, \\
 G(x,A,1) &=& \dsp g\,A\,  \overline{Z}(x,S(x))  \dsp(\cos\theta)', \\
K(x,A,1) &=& \dsp \frac{1}{K_s^{2} R_h(S(x))^{4/3}}\ . 
  \end{array}
$$
\end{itemize}
Let us notice that in this case, the term $c^2(A-S(x))$ in $p(x,A,1)$ represents the over-pressure.
%while the term  $ \dsp g\frac{A}{S(x)} I_1(x,S(x)) \cos\theta\, $ represents the hydrostatic pressure for the water at the density $\rho$. 
%Moreover, when the flow is pressurized, the term $\displaystyle g\,\frac{A}{S(x)} \,I_2(x,S(x))\cos\theta$ in the source term was not present in the model exposed in 
%\cite{BEG12} and ensures the continuity of the source term at transition points : when the flow is pressurized, the hydrostatic pressure  must take 
%into account the fact that the density of the water is no more $\rho_{0}$ but $\rho$.
\end{rque}
\begin{rque}\label{rempression}
The unknown state vector is denoted $\UU = (A,Q)$ and  the flux vector $\boldsymbol{F}$ by:
\begin{equation*}%\label{Flux}
\boldsymbol{F}(x,\UU,E) = \left(Q,\dsp \frac{Q^2}{A}+p(x,A,E)\right) \: .
\end{equation*}
We simply denote (when no ambiguity is possible):
\begin{equation}\label{QFlux}
F_{2}(A,Q) = \dsp \frac{Q^2}{A}+p(x,A,E) \: ,
\end{equation}
the second component of the preceding flux.\\
As it was pointed out in \cite[Remark 4.2]{BEG09}, the flux is continuous through the change of state of the flow whereas its derivative with respect to $A$
is discontinuous, due to the jump of the sound speed, see figure \ref{pression}. 
\begin{figure}[H]
 \begin{center}
 \includegraphics[height= 5cm]{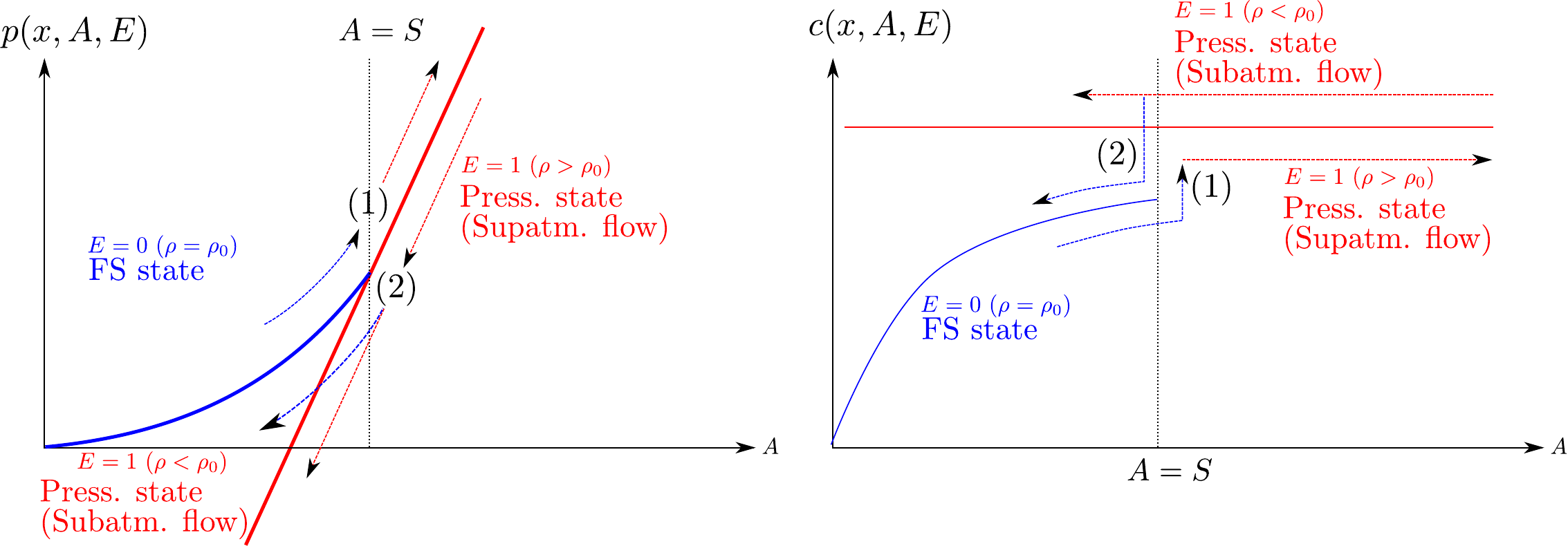}
  \caption{Pressure law and sound speed $c(x,A,E)$ in the case of a rectangular pipe.}
  \small{Trajectory (1) corresponds to a pressurization. Trajectory (2) depends on the state of the flow around}
  \label{pression}
 \end{center}
\end{figure}
\end{rque}
%%%% rem : flux discontinu
\paragraph*{Identification of the source terms. \newline}
In order to write the kinetic interpretation of the  \textbf{PFS} equations, 
we have to factorize by $A$ the right hand side of System \eqref{PFS}. Then, the source terms reads as follows:
\begin{itemize}
\item $\dsp g Z'$ is a conservative term.
\item 
$\dsp c^2\left(\frac{A-S_{w}}{A S_{w}}\right) S' 
= 
\left\{
\begin{array}{lll}
\dsp c^2\left(\frac{A-S}{A}\right) \frac{S'}{S} & \textrm{if}& E = 1\\
0 & \textrm{if} & E = 0
\end{array}
\right.
$ is a non-conservative product. 
\item $\dsp g\frac{ I_2(x,S_{w})\cos\theta}{A}$ is neither conservative nor non-conservative. 
\item $\dsp  g \,  \overline{Z}(x,S_{w})  \dsp(\cos\theta)' = g\,  \left(\mathcal{H}(S_{w})-I_1(x,S_{w})/S_{w}\right) \dsp
\cos\theta'$ is a non-conservative product. 
\item  $\dsp K(x,A,E) \dsp\frac{Q|Q|}{A^2}$ is neither conservative nor non-conservative.
\end{itemize}
Moreover, all the terms said to be non-conservative products are genuinely non-conservative product since they do not write as an exact differential form.   
\newpage
System \eqref{PFS} has the following properties:
\begin{thm}\label{ThmPFSModel} ~
\begin{enumerate}
\item System \eqref{PFS} is strictly hyperbolic on $\left\{A(t,x)>0\right\}\,.$
\item For smooth solutions, the mean velocity $u = Q/A$ satisfies:
\begin{equation}\label{ThmPFSEquationForU}
\begin{array}{c}
\partial_t u + \partial_x \left(\displaystyle\frac{u^2}{2} + c^2 \ln(A/{S_{w}}) + g\mathcal{H}(S_{w})\cos\theta + g Z\right) = -g K(x,A,E) u|u| 
\leqslant 0 \ .
\end{array}
\end{equation}
The quantity $\dsp \Phi(A,Q,\cos\theta,Z,E) =  \displaystyle\frac{u^2}{2} + c^2 \ln(A/{S_{w}}) + g\mathcal{H}(S_{w})\cos\theta + g Z$ is called the
total head.
\item The still water steady state, for $u = 0$, reads:
\begin{equation}\label{ThmPFSSteadyState}
c^2 \ln(A/{S_{w}}) + g\mathcal{H}(S_{w})\cos\theta + g Z = cte \ ,
\end{equation}
for some constant $cte$.
\item System \eqref{PFS} admits a mathematical entropy:
 $$\mathcal{E}(A,Q,E) =\displaystyle \frac{Q^2}{2A} + c^2 A \ln(A/{S_{w}})+ c^2 S + g A \overline{Z}(x,S_{w})\cos\theta + gA Z \ ,$$
which satisfies the entropy relation for smooth solutions
\begin{equation}\label{ThmPFSEntropy}
\partial_t \mathcal{E} +\partial_x \big((\mathcal{E}+p(x,A,E))u\big) =  -gAK(x,A,E) u^2 |u| \leqslant 0 \ .
\end{equation}
\end{enumerate}
\end{thm}
\begin{proof}
The proof of these assumptions relies only on algebraic combinations of the two equations forming System \eqref{PFS} and is left to the reader.
\end{proof}

In what follows, when no confusion is possible, the term $K(x,A,E)$ will be denoted simply $K(x,A)$
for free surface states  and $K(x,S)$ for pressurized states.
\begin{rque}
Equation \eqref{ThmPFSSteadyState} is the still water steady state equation associated to the \textbf{PFS} equation.
Indeed, for a pressurized flow (i.e. $S_{w} = S$),  when $u=0$ and $A = A(x)$, the following equations holds:  
$$c^2\ln(A/S) +g \mathcal{H}(S)\cos\theta+ g Z=cte \ .$$ 
Moreover when $S_{w}=A$, Equation \eqref{ThmPFSSteadyState} provides $g\mathcal{H}(A)\cos\theta  + g Z =cte$: this equation 
represents the horizontal line for a free surface still water steady states.
Moreover, when mixed still water steady states occur, i.e. when one part of the flow is pressurized and the other part of the flow is free surface, 
Equation \eqref{ThmPFSSteadyState} holds again.
\end{rque}
%%%%%%%%%%%%%%%%%%%%%%%%%%%%%
\section{The kinetic interpretation of the \textbf{PFS} model}\label{SectionKineticInterpretationPFSModel}
%%%%%%%%%%%%%%%%%%%%%%%%%%%%%
Recently in \cite{BEG09}, we  have investigated a class of approximated Godunov scheme for the
present \textbf{PFS} model in which we show how to  obtain by a suitable definition of the convection matrix an exactly well-balanced scheme for the
still water steady state. We also point  out that the upwinding of the source terms into the numerical fluxes introduces a stationary
wave with a vanishing denominator.  We also  discuss on the possibility to introduce a cut-off function to avoid the division by zero. But,
the truncation of the wet area $A$ induces  a loss of mass which implies a loss of the conservativity property. Moreover, the numerical
scheme loses accuracy. Therefore, stationary hydraulic  jump and flooding area are not accurately computed with this kind of numerical
scheme. As pointed out in \cite{PS01}, the numerical kinetic scheme are proved to satisfy the following stability properties:
the water height conservativity, the in cell entropy  inequality and the conservation of the still water steady state. 
Unfortunately, it holds only for rectangular geometry. 
% One of the main contribution will be to show how to get at least 
% a well-balanced scheme for any given geometry,  the in cell entropy  inequality being still an open problem for kinetic scheme in this framework. 

The goal of this paper is to construct a  Finite Volume-Kinetic scheme that preserves the wetted area positive, that can treat naturally the flooding and also
that is easily implemented.

The big challenge in this construction is the fact that the continuous flux 
has a discontinuous gradient at the interface between free surface and pressurized flow, see Remark \ref{rempression}. 
In \cite{BEG09}, the finite volume scheme that we have constructed uses the ``ghost waves approach''  to overcome this difficulty.
We will see that we can still use this approach for the finite volume kinetic scheme but we will prefer to construct a fully kinetic scheme
that is a scheme at the kinetic level  which treats the changes of type of the flow.

First of all, let us recall the kinetic interpretation of the \textbf{PFS} model based on Perthame's kinetic formulation of conservation laws \cite{P02}.
%%%%%%%%%%%%%%%%%%%%%%%%%%%%%%%%%%%
\subsection{The mathematical kinetic interpretation}
%%%%%%%%%%%%%%%%%%%%%%%%%%%%%%%%%%%
Let $\chi:\R\to\R$ be a given real function satisfying the following  properties:
\begin{equation}\label{propchi}
\chi(\omega)=\chi(-\omega) \geqslant 0\;,\;
\int_{\R} \chi(\omega) d\omega =1,
\int_{\R} \omega^2 \chi(\omega) d\omega=1 .
\end{equation} 
It permits to  define  the density of particles, by a so-called \emph{Gibbs equilibrium}, 
\begin{equation}\label{Gibbs}
\mathcal{M}(t,x,\xi) =
\frac{A(t,x)}{b(t,x)} \chi\left(\frac{\xi-u(t,x)}{b(t,x)}\right) ,
\end{equation}
where $b(t,x) = b(x,A(t,x),E(t,x))$ with 
\begin{equation}\label{soncinetique}
b(x,A,E)= \left\{
\begin{array}{lll}
\dsp\sqrt{g\,\frac{I_1(x,A)}{A}\cos\theta} & \textrm{ if } & E = 0,\\
~\\
\dsp\sqrt{g\,\frac{I_1(x,S)}{A}\cos\theta+c^2} & \textrm{ if } & E = 1.\\
\end{array}
\right.
\end{equation}

\begin{rque}\label{remB}
Let us remark that when $E=0$, we have $b(x,A,0)=\displaystyle \sqrt{G \overline{y} \cos \theta}$, where 
$\overline{y}$ is the distance separating the free surface to the center of the mass of the wet section (see figure \ref{ybar}).
\end{rque}

The Gibbs equilibrium $\mathcal{M}$ is related 
to the \textbf{PFS} model by the classical 
 \emph{macro-micro}scopic kinetic relations:
\begin{eqnarray}
A  &=& \dsp\int_{\R} \mathcal{M}(t,x,\xi)\,d\xi\,, \label{macroA} \\
Q  &= &\dsp\int_{\R} \xi\mathcal{M}(t,x,\xi)\,d\xi\,, \label{macroQ} \\
\dsp\frac{Q ^2}{A }+A\,b(x,A,E)^2  &= &\dsp\int_{\R} \xi^2 \mathcal{M}(t,x,\xi)\,d\xi \label{macroFlux} \ .
\end{eqnarray}
From the relations \eqref{macroA}--\eqref{macroFlux}, 
the nonlinear \textbf{PFS} model can be viewed as a single linear equation involving the nonlinear quantity 
$\mathcal{M}$:
\begin{thm}[Kinetic interpretation of the \textbf{PFS} model]\label{ThmKineticFormulationPFS}
$(A,Q)$ is a strong solution of
System \eqref{PFS} if and only if ${\mathcal{M}}$ satisfies the kinetic transport equation:
\begin{equation}\label{KineticFormulationPFS}
\partial_t \mathcal{M}+\xi \cdot \partial_x\mathcal{M} - g\phi 
\,\partial_\xi \mathcal{M} = \mathcal{K}(t,x,\xi) ,
\end{equation}
for a collision term $\mathcal{K}(t,x,\xi)$ which satisfies for $(t,x)$ a.e. 
$$\dsp  \int_{\R} \vecdeux{1}{\xi} \mathcal{K}(t,x,\xi)\,d\xi = 0 \,,\,$$
where the source term $\phi$ is defined as:
\begin{equation}\label{PFSSourceTermPhi}
\begin{array}{ll}
\mbox {if E = 0}, &\displaystyle \phi(x) = \partial_{x}Z + K(x,A) u\abs{u} - \frac{I_{2}(x,A) \cos \theta}{A} + \overline{Z}(x,A) \partial_{x} \cos \theta ,\\[0.3cm]
\mbox{if E = 1}, &\displaystyle  \phi(x) =\partial_{x}Z + K(x,S) u\abs{u}  - \frac{I_{2}(x,S) \cos \theta}{A}  - 
\frac{c^{2}}{g} \frac{ A - S}{A} \frac{\partial_{x} S}{S}+ \overline{Z}(x,S) \partial_{x} \cos \theta \ .
\end{array}
\end{equation}
%\phi(x,\WW) = \BB(x,\WW)\cdot\partial_x \WW
%\end{equation}
%with \begin{equation}\label{WW}
%\dsp\WW = \left(\boldsymbol{Z},\;S,\;\cos\theta\right)      
%     \end{equation}
%and
%$
%\BB =
%\left\{
%\begin{array}{ll}
%\dsp \left(1,\;-\frac{c^2}{g}\left(\frac{A-S}{A\,S}\right)-\frac{I_2(x,S)\cos\theta}{ S' A},\;\overline{Z}(x,S)\right) & \textrm{ if  } E = 1,\\
%\dsp \left(1,\;-\frac{I_2(x,A)\cos\theta}{ S' A},\;\overline{Z}(x,A)\right) & \textrm{ if  } E = 0\\
%\end{array}
%\right.
%$
\end{thm}
\begin{proof}
The proof relies on very obvious computations since $\mathcal{M}$ verifies the macro-microscopic kinetic relations \eqref{macroA}, \eqref{macroQ}, \eqref{macroFlux},
and from the definition of the source terms $\phi$.
\end{proof}
%%%%%%%%%%%%%%%%%%
% remark
%%%%%%%%%%%%%%%%%%
\begin{rque}
The kinetic interpretation presented in Theorem \ref{ThmKineticFormulationPFS} is a (non physical) microscopic
description of the \textbf{PFS} model. 
\end{rque}
%%%%%%%%%%%%%%%%%%%%%%%%%%%%%
\section{Construction of the kinetic scheme for the \textbf{PFS} model}\label{SectionStudyOfTheNumericalApproximationOfThePFSModel}
%%%%%%%%%%%%%%%%%%%%%%%%%%%%%
In this section, following the works of \cite{PS01,BEG11_2}, we will construct a finite volume kinetic scheme
that preserves  the wetted area positive and that will compute ``naturally'' flooding zones. 
The main feature of this scheme is the treatment of transition points between free surface and pressurized flows. 
In a first step, we will  use the ``ghost waves approach'' that we have constructed in \cite{BEG09} to treat this difficulty: to this end we will go back to the macroscopic level to compute the unknown states $(A,Q)$ at the interface between free surface and pressurized flow.

In a second step, we will construct a fully kinetic scheme to treat the interface between the free surface and the pressurized flow: the Gibbs equilibrium
 on the right hand side and the left hand side of the interface between free surface and pressurized flow will be computed by kinetic formulas.
To upwind all the source terms at the microscopic level, we will use the ideas presented in the recent work of the authors \cite{BEG11_2}.

The particular treatment of the boundary conditions only at the microscopic level  will be rapidly exposed. This is the key feature of the ``\FKAl''.
%%%%%%%%%%%%%%%%%%%%%%%%%%%%%
\subsection{The kinetic scheme without transition points}\label{SectionTheKineticSchemeWithNoTransitionPoints}
%%%%%%%%%%%%%%%%%%%%%%%%%%%%%
In this section, we will treat the parts of the flow that are either free surface or pressurized.
Under this assumption and based on the kinetic interpretation (see Theorem \ref{ThmKineticFormulationPFS}), we construct easily a Finite
Volume scheme where the conservative quantities are cell-centered and  source terms are included into the numerical fluxes by a standard
kinetic scheme with reflections \cite{PS01}. 

To this end, let $N\in\N^{*}$, and let us consider the following mesh on $[0,L]$. Cells are denoted for every
$i\in [0,N+1]$, by $m_i =(x_{i-1/2},x_{i+1/2})$, with  $x_i=\ds\frac{x_{i-1/2}+x_{i+1/2}}{2}$ and 
$h_{i}=x_{i+1/2}-x_{i-1/2} $ the space step. The ``fictitious'' cells $m_0$ and $m_{N+1}$ denote the boundary cells and 
the mesh  interfaces located at $x_{1/2}$ and $x_{N+1/2}$ are respectively the upstream and the downstream ends of the pipe (see figure \ref{pipediscret}).
\begin{figure}[H]
 \begin{center}
 \includegraphics[height = 2cm]{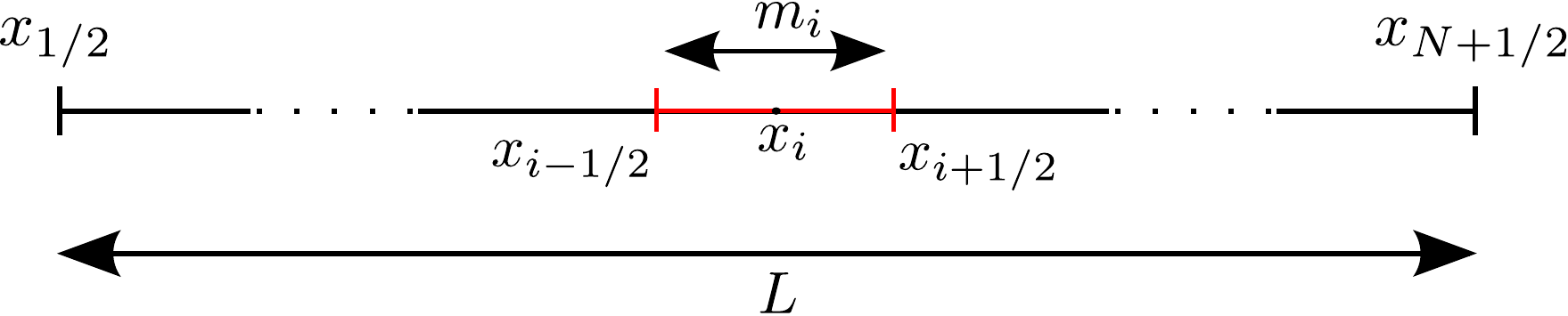}
  \caption{The space discretisation.}
  \label{pipediscret}
 \end{center}
\end{figure}
We also consider a time discretization $t^n$ defined by $t^{n+1}=t^n+\Delta t^n$ with $\Delta t^n$ the time step. 

We denote $\UU_i^n=(A_i^n,Q_i^n)$, $\dsp u_i^n = \frac{Q_i^n}{A_i^n}$, $\mathcal{M}_i^n$ 
the cell-centered approximation of $\UU = (A,Q)$, $u$ and $\mathcal{M}$ on the cell $m_i$ at time $t^n$.
We denote by $\UU_{0}^n=(A_0^n, Q_0^n)$ the upstream  and
$\UU_{N+1}^n=(A_{N+1}^n, Q_{N+1}^n)$ the downstream state vectors.

For $i\in [0,N+1] \,,\, E_i$ is the state indicator of the cell: $E_i = 0$ if in the cell $m_{i}$, the flow is a free surface flow, $E_{i} = 1$ 
if in the cell $m_{i}$, the flow is a pressurized flow.

On a time interval $[t^{n},t^{n+1}]$ and on the cell $m_{i}$, the kinetic equation \eqref{KineticFormulationPFS} writes:
\begin{equation}\label{KineticFormulationPFS2}
\left\{
\begin{array}{ll}
\partial_t \mathcal{M}+\xi \cdot \partial_x\mathcal{M} - g\phi 
\,\partial_\xi \mathcal{M} = \mathcal{K}(t,x,\xi) &\mbox{ for } x \in m_i \,,\, t \in (t^{n},t^{n+1}) \,,\,\xi \in \R ,\\[0.25cm]
\displaystyle \mathcal{M}(t^{n},x,\xi) =  \mathcal{M}_i^n(\xi)& \mbox{ for } x \in m_i \,,\, \xi \in \R \ .
\end{array}
\right.
\end{equation}

\begin{rque}\label{remkinetic}
At this stage, for $x \in m_{i}$, it is convenient to introduce the term: 
$$\ds \boldsymbol{Z}_{i}(t,x) = \left(Z(x)+\int_{x_{i-1/2}}^{x}K\big(s,A(t,s),E\big) u(t,s)\abs{u(t,s)}\,ds\right)  ,$$
 which is called \textbf{\textit{ the dynamic topography}} since it is time and space variable dependent contrary to the \textbf{static} topography
 $Z$ which is only $x$-dependent. 
This notion is closely related to the ``apparent topography'' 
introduced by Bouchut {\it et al} \cite{B04,BW04,MBTVB07}.
\end{rque}
For $x \in m_{i}$, denoting  
\begin{equation}\label{WW}
\dsp\WW = \left(\boldsymbol{Z}_{i},\;S,\;\cos\theta\right), 
\end{equation}
the source term $\phi$ defined by \eqref{PFSSourceTermPhi} becomes:
\begin{equation}\label{PFSSourceTermPhibis}
\begin{array}{ll}
\mbox {if } E = 0, &\displaystyle \phi(x,\WW) = \partial_{x}\boldsymbol{Z}_{i}- \frac{I_{2}(x,A) \cos \theta}{A} + \overline{Z}(x,A) \partial_{x} \cos \theta \ ,\\[0.3cm]
\mbox{if } E = 1, &\displaystyle  \phi(x,\WW) =\partial_{x}\boldsymbol{Z}_{i} - \frac{I_{2}(x,S) \cos \theta}{A}  - 
\frac{c^{2}}{g} \frac{ A - S}{A} \frac{\partial_{x} S}{S}+ \overline{Z}(x,S) \partial_{x} \cos \theta \ .
\end{array}
\end{equation}
The piecewise constant representation of $\dsp\WW$ defined by \eqref{WW} is given by,   
$\WW(t,x) = \dsp \WW_i(t) \mathds{1}_{m_i}(x)$ 
where $\WW_i(t)$ is defined as $\WW_i(t) = \dsp \frac{1}{\Delta x}\int_{m_i}\WW(t,x)\,dx$ for instance.
\begin{rque}Let us notice that as $\WW(x,t) = \WW_i(t)$ is constant on the cell $m_{i}$, thanks to remark \ref{I2},  $\phi(x,\WW) = 0$ on the cell $m_{i}$.
 Indeed, since $\boldsymbol{Z}_{i}$ is constant, $\partial_{x}\boldsymbol{Z}_{i} = 0$, since $S$ is constant, 
 $I_{2}(x,S_{w})~=~0$ and $\partial_{x}S = 0$ and since $\theta$ is constant, $\partial_{x} \cos \theta$ = 0.
 
 We use this simple fact to construct the kinetic scheme as follows.
\end{rque}

Neglecting the collision kernel as in \cite{PS01,BEG11_2} the kinetic transport equation
\eqref{KineticFormulationPFS2} simply reads:
%  Equation in the cell m_i
\begin{equation}\label{eqcin}
\left\{
\begin{array}{ll}
\displaystyle \frac{\partial}{\partial t} f +\xi \cdot \frac{\partial}{\partial x} f =0 & \mbox{ for } x \in m_i \,,\, t \in (t^{n},t^{n+1}) \,,\,\xi \in \R ,\\[0.25cm]
\displaystyle f(t^{n},x,\xi) =  \mathcal{M}_i^n(\xi) &\mbox{ for } x \in m_i \,,\, \xi \in \R \ .
\end{array}
\right.
\end{equation}

This equation is a linear transport equation whose explicit discretisation may be done directly by
the following way. A finite volume discretisation of Equation \eqref{eqcin} leads to:
%  Equation in the cell m_i: discretisation
\begin{equation} \label{cindiscret}
\forall \xi \in \R \,,\, \forall x \in m_{i} \,,\,  f(t^{n+1},x,\xi) = f_i^{n+1}(\xi) = \mathcal{M}_i^n(\xi)+\frac{\Delta t}{h_i}\,
\xi \, \left(\mathcal{M}_{i+\frac{1}{2}}^-(\xi)-\mathcal{M}_{i-\frac{1}{2}}^+(\xi)\right),
\end{equation}
where the fluxes $\mathcal{M}_{i+\frac{1}{2}}^\pm$ have to take into account the discontinuity of
the source term $\phi(x,\WW)$ at the cell interface $x_{i+1/2}$. This is the principle of interfacial  source upwind. 
Indeed, noticing that the fluxes can also be written as:
% Remark on the difference on the fluxes
\begin{equation*}
\mathcal{M}_{i+\frac{1}{2}}^-(\xi) = \mathcal{M}_{i+\frac{1}{2}} +
\left(\mathcal{M}_{i+\frac{1}{2}}^- - \mathcal{M}_{i+\frac{1}{2}}\right),
\end{equation*}
the quantity $\delta \mathcal{M}_{i+\frac{1}{2}}^- = \mathcal{M}_{i+\frac{1}{2}}^- -
\mathcal{M}_{i+\frac{1}{2}}$
holds for the discrete contribution of the source term $\phi(x,\WW)$ in the system for negative
velocities $\xi \leq 0$  due to the upwinding of the source term.
Thus $\delta \mathcal{M}_{i+\frac{1}{2}}^-$ has to vanish for positive velocity $\xi > 0$,
as proposed by the choice of the interface fluxes below.
% The choice for the fluxes.
Let us now detail our choice for the fluxes $\mathcal{M}_{i+\frac{1}{2}}^\pm$ at the interface. It
can be justified by using a generalized characteristic method for Equation \eqref{KineticFormulationPFS}
(without the collision kernel) but we give instead a presentation based on some physical energetic
balance. The details of the construction of these fluxes by the general characteristics  method (see \cite[Definition 2.1]{Da89}) is
done in \cite[Chapter 2]{TheseErsoy}.

In order to take into account the neighboring cells by means of a natural interpretation of the
microscopic features of the system, we formulate a peculiar  discretisation for the fluxes in
\eqref{cindiscret}, computed by the following upwinded formulas:
\begin{equation}\label{MicroscopicInterfaceFluxes}
\begin{array}{lll}
\mathcal M_{i+1/2}^{-}(\xi) &=&  
\overbrace{\dsp \mathds{1}_{\{\xi>0\}}\mathcal M_i^n(\xi)}^{{\textrm{positive  transmission}}}
+\overbrace{\mathds{1}_{\{\xi<0,|\xi|^2-2g\phi^n_{i+1/2}<0\}}\mathcal M_i^n(-\xi)}^{\textrm{reflection}}\\ 
&+& \underbrace{\dsp \mathds{1}_{\{\xi<0,|\xi|^2-2g\phi^n_{i+1/2}>0\}}
\mathcal M_{i+1}^n\left(\dsp-\sqrt{|\xi|^2-2g\phi^n_{i+1/2}}\right)}_{\textrm{negative   transmission}} \ ,\\
 & & \\
\mathcal M_{i+1/2}^{+}(\xi) &=& \overbrace{\dsp \mathds{1}_{\{\xi<0\}}\mathcal M_{i+1}^n(\xi)}^{\textrm{negative transmission}}
+\overbrace{\mathds{1}_{\{\xi>0,|\xi|^2+2g\phi^n_{i+1/2}<0\}}\mathcal M_{i+1}^n(-\xi)}^{\textrm{reflection}}\\ 
&+& \underbrace{ \dsp \mathds{1}_{\{\xi>0,|\xi|^2+2g\phi^n_{i+1/2}>0\}}
\mathcal M_{i}^n\left(\dsp\sqrt{|\xi|^2+2g\phi^n_{i+1/2}}\right)}_{\textrm{positive  transmission}}\ .
\end{array}\,
\end{equation}
The term $\phi^n_{i\pm 1/2}$ in \eqref{MicroscopicInterfaceFluxes}  is the upwinded source term
\eqref{PFSSourceTermPhi}. It also plays the role of the potential barrier:  
the term $|\xi|^2\pm 2g \phi^n_{i+1/2}$ is the jump condition for a particle with a kinetic speed $\xi$  
which is necessary  to
\begin{itemize}
\item be reflected: this means that the particle has not enough kinetic energy $|\xi|^2/2$ 
to overpass the potential barrier (reflection term  in \eqref{MicroscopicInterfaceFluxes}),
\item overpass the potential barrier with a positive speed 
(positive  transmission term in \eqref{MicroscopicInterfaceFluxes}),
\item overpass the potential barrier with a negative speed 
(negative  transmission term in \eqref{MicroscopicInterfaceFluxes}).
\end{itemize}
Having in mind the so-called non conservative products defined by Dal Maso, Murat and Lefloch \cite{DLM95}, and recalling that the dynamic topography is defined 
on the cell $m_{i}$ as $\ds \boldsymbol{Z}_{i}(t,x) = \left(Z(x)+\int_{x_{i-1/2}}^{x}K\big(s,A,E\big) u\abs{u}\,ds\right)$, we choose a midpoint approximation 
of the source term $\phi$   defined by Equation \eqref{PFSSourceTermPhi} at the interface $x_{i}$.\\
The same ``trick'' is used to take into account the pressure source term involving $I_{2}(x,S_{w})$ that is :
$$I_{2}(x,S_{w}) = \displaystyle \frac{ \displaystyle \partial} {\partial x}\left(\int_{x_{i-1/2}}^{x} I_{2}(s,S_{w}) ds\right) .$$

Thus, the potential barrier $\phi^n_{i+ 1/2}$ has the following expression:
\begin{eqnarray*}
\mbox{if } E = 0\,,\, \phi_{i+\frac{1}{2}}^n&=&Z_{i+1}-Z_i \\
&&+\frac{h_i}{2}\,K(x_i,A_i^n,E_i^n)\,u_i^n\,\abs{u_i^n}+\frac{h_{i+1}}{2}\,K(x_{i+1},A_{i+1}^n,E_{i+1}^n)\,u_{i+1}^n\,\abs{u_{i+1}^n}\\
&&- \left(\frac{h_{i}} {2} \frac{I_{2}(x_{i},A_{i}) \cos \theta_{i}}{A_{i}}  + \frac{h_{i+1}} {2} \frac{I_{2}(x_{i+1},A_{i+1})\cos \theta_{i+1}}{A_{i+1}} \right)\\
&&+\,\frac{\overline{Z}_i^n+\overline{Z}_{i+1}^n}{2}\,(\cos\theta_{i+1}-\cos\theta_i) \ ,\\
\end{eqnarray*}
\begin{eqnarray*}
\mbox{if } E = 1\,,\, \phi_{i+\frac{1}{2}}^n&=&Z_{i+1}-Z_i \\
&&+\frac{h_i}{2}\,K(x_i,A_i^n,E_i^n)\,u_i^n\,\abs{u_i^n}+\frac{h_{i+1}}{2}\,K(x_{i+1},A_{i+1}^n,E_{i+1}^n)\,u_{i+1}^n\,\abs{u_{i+1}^n}\\
&&- \left(\frac{h_{i}} {2} \frac{I_{2}(x_{i},S(x_{i})) \cos \theta_{i}}{A_{i}}  + \frac{h_{i+1}} {2} \frac{I_{2}(x_{i+1},S(x_{i+1})) \cos \theta_{i+1}}{A_{i+1}} \right) \\
&&-\frac{c^{2}}{g} \left(\frac{1} {2} \frac{(A_{i}- S(x_{i}))}{A_{i}} + \frac{1} {2} \frac{(A_{i+1}-S(x_{i+1})) }{A_{i+1}}\right)
\Big(\ln(S(x_{i+1}) - \ln(S(x_{i})\Big) \\
&&+\,\frac{\overline{Z}_i^n+\overline{Z}_{i+1}^n}{2}\,(\cos\theta_{i+1}-\cos\theta_i) \ .
\end{eqnarray*}

As the first term of  $\phi_{i+\frac{1}{2}}^n$ is $Z_{i+1}-Z_i$, we recover the classical interfacial upwinding for the  conservative term $Z$ used
by Perthame and Simeoni \cite{PS01}.

Since we neglect the collision term, it is clear that $f^{n+1}$ computed by the discretised
kinetic equation \eqref{cindiscret} is no more a Gibbs equilibrium. Therefore, to recover the macroscopic variables $A$ and $Q$,
according to  the identities \eqref{macroA}-\eqref{macroQ}, we set:
\begin{equation*}%\label{macro}
 \UU_i^{n+1}=
\left(\begin{array}{l}
A_i^{n+1} \\
Q_i^{n+1}
\end{array}
\right)
\overset{def}{=} \int_\R \left(\begin{array}{l}
1 \\
\xi
\end{array}
\right) f_i^{n+1} \,d\xi \ .
\end{equation*}
In fact at each time step, we projected $f^{n}(\xi)$ on $\mathcal{M}_i^n(\xi)$, which is a way to perform all collisions at once and to recover a Gibbs equilibrium
without computing it.

Now, we can integrate the discretised kinetic equation \eqref{cindiscret} against 1 and $\xi$ to obtain the macroscopic kinetic scheme:
\begin{equation} \label{macrodiscret}
\UU_i^{n+1} = \UU_i^n +\frac{\Delta t^{n}}{h_i}\left(\F_{i+\frac{1}{2}}^- - \F_{i-\frac{1}{2}}^+ \right) \: .
\end{equation}
The numerical fluxes are thus defined by the kinetic fluxes as follows:
\begin{equation}\label{fluxdiscret}
\F_{i+\frac{1}{2}}^\pm
\overset{def}{=} \int_\R \xi \left(\begin{array}{l}
1 \\
\xi
\end{array}
\right) \mathcal{M}^\pm_{i+\frac{1}{2}}(\xi)\,d\xi \ .
\end{equation}
At this stage of the construction of the scheme, the choice of the function $\chi$ is crucial. The two main features of this choice are the following:
\begin{itemize}
\item the numerical kinetic scheme should preserve the still water steady state $u=0$ which gives equation \eqref{ThmPFSSteadyState}. Writing the
kinetic formulation of the still water steady state, we obtain an ordinary differential equation for the function $\chi$  which looks like equation (16) in the work
of Perthame and Simeoni \cite{PS01} with additional terms. For a non uniform pipe, we are not able to find an explicit solution of this ordinary differential equation.
\item the numerical scheme should also exhibit an in cell entropy inequality. This fact may be obtained for a function $\chi$ that minimizes a kinetic energy.
Unfortunately, as pointed out in the work of Bourdarias {\it et al.}  \cite[Equation (24)]{BGG08}, for a pressurized flow in a uniform pipe, the function $\chi$ that minimizes the kinetic energy has a non compact support. In the same work, the authors had noticed that the kinetic energy for a 
free surface flow is not always convex \cite[Remark 7]{BGG08}.

Nevertheless, as done by the authors in \cite[Equations (55)-(56)]{BEG09}, a well balanced corrections may be implemented at the macroscopic level to recover
the still water steady state. This is not the goal of the present work.
%and the reader may apply this correction easily.
\end{itemize}

In this work, we have chosen a compact support function $\chi$ in order to fulfill a CFL type condition to obtain a scheme that preserves the positivity of the wet area
and therefore to ensure the $L^{\infty}$ stability of the scheme. Moreover, every annoying computations of integrals involving the function $\chi$ will lead to exact and
easy computations.

For the direct computations of all integral terms, we have chosen in the industrial code \verb+FlowMix+ 
(used by the engineers of  Electricit\'e de France, Centre d'Ing\'enierie Hydraulique, Chamb\'ery):
\begin{equation}\label{chifunction}
\chi=\frac{1}{2\sqrt{3}}\mathds{1}_{[-\sqrt{3},\sqrt{3}]} \ .
\end{equation}

Let us emphasize that although the function $\chi$ chosen for the computations does not lead to the preservation of steady states and to the in cell entropy inequality,
the numerical results presents in section \ref{steady} show a very good behavior of the numerical kinetic scheme to recover free surface steady state as well as 
mixed steady states without notable damages.
%
%Computing the macroscopic state $\UU$ by Equations \eqref{macrodiscret}-\eqref{fluxdiscret} is not easy if the function $\chi$ 
%verifying the properties \eqref{propchi} is not compactly supported. 
%
%Moreover, as we shall see in the next proposition, a CFL condition is needed to obtain a scheme that preserves the positivity of the wet area
%and therefore to ensure the stability of the scheme.
\begin{prop}\label{SchemaCinProprieteClassique}
Let $\chi$ be a compactly supported function verifying  \eqref{propchi} and denote $[-M,M]$ its support. 
The kinetic scheme \eqref{macrodiscret}-\eqref{fluxdiscret} has the following properties:
\begin{enumerate}
\item The kinetic scheme is a wet area conservative  scheme,
\item Assume the following CFL condition  
\begin{equation}\label{CFLcondition}
\Delta  t^n \max_i\left(\abs{u_i^n}+ M\,b_i^n\right)\leqslant \max_i h_i
\end{equation} holds. 
Then the kinetic scheme keeps the wet area $A$ positive i.e:
$$\mbox{ if, for every } i\in [0,N+1] \,,\, A^0_i \geqslant 0 \mbox{ then, for every } i\in [0,N+1] \,,\, A_i^n \geqslant 0.$$
\item The kinetic scheme treats ``naturally'' flooding zones.
\end{enumerate}
\end{prop}
\begin{proof}
We will adapt the proof of \cite{PS01} to show the three properties that verify  the kinetic scheme.

1. Let us denote the first component of the discrete fluxes \eqref{fluxdiscret}
$\left(F_A\right)^\pm_{i+\frac{1}{2}}$:
\begin{equation*}
\left(F_A\right)^\pm_{i+\frac{1}{2}}
\overset{def}{=} \int_\R \xi \mathcal{M}^\pm_{i+\frac{1}{2}}(\xi)\,d\xi
\end{equation*}
An easy computation, using the change of variable $\mu = \abs{\xi}^2 - 2g\phi_{i+\frac{1}{2}}^{n}\,$, allows us
to show that:
\begin{equation*}
\left(F_A\right)^+_{i+\frac{1}{2}} = \left(F_A\right)^-_{i+\frac{1}{2}} \: .
\end{equation*}
2. Suppose that for every $i\in [0,N+1]$, $A_i^{n}>0$. Let us denote $\xi_{\pm} = \max(0,\pm\xi)$  and  $\sigma=\dsp\frac{\Delta t^n}{\max_{i} h_i}$.
From Equation \eqref{cindiscret},  we get the following equation:
$$\begin{array}{lll}
 f_i^{n+1}(\xi)  & =&  (1-\sigma |\xi|)\mathcal M_i^n(\xi)\\
                 &   & + \sigma \xi_+ \bigg[\mathds{1}_{\{|\xi|^2+2g\Delta \phi_{i+1/2}<0\}}\mathcal
M_i^n(-\xi)\\ 
                 &   & +  \mathds{1}_{\{|\xi|^2+2g\Delta \phi_{i-1/2}>0\}}\mathcal
M_{i-1}^n\left(\displaystyle\sqrt{|\xi|^2+2g\Delta
\phi_{i+1/2}}\right)\bigg]  \\
		 &   &  +  \sigma \xi_- \bigg(\mathds{1}_{\{|\xi|^2-2g\Delta \phi_{i+1/2}<0\}}\mathcal
M_i^n(-\xi)\\  
                 &   &  + \mathds{1}_{\{|\xi|^2-2g\Delta \phi_{i-1/2}>0\}}\mathcal
M_{i+1}^n\left(\displaystyle-\sqrt{|\xi|^2-2g\Delta
\phi_{i+1/2}} 
\right)\bigg)  \ . 
\end{array}$$
Since the function $\chi$ is compactly supported $\textrm{ if } \abs{\xi - u_i^n}\geqslant M b_i^n \textrm{ then } \mathcal{M}_i^n(\xi) = 0 .$
Thus $$f_i^{n+1}(\xi)\geqslant 0 \textrm{ if } \abs{\xi - u_i^n}\geqslant M b_i^n \,,$$ 
as a sum of non negative terms.\\
On the other hand, for $\dsp \abs{\xi - u_i^n}\leqslant M b_i^n$, using the CFL condition   $0 < \sigma \abs{\xi}\leqslant 1$, 
for all $i$, $f_i^{n+1}\geqslant 0$ since it is a convex combination of non negative terms. 

Finally we have $\forall i \in [0,N+1] \,,\, f^{n}_{i}  \geq 0$. Since $ \forall i \in [0,N+1] \,,\, A_i^{n+1} = \displaystyle \int_{\R} f_i^{n+1}(\xi)\,d\xi,$ we finally get  
$ \forall i \in [0,N+1] \,,\ A_i^{n+1} \geq 0.$\\
$3.$ Suppose $A_i^{n} = 0$. Of course, in the mesh $i$  the flow is a free surface flow. 
So using the definition of $\mathcal{M}$, and the fact that the function $\chi$ is compactly
supported, the only term that may cause problem is $\dsp \frac{A}{b(t,x)}.$

But since when $E=0$, we have (thanks to Remark \ref{remB}):
$$\dsp \frac{A}{b(t,x)} = \sqrt{\frac{A}{g \overline{y}\cos\theta}} \  .$$
Thus, for all the usual pipe geometries (rectangular, trapezoidal, circular, ...) an easy computation gives:  
$$\dsp  \lim_{\underset {A \geq 0} {A \rightarrow 0}} \frac{A}{b(t,x)}  =  0 \; .$$ 
Otherwise, we add as an assumption that for the considered pipe geometry, we suppose that:
$$\dsp  \lim_{\underset {A \geq 0} {A \rightarrow 0}} \sqrt{\frac{A}{ \overline{y}}} = 0 \ .$$
Therefore, we get $\mathcal{M}_i^n(\xi) = 0$. This is the reason why we say
that the kinetic scheme treats ``naturally'' the flooding zones.

\end{proof}
%%%%%%%%%%%%%%%%%%%%%%%%%%%%%
\subsection{The kinetic scheme with transition points}\label{SectionTheKineticSchemeWithTransitionPoints}
%%%%%%%%%%%%%%%%%%%%%%%%%%%%%
The transition points are characterized by the points of the pipe where the flow is not of the same type on the left hand side and the right hand side of these points.
More precisely, they are characterized by their localization in the pipe and the speed of propagation $w$ of the change of the sate. 
We will assume that it exists at most a finite set of transition points and we will consider that these transition points are located at the interface 
between two cells of the mesh, i.e. a point $x_{i+1/2}$ of the mesh is a transition point if $E_i\neq E_{i+1}$.
The speed of propagation of the interface defines a discontinuity line $x=w \ t$ and let us introduce $\UU^-=(A^-,Q^-)$ and  $\UU^+=(A^+,Q^+)$ the (unknown) states 
respectively on the left  hand side and on the right hand side of this line (see figure \ref{ExampleStateTransition}). 
The speed $w$ is related to the unknowns $\UU^{\pm}$  by the Rankine-Hugoniot condition on the mass
equation $$\partial_t A + \partial_x Q=0,$$ and thus $w$ is given by:
\begin{equation*}%\label{w}
w = \frac{Q^+-Q^-}{A^+-A^-}\,.
\end{equation*}
%%%%%%%%%%%%%%%%%%%%%%%%%%
\begin{figure}[H]
 \begin{center}
 \includegraphics[height = 4.cm]{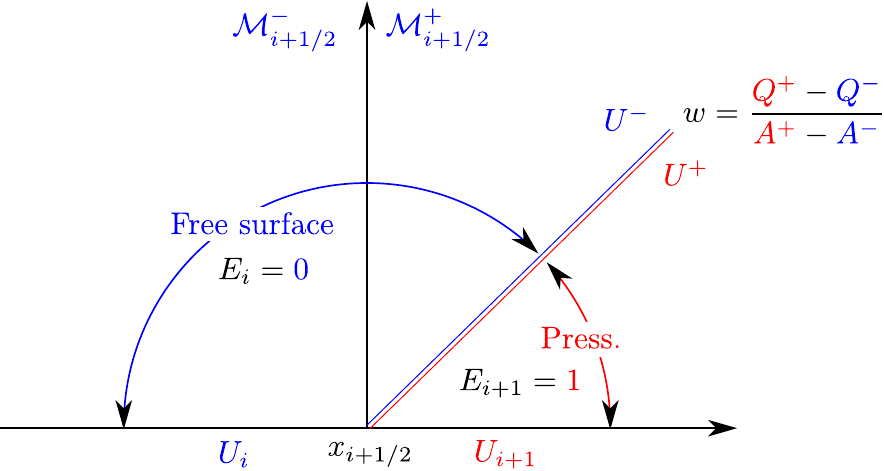}
  \caption{Rankine-Hugoniot condition through the line $x=wt$.}
  \small{Free surface state propagating downstream.}
  \label{ExampleStateTransition}
 \end{center}
\end{figure}
%%%%%%%%%%%%%%%%%%%%%%%%%%
According to the left  $\UU^-$ and right unknowns $\UU^+$ at the interface $x_{i+1/2}$ and 
the sign of the speed $w$, we
have to deal with four cases:
\begin{itemize}
 \item pressurized state propagating downstream,
 \item pressurized state propagating upstream,
 \item free surface state propagating downstream,
 \item free surface state propagating upstream.
\end{itemize}
Assume that  $\UU^{\pm}$ are given then the kinetic scheme in the case of transition points reads:
\begin{equation*}
\UU_i^{n+1} = \UU_i^n+\frac{\Delta t^n}{h_{i}}\,(\F_{i+\frac{1}{2}}^{-}- \F_{i-\frac{1}{2}}^+)
\end{equation*}
where the numerical fluxes are computed by Equation \eqref{fluxdiscret} and the numerical
microscopic interface quantities
$\mathcal{M}_{i\pm1/2}^{\pm,n}$ are obtained according to the formulas
\eqref{MicroscopicInterfaceFluxes} and the sign of speed $w$ as
follows:
\begin{equation}\label{DefinitionMicroscopicFluxeW}
\begin{array}{l}
\mathcal{M}_{i+1/2}^{-} = 
    \left\{
       \begin{array}{lll}
         \mathcal{M}_{i+1/2}^{-}(\xi,\mathcal{M}_{i}^n,\mathcal{M}^{-})  &\textrm{ if }& w>0,\\
         \mathcal{M}_{i+1/2}^{-}(\xi,\mathcal{M}^{+},\mathcal{M}_{i+1}^{n})  &\textrm{ if }& w<0 ,\\
       \end{array}
    \right. \\  \\
\mathcal{M}_{i+1/2}^{+} = 
    \left\{
       \begin{array}{lll}
         \mathcal{M}_{i+1/2}^{+}(\xi,\mathcal{M}_{i}^n,\mathcal{M}^{-})  &\textrm{ if }& w>0,\\
         \mathcal{M}_{i+1/2}^{+}(\xi,\mathcal{M}^{+},\mathcal{M}_{i+1}^{n})  &\textrm{ if }& w<0.\\
       \end{array}
    \right. \\
\end{array}
\end{equation}
where $\mathcal{M}^{\pm}$ are the Gibbs equilibrium associated to $\UU^{\pm}$ according to the formula \eqref{Gibbs}.

A first way to compute  the unknowns $\UU^{\pm}$ called ``the ghost waves approach'' is to go back to the macroscopic level and to solve a linearized 
Riemann problem with discontinuous convection matrix.

The second way called ``\FKAl'' computes  the states $\mathcal{M}^{\pm}$  at the microscopic level and 
 the state $\UU^{\pm}$ are recovered by the relations \eqref{macroA}-\eqref{macroQ}.

Only two of four cases are considered since we have  two  couples of ``twin cases'': pressurized state is propagating
downstream (or upstream) as shown in figure \ref{plus1} and free surface state propagating
downstream (or upstream) as shown in figure \ref{plus2}.
%%%%%%%%%%%%%%%%%%%%%%%%%%%%%
\subsubsection{The ghost waves approach}\label{SectionTheGhostWavesApproach}
%%%%%%%%%%%%%%%%%%%%%%%%%%%%%
In order to specify the unknowns $\UU^{\pm}$, we have to define four equations. To this end,  the
ghost waves approach is a way to obtain a system of four equations related to the \textbf{PFS} system. 
Adding the equations $\partial_t Z=0$, $\partial_t \cos\theta=0$ and $\partial_t S=0$, the \textbf{PFS} model (without friction since it is taken into account in the term 
$\boldsymbol{Z}$ in the kinetic interpretation \eqref{KineticFormulationPFS}, \eqref{PFSSourceTermPhi}) can be written under a non-conservative form
with the variable $\VV =(Z,\cos\theta,S,A,Q)^t$:
$$\partial_t \VV+ D(\VV) \;\partial_x \VV =0,$$ with $D$ the convection matrix defined by
\begin{equation*}\label{ConvectionMatrix}
D(\VV) = \left(
\begin{array}{ccccc}
0 & 0 & 0 & 0 &  0\\
0 & 0 & 0 & 0 &  0\\
0 & 0 & 0 & 0 &  0\\
0 & 0 & 0 & 0 &  1\\
gA & gA\mathcal{H}(S_{w}) & \Psi(\VV) & c^2(\VV)-u^2 &  2u\\
\end{array}
\right)
\end{equation*}
where $\ds\Psi(\VV) = gS \partial_S \mathcal{H}(S_{w}) \cos\theta -c^2(\VV)\ds\frac{A}{S_{w}}$
and $u=Q/A$ denotes the speed of
the water. $c(\VV)$ is then  $c$ for the  pressurized flow  or  $\ds \sqrt{g\frac{A}{T(A)}\cos\theta}$ for the free surface flow.
%%%%%%%%%%%%%%%%%%%%%%%%%%%%%
\begin{rque}
Let us remark that, since $\partial_x I_1(x,A) = I_2(x,A)+\partial_A I_1(x,A) \partial_x
A$, the pressure source term $I_2$ does
not appear.
\end{rque}
%%%%%%%%%%%%%%%%%%%%%%%%%%%%%
We assume that the propagation of the interface (pressurized-free surface or  free
surface-pressurized) has a constant speed $w$ during a time step. 
The ghost waves approach consists to solve a linearized Riemann problem in each zone (see \cite{BEG09}):
\begin{equation*}%\label{rielin}
\left\lbrace
\begin{array}{lll}
\partial_t \VV + \widetilde{D} \;\partial_x \VV & = & 0,\\
\VV & = &
\left\lbrace
\begin{array}{lcr}
\VV_l & \hbox{ if }& x< w t,\\
\VV_r & \hbox{ if }& x> w t.
\end{array}
\right.
\end{array}
\right.
\end{equation*}

where $\ds \widetilde{D}=D(\widetilde{\VV})$. The term $\ds \widetilde{\VV}$ is some
approximation depending on the left
$\VV_{l}$ and the right $\VV_{r}$  state.

The half line  $x = w\,t$ is then the discontinuity
line of $\widetilde{D}$. Both states $\UU_{i}$ and  $\UU^-$ (resp. $\UU_{i+1}$ and $\UU^+$) correspond  to the same type of flow. 
Thus it makes sense to define the averaged matrices in each zone as follows:
\begin{itemize}
\item for $x < w\,t$, we set $\widetilde{D}_{i}=D(\widetilde{\VV}_{i})$ for some approximation
$\widetilde{\VV}_{i}$ which connects the state $\VV_{i}$ and $\VV^-$.

\item for  $x > w\,t$, we set  $\widetilde{D}_{i+1}=  D(\widetilde{\VV}_{i+1})$ for some
approximation $\widetilde{\VV}_{i+1}$ which connects the state $\VV^+$ and $\VV_{i+1}$.
\end{itemize}

Then we formally solve two Riemann problems and use the Rankine-Hugoniot jump conditions
through the line $x = w\,t$ which writes:
\begin{eqnarray}
Q^+ - Q^- & = & w\,(A^+ - A^-),\label{rh1}\\
F_2(A^+,Q^+) - F_2(A^-,Q^-)& =& w\,(Q^+ - Q^-),\label{rh2}
\end{eqnarray}
with $F_2$ defined by \eqref{QFlux}. In what follows, all quantities of the form
$\widetilde{v}_i$ (resp. $\widetilde{v}_{i+1}$)
define some approximation which connects the state $v_{i}$ and $v^-$ (resp. $v^+$ and $v_{i+1}$).
States can be connected, for instance, by the mean value of each state.
%%%%%%%%%%%%%%%%%%%%%%%%%%%%%
\paragraph{Pressurized state propagating downstream:}
%%%%%%%%%%%%%%%%%%%%%%%%%%%%%
it is the case when on the left hand side of the line $\xi = w t$, we have a pressurized flow while on the right hand side we have a free
surface flow and the speed $w$ of the transition point is positive.
Following Song \cite{SCL83} (see also \cite{M02}), an equivalent stationary hydraulic jump must
occur from a supercritical to a subcritical
condition and thus the characteristics speed satisfies the inequalities:
\begin{equation*}
\ut_{i+1}+\ct_{i+1}< w <\ut_{i} + c\,.
\end{equation*}
%%%%%%%%%%%%%%%%%%%%%%%%%%%%%
\begin{figure}[H]
\centering
\includegraphics[height=5cm]{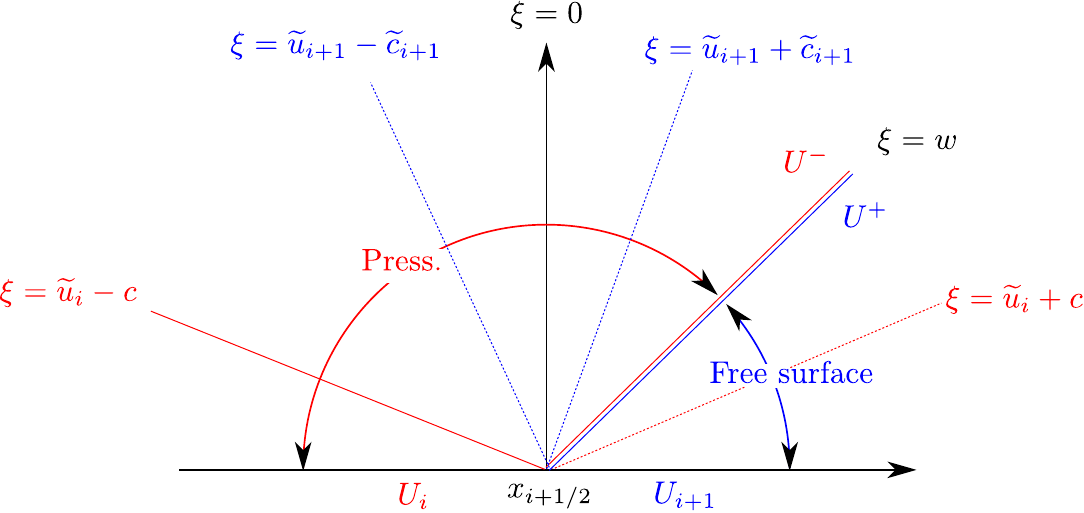}
\caption{Pressurized state propagating downstream.}\label{plus1}
\end{figure}
%%%%%%%%%%%%%%%%%%%%%%%%%%%%%
Therefore, only the characteristic lines drawn with solid lines are taken into account.
Indeed they are related to incoming waves
with respect to the corresponding space-time area $-\infty<\xi<w$. Conversely, the dotted line $
\xi=\ut_{i+1}-\ct_{i+1}$, for instance,
related to the free surface zone but drawn in the area of pressurized flow is a ``ghost wave'' and is not
considered. Thus $\UU^+=\UU_{i+1}$ and  $\UU_{i}$, $\UU^-$ are connected through the jumps across the characteristics $\xi=0$ and $\xi = \ut_i-c$.
Eliminating $w$ in the Rankine-Hugoniot jump relations \eqref{rh1}-\eqref{rh2}, we get $\UU^-$ as the
solution to the nonlinear system:

\begin{eqnarray}
\ds (F_2(A_{i+1},Q_{i+1}) - F_2(A^-,Q^-)) = \frac{(Q_{i+1} - Q^-)^2 }{(A_{i+1} - A^-)} ,\nonumber\\%\label{chslp1}\\
Q^- -Q_{i} - (A^- - A_{i})(\ut_{i} - c) + \frac{g \psi^{i+1}_{i} \,\At_{i}}{c + \ut_{i}}  & = & 0, \nonumber\\%\label{chslp2}
\end{eqnarray}
where  $ \psi^{i+1}_{i}$ is the upwinded source term $$
\boldsymbol{Z_{i+1}}-\boldsymbol{Z_{i}}+\mathcal{H}(\widetilde{S_{w}})(\cos\theta_{i+1}-\cos\theta_{i})+\Psi(\widetilde{\VV})(S_{i+1}-S_{i})\,.$$
%%%%%%%%%%%%%%%%%%%%%%%%%%%%%
\paragraph{Free surface state propagating downstream:}
%%%%%%%%%%%%%%%%%%%%%%%%%%%%%
it is the case when on the left hand side of the line $\xi = w t$, we have a free surface flow while on the right hand side we have a pressurized flow and 
the speed $w$ of the transition point is  positive. Following again Song \cite{SCL83},  the characteristic speed
satisfies the inequalities:
\begin{equation*}
\ut_{i}+\ct_{i} < w < \ut_{i+1} + c
\end{equation*}
%%%%%%%%%%%%%%%%%%%%%%%%%%%%%
\begin{figure}[H]
\centering
\includegraphics[height=5cm]{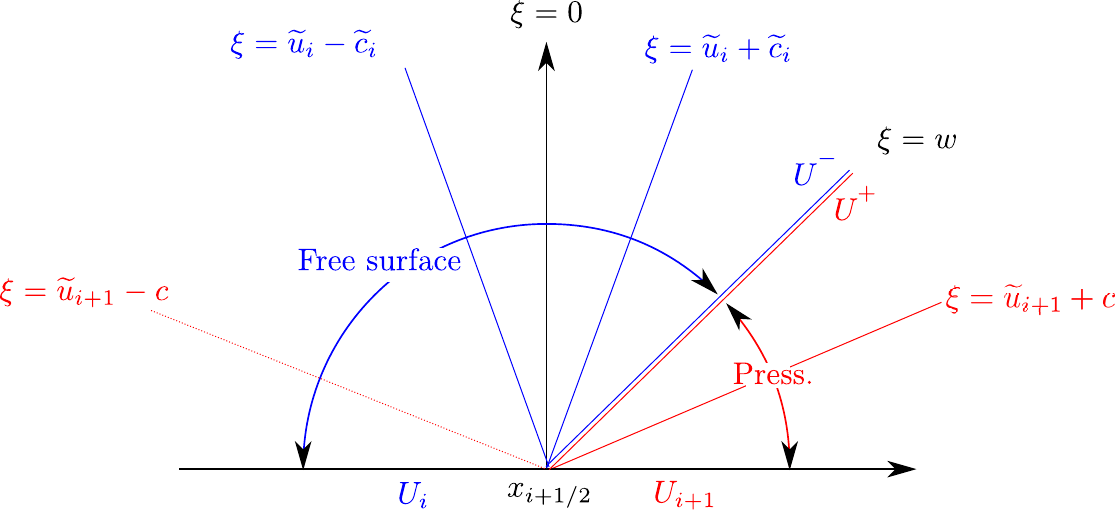}
\caption{Free surface state propagating downstream.} \label{plus2}
\end{figure}
%%%%%%%%%%%%%%%%%%%%%%%%%%%%%

There are two incoming characteristic lines with respect to the free surface area
$-\infty<\xi<w$ (actually three with $\xi=0$) and they can connect the given left state $\UU_{i}$ with any arbitrary free surface state.  Thus only
one  characteristic line ($\xi=\ut_{i+1} + c$) gives any information (it is  Equation \eqref{slchp1} above) as an incoming
characteristic line with respect to the pressurized zone $w<\xi<+\infty$. From the jump relations through the characteristic $\xi~=~0$, and
after the elimination of $w$ in the Rankine-Hugoniot jump relations \eqref{rh1},\eqref{rh2} we get another equation, namely Equation \eqref{slchp2} above.
It remains to close the system of four unknowns $(A^-,\,Q^-,\,A^+,\,Q^+)$.
Thus, from Equation \eqref{ThmPFSEquationForU}, we use the jump relation across the transition point (with speed $w$) for the total head $\ds
\Phi(A,Q,\cos\theta,\boldsymbol{Z},E) = \frac{u^2}{2} + c^2 \ln\left(\frac{A}{S_{w}}\right) + g\,\boldsymbol{\mathcal{H}}(A)\cos\theta + g\,\boldsymbol{Z}$, which writes:
\begin{equation}\label{sautdecharge}
\Phi^+ - \Phi^- = w\,(u^+ - u^-)\,.
\end{equation}

Finally, we use the relation:
$$ w = w_{pred} \, \textrm{ with } w_{pred} = \frac{Q_{i+1}-Q_{i}}{A_{i+1}-A_{i}}\,,$$
which is the predicted speed of the discontinuity.
We have then to solve the nonlinear system:
\begin{eqnarray}
(Q_{i+1} - Q^+)  = (A_{i+1} - A^+)\,(\ut_{i+1} + c), \label{slchp1}\\
\nonumber\\
(Q^+ - Q^-)\,(Q_{i+1} - Q_{i})  =  (A_{i+1} - A_{i})\, (F_2(A^+,Q^+) - F_2(A^-,Q^-)),\label{slchp2}\\
\nonumber\\
\frac{(Q^+)^2}{2\,(A^+)^2} + c^2 \ln\left(A^+\right) + g\cos\theta\,\mathcal{H}(A^+)
- \frac{(Q^-)^2}{2\,(A^-)^2}- c^2 \ln\left(A^-\right)- g\cos\theta\,\mathcal{H}(A^-) ,\nonumber\\
=  \frac{Q_{i+1} - Q_{i}}{A_{i+1} - A_{i}}\,\left( \frac{Q^+}{A^+} - \frac{Q^-}{A^-}  \right),
\label{slchp3}\\
\nonumber\\
(Q_{i+1} - Q_{i})\, (A^+ - A^-) = (Q^+ - Q^-)\, (A_{i+1} - A_{i})\label{slchp4} \ .%
\end{eqnarray}
%%%%%%%%%%%%%%%%%%%%%%%%%%%%%
\subsubsection{The \FKAl}\label{FKA}
%%%%%%%%%%%%%%%%%%%%%%%%%%%%%
As we will show in the numerical experiments part of this article, the  ``ghost waves approach'' method seems to produce very good
results in agreement with experimental data. But, as we have shown in the previous section, to solve the interface problem between the free surface and
 the pressurized part of the flow, we have expressed mathematical relations between the macroscopic unknowns $A$ and $Q$ whereas 
 in the free surface part as well as the pressurized part of the flow, we only deal with the microscopic quantities $\mathcal{M}_{i}^n$ and we have constructed
 kinetic fluxes at the interface that take into account every source terms.

At this point, we proposed to solve the interface problem at the microscopic level.
This method is based on the generalized characteristics method applied to the transport equation
\eqref{KineticFormulationPFS}, see \cite{Da89,TheseErsoy}.

As in the previous method, we will only consider two generic cases: pressurized state propagating downstream, as shown in figure  \ref{chslpluscin},  
and free surface state propagating downstream, as shown in figure  \ref{slchpluscin}.
 For these two cases, we have to determine the Gibbs equilibrium $\M^\pm$  corresponding to the macroscopic states
$\UU^\pm$ on both sides of the curve $\Gamma$ representing the trajectory of the transition point whose unknown velocity is $w$. 
On each time step, this trajectory is supposed to be a line and $w$ must be defined such that the Rankine Hugoniot relations are satisfied.

We will now use the Gibbs equilibrium $\M^-$ instead of  $\M_{i+1}$ in the computations of the microscopic fluxes,
see formula \eqref{DefinitionMicroscopicFluxeW}, for the case $w>0$

The transport equation \eqref{KineticFormulationPFS} is used only for the transmission cases since we are only interested in trajectories which
intersect $\Gamma$.

At the microscopic level, we obtain:
\begin{equation*}%\label{MiM-}
 \forall \xi>w,\quad \M^-(\xi)\,\mathds{1}_{\{|\xi|^2+2\,g\,\Delta \phi^n_{i+1/2}>0\}}=
 \M_i^n\left(\sqrt{|\xi|^2+2\,g\,\Delta \phi^n_{i+1/2}}\;\right)\,\mathds{1}_{\{|\xi|^2+2\,g\,\Delta \phi^n_{i+1/2}>0 \}}, 
\end{equation*} 
and
\begin{equation*}%\label{Mi+1M+}
 \forall \xi<w,\quad \M^+(\xi)=\M_{i+1}^n(\xi).
\end{equation*} 

These relations give, at the macroscopic level:

\begin{equation}\label{macroMiM-}
\int_w^{+\infty}\left(\begin{array}{c}
1 \\
\xi
\end{array}\right)\M^-(\xi)\,\mathds{1}_{\{|\xi|^2+2\,g\,\Delta \phi^n_{i+1/2}>0
\}}d\xi=\int_w^{+\infty}\left( \begin{array}{c}
1 \\
\xi
\end{array}\right)\M_i^n\left(\sqrt{|\xi|^2+2\,g\,\phi^n_{i+1/2}}\; \right)\,
\mathds{1}_{\{|\xi|^2+2\,g\,\phi^n_{i+1/2}>0 \}}d\xi 
\end{equation} 
and
\begin{equation}\label{macroMi+1M+}
\int_{-\infty}^w\left( \begin{array}{c}
1 \\
\xi
\end{array}\right)\, \M^+(\xi)\,d\xi=\int_{-\infty}^w\left( \begin{array}{c}
1 \\
\xi
\end{array}\right)\,\M_{i+1}^n(\xi)\,d\xi.
\end{equation} 
For the sake of simplicity, let us denote $c(A)$ the sound speed $b(x,A,E)$ defined by equation \eqref{soncinetique}.

Let us mention that some preceding relations are obvious since the support of $\M(\xi)$ is $[u\pm c(A)\sqrt{3}]$.
For instance for a free surface flow propagating downstream (or for the twin case of a pressurized flow propagating upstream),
in the free surface zone, we have to compare in Equations  \eqref{macroMiM-} the velocity of the interface 
$w$ with $u+c(A)\sqrt{3}$  while in the pressurized zone $u+c(A)\sqrt{3}$ is very large . The only corrective term is due to the slope.

From now on, we omit the indexes $n$ to simplify the notations and every computations are done with the density function $\chi$ defined by 
\eqref{chifunction}.

Let us define the ``effective'' boundary in the integral used in Equations \eqref{macroMiM-}:
\begin{eqnarray}
w'&=&\max\Big(w,\sqrt{2\,g\,\max(0,-\phi_{i+1/2}})\Big)\label{wprime} ,\\
 w''&=&\max\left(\sqrt{\max(0,w^2+2\,g\,\Delta \phi^n_{i+1/2}}),\sqrt{2\,g\,\max(0,\phi_{i+1/2}})\right) \label{wseconde} \ .
\end{eqnarray}

Equations  \eqref{macroMiM-} write:
\begin{eqnarray}
 \frac{A^-}{c(A^-)}\,(\delta^- - \gamma^-) & =
&\frac{A_i}{c(A_i)}\,\left(\sqrt{\delta_i^2-2\,g\,\phi_{i+1/2}}-\sqrt{\gamma_i^2-2\,g\,
\phi_{i+1/2}}\right)\label{unmoins},\\
\frac{A^-}{c(A^-)}\, ((\delta^-)^2 - (\gamma^-)^2) & = &\frac{A_i}{c(A_i)}\,
(\delta_i^2-\gamma_i^2)\label{deuxmoins},
\end{eqnarray}
where $\gamma^-,\,\delta^-,\,\gamma_i,\,\delta_i$ are defined by:
$$\gamma^-=\max\left(w',u^--c(A^-)\sqrt{3}\right),\quad \delta^-=\max\left(w',u^-+c(A^-)\sqrt{3}\right),$$
 and
$$\gamma_i=\max\left(w'',u_i-c(A_i)\sqrt{3}\right),\quad \delta_i=\max\left(w'',u_i+c(A_i)\sqrt{3}\right) \ .$$

In the same way, Equations \eqref{macroMi+1M+} write:

\begin{eqnarray}
 \frac{A^+}{c(A^+)}\,(\beta^+ - \alpha^+) & =
&\frac{A_{i+1}}{c(A_{i+1})}\,(\beta_{i+1} - \alpha_{i+1})\label{unplus},\\
\frac{A^+}{c(A^+)}\,((\beta^+)^2 - (\alpha^+)^2) & =
&\frac{A_{i+1}}{c(A_{i+1})}\,(\beta_{i+1}^2 - \alpha_{i+1}^2)\label{deuxplus},
\end{eqnarray}
where $\alpha^+,\,\beta^+,\,\alpha_{i+1},\,\beta_{i+1}$ are defined by:
$$\alpha^+=\min\left(w,u^+-c(A^+)\sqrt{3}\right),\quad \beta^+=\min\left(w,u^++c(A^+)\sqrt{3}\right),$$
and
$$\alpha_{i+1}=\min\left(w,u_{i+1}-c(A_{i+1})\sqrt{3}\right),\quad
\beta_{i+1}=\min\left(w,u_{i+1}+c(A_{i+1})\sqrt{3}\right).$$
Taking into account Equation \eqref{unplus},  Equation \eqref{deuxplus} becomes: 
\begin{equation*}%\label{deuxplusbis}
\alpha^+ +  \beta^+ = \alpha_{i+1} + \beta_{i+1} \ .
\end{equation*} 

The unknowns are still $w,\,A^-,\,Q^-,\,A+,\,Q^+$. 
%%%%%%%%%%%%%%%%%%%%%%%%%%%%%
\paragraph{Pressurized state propagating downstream:}
%%%%%%%%%%%%%%%%%%%%%%%%%%%%%
let us first remark that in this case, since in the mesh $i$ the flow is pressurized, we always have  $w' \leq u_i+c(A_i)\sqrt{3}$ so that $\gamma^- - \delta^- \neq 0$.
We solve then the system formed by the  Rankine-Hugoniot jump conditions \eqref{rh1}-\eqref{rh2}, and Equations \eqref{unmoins},  \eqref{deuxmoins},
and  \eqref{unplus} which formed a full nonlinear system of 5 equations with 5 unknowns. 
So this procedure privileges the information coming from the zone containing the cell interface $x_{i+1/2}$.

%%%%%%%%%%%%%%%%%%%%%%%%%%%%%
\paragraph{Free surface state propagating downstream:}
%%%%%%%%%%%%%%%%%%%%%%%%%%%%%
again, we solve the nonlinear system formed by Equations \eqref{rh1}-\eqref{rh2},  \eqref{unmoins}, \eqref{deuxmoins}
 and \eqref{unplus} and the experience acquired with the industrial code  \verb+FlowMix+ has shown that the system is a  full nonlinear system of 
 5 equations with 5 unknowns.

But since in the mesh $i$ the flow is a free surface one, we may have the critical case where  $w'>u_i+c(A_i)\sqrt{3}$ with $w'$ defined by Equation \eqref{wprime}.
In this case $\gamma^- - \delta^-   = 0$ and the system formed by the nonlinear equations  \eqref{rh1}-\eqref{rh2},  \eqref{unmoins}, \eqref{deuxmoins} 
and \eqref{unplus} is under determined.
In this case, as in the ghost waves approach, we replace $w$ by $w_{pred} = \ds\frac{Q_{i+1}-Q_{i}}{A_{i+1}-A_{i}}$ and we solve the system formed by 
 the  Rankine-Hugoniot jump conditions  \eqref{rh1}, \eqref{rh2} and \eqref{sautdecharge}, as in the ghost waves approach, completed by Equations 
 \eqref{unplus} and \eqref{deuxplus}.

%%%%%%%%%%%%%%%%%%%%%%%%%%%%%
\begin{figure}[H]
\begin{center}
\subfigure[Pressurized state propagating downstream. \label{chslpluscin}]
{
\includegraphics[height=5cm]{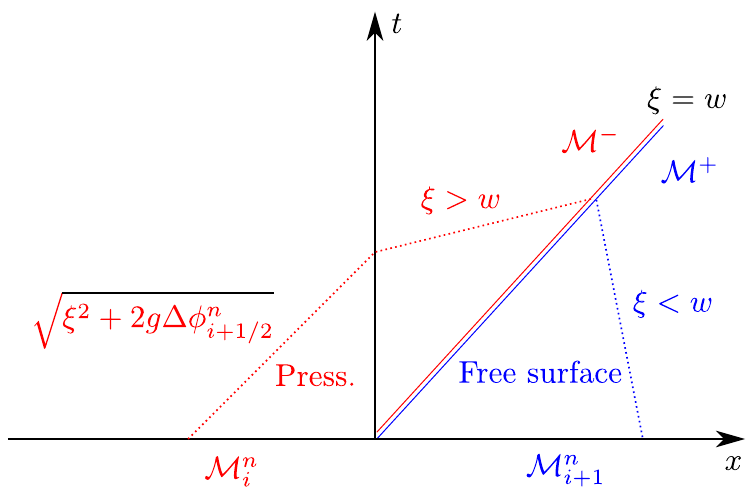}
}
\subfigure[Free surface state propagating downstream. \label{slchpluscin}]
{
\includegraphics[height = 5cm]{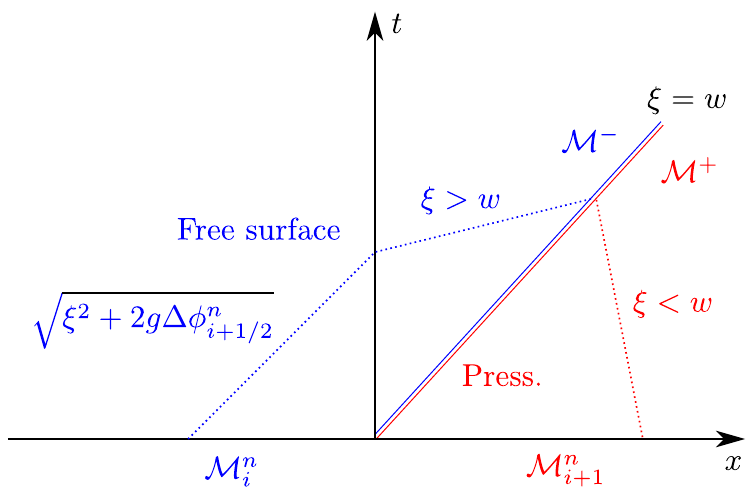}
}
\caption{\FKAl  \: for the case of a transition point.}
\end{center}
\end{figure}

%%%%%%%%%%%%%%%%%%%%%%%%%%%%%
\subsection{Boundary conditions}\label{SectionBoundaryConditionsWithNoTransitionPoints}
%%%%%%%%%%%%%%%%%%%%%%%%%%%%%
Let us recall that $x_{1/2}$ and $x_{N+1/2}$ are respectively the upstream and the downstream ends of the pipe. 
At this stage, we have computed all the ``interior'' states at time $t^{n+1}$, that is $(\UU_{i}^{n+1})_{i=1,N}$ are computed, and also the state of the ``interior'' cells
$(E_{i}^{n+1})_{i=1,N}$ by the method presented in section \ref{SectionUpdatingTheStateEInACell}.

The upstream state $\UU_{0}$ corresponds to the mean value of $A$ and $Q$ on the ``fictitious'' cell $m_0=(x_{-1/2},x_{1/2})$ (at the left of the upstream
boundary of the pipe) and the downstream state $\UU_{N+1}$ corresponds to the mean value of $A$ and $Q$ on the ``fictitious'' cell $m_{N+1}=(x_{N+1/2},x_{N+3/2})$
(at the right of the downstream end of the pipe).

Usually,  we have to prescribe one boundary condition related to the state vectors
$\UU_0^n$ and $\UU_{N+1}^n$ (there is generically only one incoming characteristics curve  for the upstream and  only one outgoing characteristics for the downstream).
For instance, at the upstream end of the pipe, one of the following boundary conditions may be prescribed (we omit the index $n+1$ for the sake of simplicity): 
\begin{enumerate}
 \item the water level is prescribed. So let $H_{up}(t)$ be a given function of time. Then we have:  
\begin{equation}\label{amonthauteur}
\forall t > 0 ,\quad  \ds \frac{c^2}{g}\ln\left(\frac{A_0(t)}{S_{w 0}(t)}\right) + \mathcal{H}(S_{w 0}(t))\cos\theta_0+\mathbf{Z_0} = H_{up}(t) \ .
\end{equation}
 \item the discharge is prescribed. So let $Q_{up}(t)$ be a given function of time. Then we have: 
 \begin{equation}\label{amontdebit}
\forall t > 0 ,\quad \ds Q_0(t) = Q_{up}(t) \ .
\end{equation}
 \item the total head may be prescribed. So let $\Phi_{up}(t)$ be a given function of time. Then we have:  
 \begin{equation}\label{amontcharge}
 \forall t > 0,\quad \frac{Q^2_0(t)}{2\,A_0(t)} + c^2\ln\left(\frac{A_0(t)}{S_{w 0}(t)}\right) + g\mathcal{H}(S_{w 0}(t))\cos\theta_0+g\boldsymbol{Z}_0 = \Phi_{up}(t) \ .
 \end{equation}
 \end{enumerate}
At the downstream end, similar boundary conditions may be defined.
%\begin{enumerate}
% \item the water level is prescribed. So let $H_{down}(t)$ be a given function of time. Then we have:  
%\begin{equation*}
%\forall t > 0 ,\quad  \ds \frac{c^2}{g}\ln\left(\frac{A_{N+1}(t)}{S_{w}_{N+1}(t)}\right) + \mathcal{H}(S_{w}_{N+1}(t))\cos\theta_{N+1}+Z_{N+1} = H_{down}(t) \ .
%\end{equation*}
% \item the discharge is prescribed. So let $Q_{down}(t)$ be a given function of time. Then we have: 
% \begin{equation*}
% \forall t > 0 ,\quad \ds Q_{N+1}(t) = Q_{down}(t) \ .
% \end{equation*}
% \item the total head may be prescribed. So let $\Phi_{down}(t)$ be a given function of time. Then we have:  
% \begin{equation*}
% \forall t > 0,\quad \frac{Q^2_{N+1}(t)}{2\,A_{N+1}(t)} + c^2\ln\left(\frac{A_{N+1}(t)}{S_{w}_{N+1}(t)}\right) + g\mathcal{H}(S_{w}_{N+1}(t))\cos\theta_{N+1}+g
% \boldsymbol{Z}_{N+1} 
% = \Phi_{down}(t) \ .
% \end{equation*}
%\end{enumerate}
In order to find a complete state for the upstream boundary, $\UU_{0}^{n+1}$, and the downstream boundary, $\UU_{N+1}^{n+1}$, we have to define the
missing equation $b_{up}$ and $b_{down}$ respectively. 
We will present the method at the upstream boundary of the pipe (it is easy to adapt it at the downstream boundary).

At this stage, we have to consider two cases:
\begin{enumerate}
\item the upstream boundary is not a transition point, that is $E_{0} = E_{1}$,
\item  the upstream boundary is a transition point, that is $E_{0} \neq E_{1}$.
\end{enumerate}
We will treat the first case at the microscopic level while for the second case, we will describe the ``\FKAl'' (the ``ghost waves approach'' for interior cells
may be easily adapted).

%%%%%%%%%%%%%%%%%%%%%%%%%%%%%%%%%%%
\subsubsection{The boundary is not a transition point}
%%%%%%%%%%%%%%%%%%%%%%%%%%%%%%%%%%%

The jump $\Delta \phi^n_{1/2}$  is thus computed on the first half mesh.

Figure  \ref{amontcin1}  represents the case when the upstream boundary of the pipe is not a transition point between free surface and pressurized flow.
The Gibbs equilibrium $\mathcal{M}_0$ must be determined and we recall that we know either $A_{0}$ (if the upstream water level is prescribed by Equation
\eqref{amonthauteur}), $Q_{0}$ (if the upstream discharge is prescribed by Equation \eqref{amontdebit}) or a relation between these quantities 
(if the upstream total head is prescribed by Equation \eqref{amontcharge}).

\begin{figure}[H]
\centering
\includegraphics[height=4cm]{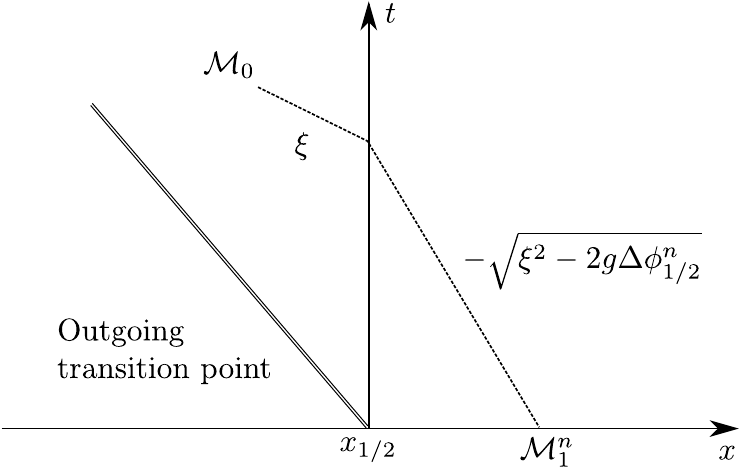}
\caption{The upstream boundary of the pipe is not a transition point.} \label{amontcin1}
\end{figure}
%%%%%%%
To obtain another equation denoted $b_{up}(A_{0},Q_{0})$, we will use the information transmitted 
from the outgoing kinetic characteristic (see figure \ref{amontcin1}).
At the macroscopic level, we obtain:
\begin{equation*}%\label{macroamont}
 \int_{-\infty}^{\xi_0}\left( \begin{array}{c}
 1 \\
 \xi
 \end{array}\right)\, \M_0(\xi)\,d\xi=\int_{-\infty}^{\xi_0}\left( \begin{array}{c}
 1 \\
 \xi
 \end{array}\right)\,\M_1^n\left(-\sqrt{|\xi|^2-2\,g\,\Delta \phi^n_{1/2}}\;\right)\,d\xi, 
 \end{equation*} 

with $\ds\xi_0 = -\sqrt{2\,g\,\max(0,\Delta \phi^n_{1/2})}$. This also writes as:

\begin{eqnarray}
 \frac{A_0}{c(A_0)}\,(\delta_0 - \gamma_0) & =
&\frac{A_1}{c(A_1)}\,\left(\sqrt{\gamma_1^2+2\,g\,\phi_{1/2}}-\sqrt{\delta_1^2+2\,g\,
\phi_{1/2}}\right)\label{unamont},\\
\frac{A_0}{c(A_0)}\, (\delta_0^2 - \gamma_0^2) & = &\frac{A_1}{c(A_1)}\,
(\delta_1^2-\gamma_1^2)\label{deuxamont},
\end{eqnarray}
where $\gamma_0,\,\delta_0,\,\gamma_1,\,\delta_1$ are defined by:

$$\gamma_0=\min\left(\xi_0,u_0-c(A_0)\sqrt{3}\right),\quad \delta_0=\min\left(\xi_0,u_0+c(A_0)\sqrt{3}\right),$$
and
$$\gamma_1=\min\left(\xi_1,u_1-c(A_1)\sqrt{3}\right),\quad \delta_1=\min\left(\xi_1,u_1+c(A_1)\sqrt{3}\right),$$
and $\xi_1=-\sqrt{2\,g\,\max(0,-\phi^n_{1/2})}$.

So we will use only one of the two equations \eqref{unamont} or  \eqref{deuxamont} as $b_{up}$.

Let us mention that in the case when $\gamma_1=\delta_1$, Equations  \eqref{unamont} and \eqref{deuxamont} are useless. This happens when:
 \begin{itemize}
 \item the flow is an incoming supcritical free surface flow at the upstream boundary condition (this could be due to a high slope inducing a great $\abs{\xi_1}$).
 In this case, we impose the critical flow that is $u_{0} = c(A_{0})$. 
 \item the flow is an outgoing supcritical free surface flow at the upstream boundary condition. The boundary conditions are then useless
 (two outgoing characteristics), so we impose $A_0=A_1$, $Q_0=Q_1$.
 \end{itemize}
\begin{rque} We will use
\begin{enumerate}
\item  Equation  \eqref{unamont} ``momentum of order 0'' if $Q_{0}$ is prescribed at the upstream end,
\item Equation  \eqref{deuxamont} ``momentum of order 1'' if $A_{0}$ is prescribed at the upstream end,
\item Equation \eqref{deuxamont} if the total head is prescribed. Since Equation \eqref{amontcharge} couples the two unknowns, 
we had the choice between the two equations \eqref{unamont} and \eqref{deuxamont}. The experience acquired with the industrial code 
\verb+FlowMix+ makes us to use Equation \eqref{deuxamont}.
\end{enumerate}
\end{rque}

%%%%%%%%%%%%%%%%%%%%%%%%%%%%%
\subsubsection{The boundary is a transition point}
%%%%%%%%%%%%%%%%%%%%%%%%%%%%%

We suppose that $E_{0} \neq E_{1}$. In the first step, we apply the procedure described in section \ref{FKA}:
the left state is the downstream state $\UU_{0}$, known at time $t^{n}$, and the right state is $\UU_{1}$ known at time $t^{n+1}$. We determine the state
$\UU^-$  at the left of the transition curve. As the two states $\UU_{0}$ and $\UU^{-}$ represent the same type of flow (free surface or pressurized),
we apply the preceding method by just replacing $\M_1^n$ by $\M^-$ in all the formula above (see figure \ref{amontcin2}).

%%%%%%%%%%%%%%%%%%%%%%%%%%%%%
\begin{figure}[H]
\centering
\includegraphics[height=4cm]{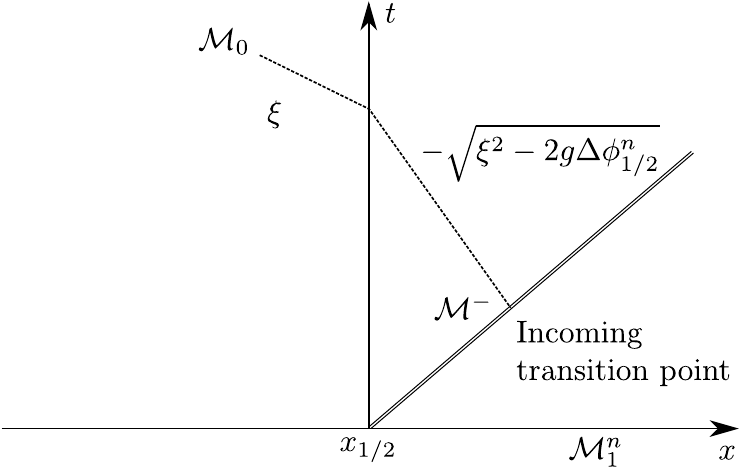}
\caption{The upstream boundary of the pipe is a transition point.} \label{amontcin2}
\end{figure}
%%%%%%%%%%%%%%%%%%%%%%%%%%%%%
\subsection{Updating the state of the mesh $E$}\label{SectionUpdatingTheStateEInACell}
%%%%%%%%%%%%%%%%%%%%%%%%%%%%%
To update the state of the mesh (see figure \ref{UpdateA}), we use a discrete version of the state
indicator $E$  equal to $1$ for a pressurized flow and $0$ otherwise. Following \cite{BG07}, after the computation of the wet area
$A_i^{n+1}$, we predict the state of the 
cell $m_i$ by the following criterion:
\begin{itemize}
\item[$\bullet$]  if $E_i^n = 0$ then:\\
if $A_i^{n+1}<S_i$ then   $E_i^{n+1}=0$, else $E_i^{n+1}=1$,

\item[$\bullet$]   if $E_i^n=1$:\\
if   $A_i^{n+1}\geq S_i$ then   $E_i^{n+1}=1$,
else $E_i^n=E_{i-1}^n\cdot E_{i+1}^n$.
\end{itemize}

Indeed, if $A_i^{n+1}\geq S_i$ it is clear that the mesh $m_i$ becomes pressurized, on the
other hand if $A_i^{n+1}<{S_i}$ in a
mesh previously pressurized, we do not know \textit{a priori} if the new state is free surface
($\rho=\rho_0$ and the value of the wet
area is less than ${S_i}$) or pressurized (in depression, with $\rho<\rho_0$ and the value of the
wet area is equal to $S_i$: see Remark \ref{RemarkDepressionOverpressureStates} and figure 
\ref{RhoCriterion}).\\
So far, as we do not take into account complex phenomena such as entrapment of air pockets or
cavitation and keeping in mind that the CFL condition, defined by \eqref{CFLcondition}, ensures that a transition point
crosses at most one mesh at each time step, we  postulate that:
\begin{enumerate}
\item if the mesh $m_i$ is free surface at time $t^{n}$, its state at time $t^{n+1}$ is only
determined by the value of $A_i^{n+1}$ and it
cannot become in depression.
\item if the mesh $m_i$ is pressurized at time $t^{n}$ and if $A_i^{n+1}<S_i$, it becomes free surface
if and only if at least one adjacent mesh was free surface  at time $t^{n}$. This is exactly the discrete version of the continuous
$\ds\frac{A}{S_{w}}$ criterion explained in
Remark \ref{RemarkDepressionOverpressureStates} and displayed on figure  \ref{RhoCriterion}.
\end{enumerate}

\begin{figure}[H]
 \begin{center}
 \includegraphics[height =7cm]{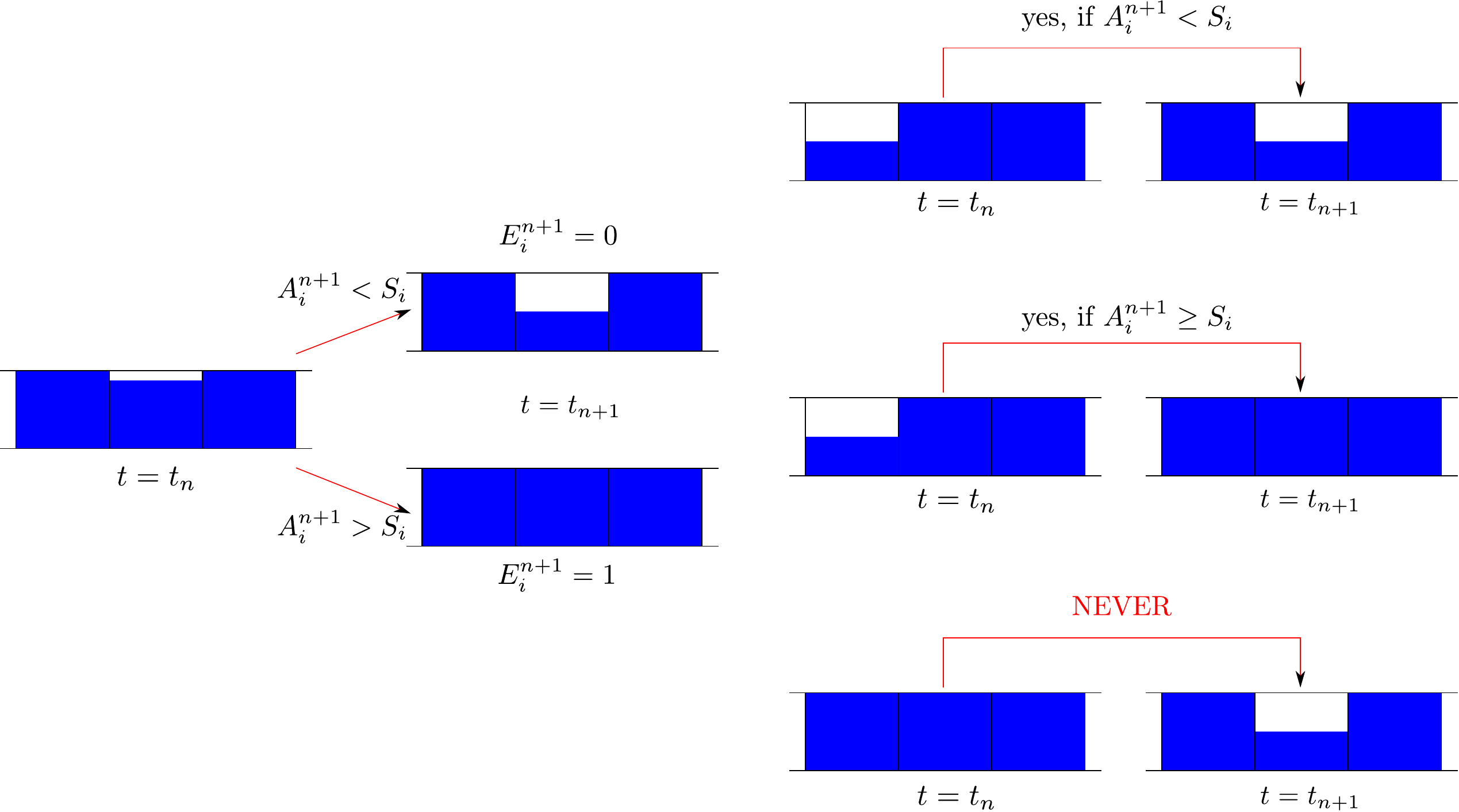}
  \caption{Update of the state $E_i^{n+1}$ of the mesh $m_i$.}
  \label{UpdateA}
 \end{center}
\end{figure}
\begin{rque}\label{RemarkDepressionOverpressureStates}
During a pressurized state, we may have $A>S$ or $A<S$. The case $A>S$ corresponds to
an overpressure while the second one is observed for the depression. As said before, the  main
difficulty to detect depression comes from
the fact that the criterion $A<S$ corresponds also to a free surface state. To dissociate these two
cases, we have to know if $\rho =
\rho_0$, that is also $\frac{A}{S_{w}}=1$
or not. On figure 
\ref{RhoCriterion}, we represent how to detect a depression, overpressure and free
surface state with the $\ds\frac{A}{S_{w}}$
criterion. To this end, we show  a physical situation at different time $t_i$, $i=0,\ldots,3$. We
draw the behavior of the interface speed
$w$ in the $(x,t)$-plane and the graph of the function $\ds\frac{A}{S_{w}}$ at fixed time $t_3$.
\begin{figure}[H]
 \begin{center}
 \includegraphics[height=17cm]{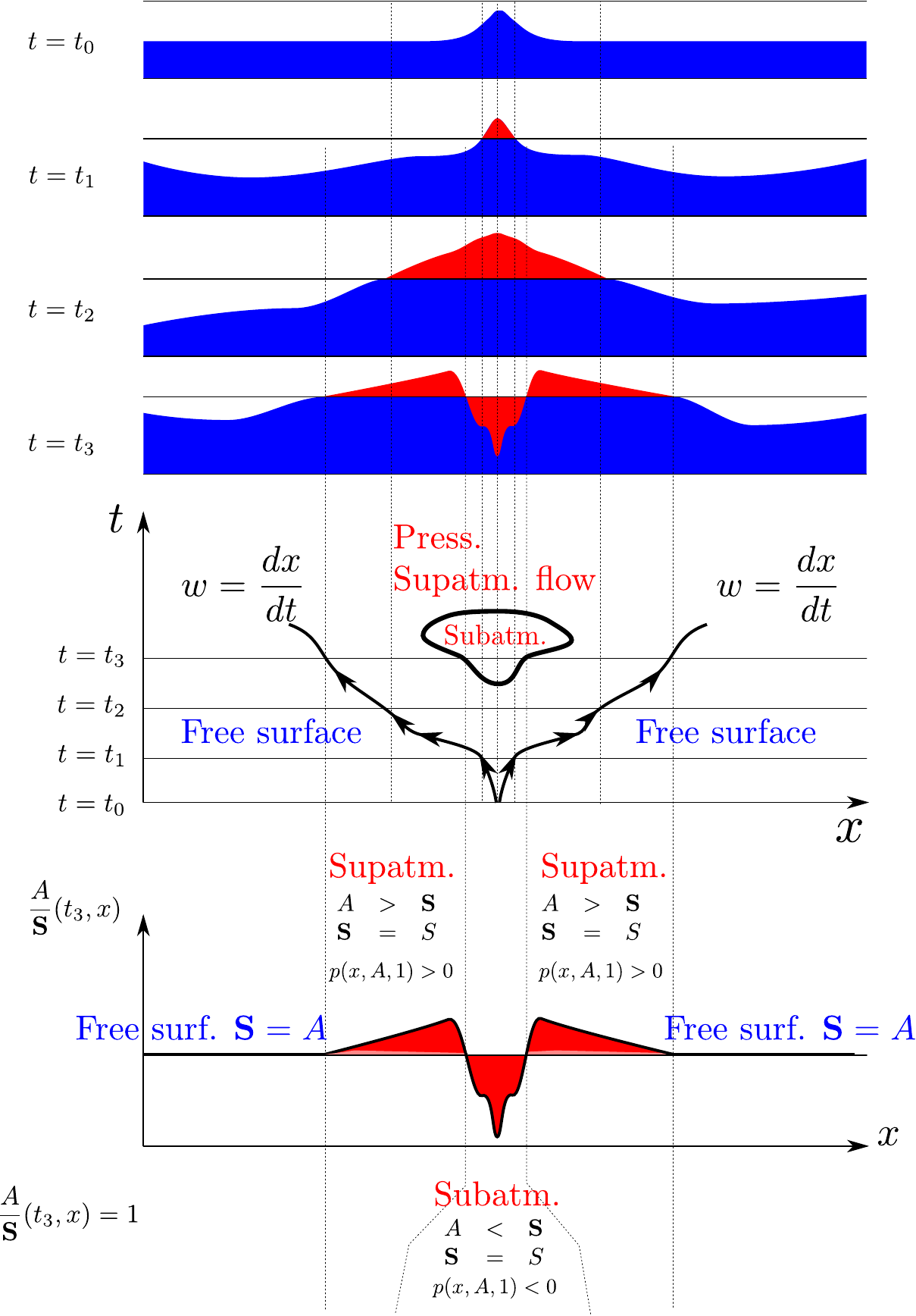}
  \caption{$\ds\frac{A}{S_{w}}$ criterion to detect depression area.}
  \label{RhoCriterion}
 \end{center}
\end{figure}
\end{rque}

%%%%%%%%%%%%%%%%%%%%%%%%%%%%%
\section{Numerical experiments}\label{numerics}
%%%%%%%%%%%%%%%%%%%%%%%%%%%%%
This section is devoted to a numerical validation of the model and the kinetic numerical scheme for the following main cases:
\begin{itemize}
\item single point pressurized flow: the Wiggert's test case,
\item a code to code validation for pressurized flow in uniform pipe,
\item numerical computations of steady states: free surface steady state in non uniform pipe as well as ``mixed'' steady state in uniform pipe. 
We present a numerical study for the order of the discretisation.
\item an example of drying and flooding flow,
\item and finally a comparison between the two approaches : the ghost wave approach and the \FKAl.
\end{itemize}
%%%%%%%%%%%%%%%%%%%%%%%%%%%%%
\subsection{Single point pressurised flow}\label{wigg}
%%%%%%%%%%%%%%%%%%%%%%%%%%%%%
In this section, we present  numerical results for the case of a single point pressurized flow, namely the test proposed by Wiggert
\cite{W72}. The numerical results are then compared with the experimental ones: a very good agreement between them is shown.
According to the experimental data, we were able to propose a value (or a range of values) for $c$ which seems physically relevant. 
On the contrary, in the Preissmann slot technique (\cite{W72,GAP94}) the value of $c$ is related to an arbitrary value (the width of the slot) and cannot exceed practically $10\:m/s$, otherwise the method becomes unstable.
The following test case, is due to Wiggert \cite{W72}.
The experimental device (see figure \ref{exp}) is an horizontal 10 $m$ long closed pipe with width 0.51 $m$ and height
$H = 0.148$ $m$.
The Manning number is $1/K_s^ 2= 0.012 \; s/m^{1/3}$. The initial conditions are a stationary state with the discharge $Q_0=0$ and the water
level $h_0=0.128 \, m$. \\
Then a wave coming from the left side causes the closed channel to pressurise.
The upstream condition is a given hydrograph ($y_2$ in figure \ref{Wdat}), at the downstream end, a step function is imposed:
the water level is kept constant to $h_0=0.128 \,m$ until  the wave reaches the exit.
At this time, the level is suddenly increased  (see $y_3$ in figure \ref{Wdat}).
For the computations, these boundary conditions have been read on Wiggert's article and rebuilt using piecewise polynomial interpolations (figure \ref{Wamav} below).\\
Other parameters are: 
$$
\begin{array}{lcl}
\hbox{Discretisation points}& : & 80,\\ 
\hbox{Delta x }(m)& : &0.125,\\
\hbox{CFL }& : & 0.5,\\
\hbox{Simulation time } (s) & : &18,\\
\hbox{Sound speed } (ms^{-1})& : &40.
\end{array}
$$

\vspace*{1cm}

\begin{figure}[H]
\centering
\includegraphics[height=3cm]{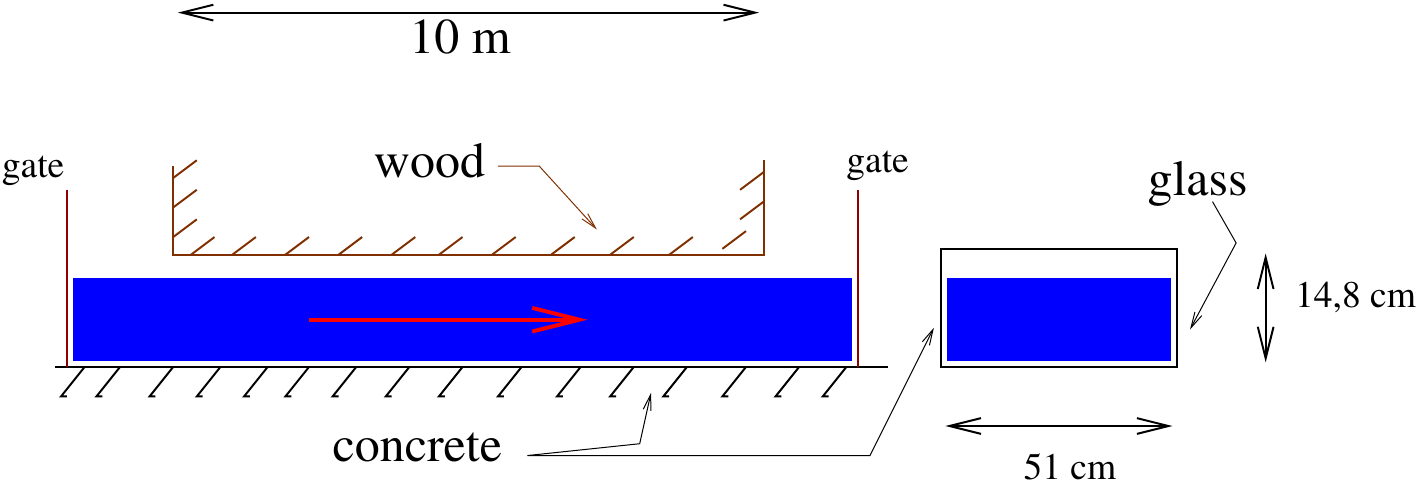}
\caption{Experimental device (adapted from Wiggert \cite{W72})\label{exp}.}
\end{figure}
\begin{figure}[H]
\centering
\includegraphics[height=6cm,angle=-90]{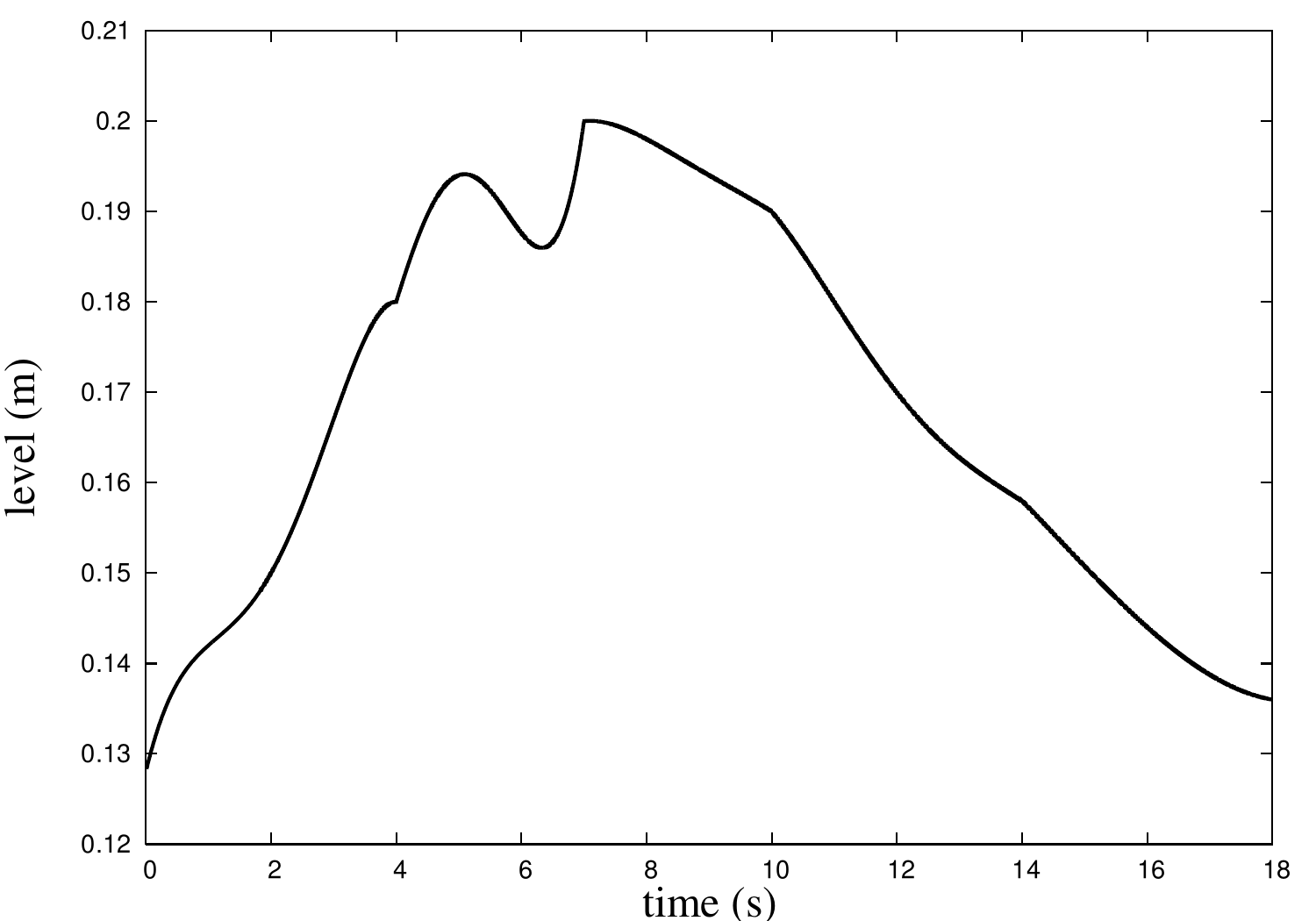}
\includegraphics[height=6cm,angle=-90]{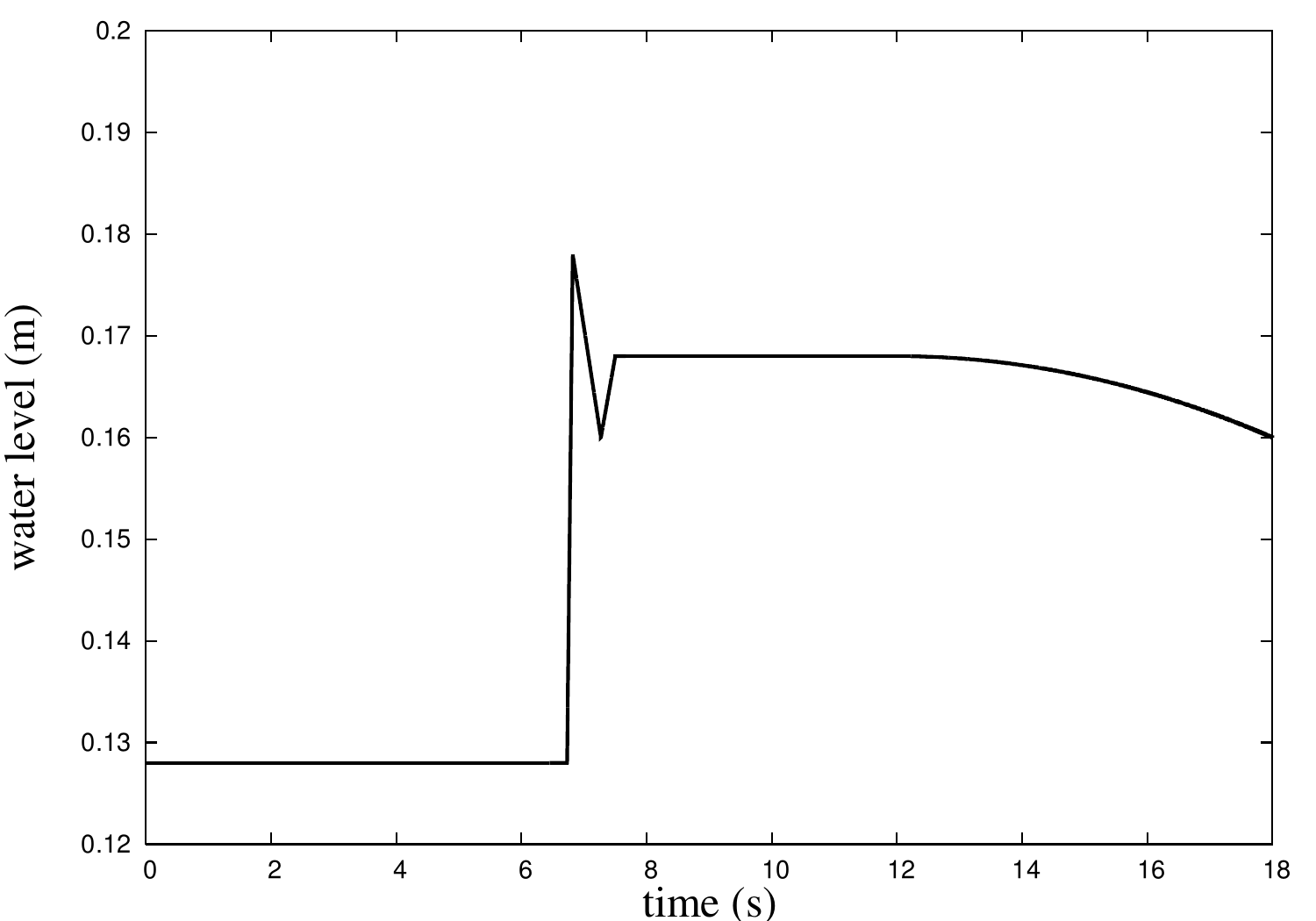}
\caption{Wiggert's test : upstream hydrograph (up) and downstream water level (down).\label{Wamav}}
\end{figure}
\begin{figure}[H]
\centering
\includegraphics[height=5cm]{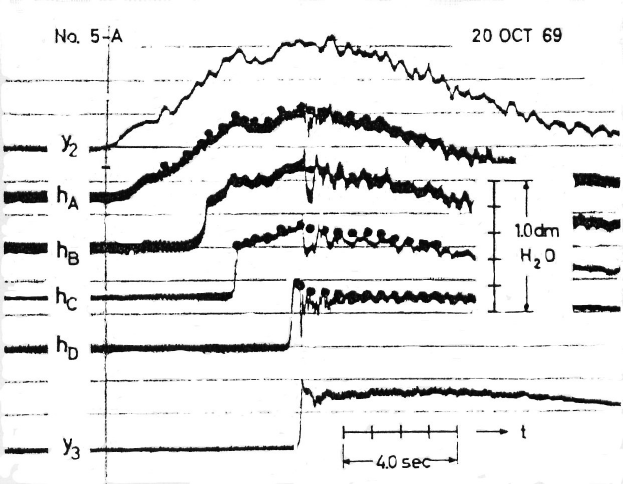}
\caption{Wiggert : experimental data. $y_2$ : upstream hydrograph, $y_3$ :  downstream hydrograph.
$h_A$, $h_B$, $h_C$, $h_D$ : pressure head at 0.5 $m$, 3.5 $m$, 5.5 $m$ and 9.5 $m$ from the tunnel entrance
(location of recording instruments) (\cite{W72}).\label{Wdat}}
\end{figure}

Let us define the piezometric head by:
$$\displaystyle  piezo = Z + \mathcal{H} + p \; \mbox{ with }
\left\{\begin{array}{l}
\displaystyle p = \frac{c^2 \, (\rho -\rho_0)}{\rho_0 \, g}
\mbox{ if the flow is pressurised},\\
p = h \mbox{ the water height if the flow is free surface}.
\end{array}
\right.
$$
In figure \ref{comp}, we present the piezometric line computed at 3.5 $m$ from the tunnel entrance (solid curve).
Circles represent experimental data read on curve $h_B$, including maxima and minima points of the oscillating parts.
We can observe a very good agreement with the experimental data even for the oscillations.
We point out that we did not find in other papers, by authors carrying out the same simulation, a convenient numerical
reproduction of these oscillations : they do not treat the dynamical aspect of the pressurized flow,
in particular when using the Preissmann slot technique (\cite{W72,GAP94}).
On the other hand, we found in M. Fuamba \cite{M02} a similar and interesting approach with a non conservative formulation and
another numerical method (characteristics).\\
The value of the sound speed $c$ was taken equal to 40 $m/s$, roughly according to the frequency of the oscillations observed
during the phase of total submersion of the tunnel. This low value can be explained by the structure of the tunnel and by
bubble flow (see \cite{HM82,WS93} for instance).\\
We  observe that the front reaches the control point at 3.6 $s$, in a good agreement with the experimental data
(less than 0.15 $s$ late). Let us mention that before it reaches the exit (part AB in figure \ref{comp}) the oscillations
of the pressure associated with the moving front reflect between upstream and the front itself
(since the free surface is at constant pressure) where the channel is flooded.
Beyond point B the oscillations result from the step in the downstream water level and they
propagate in the fully pressurized flow (their frequency was estimated using the BC part of the experimental curve).\\
Figure \ref{speed} gives the evolution of the front's speed. We observe the same behaviour as in \cite[Figure 7]{W72}:
the front quickly attains a maximum speed, decelerates and then slowly accelerates as it approaches the tunnel exit.
Moreover the values are consistent with those of Wiggert. 
\begin{figure}[H]
\begin{center}
\subfigure[Piezometric line at location $x=3.5 \,m$.
%(corresponding to $h_B$) with $c = 40 \; m/s$
\label{comp}]
{
\includegraphics[height=7cm,angle=-90]{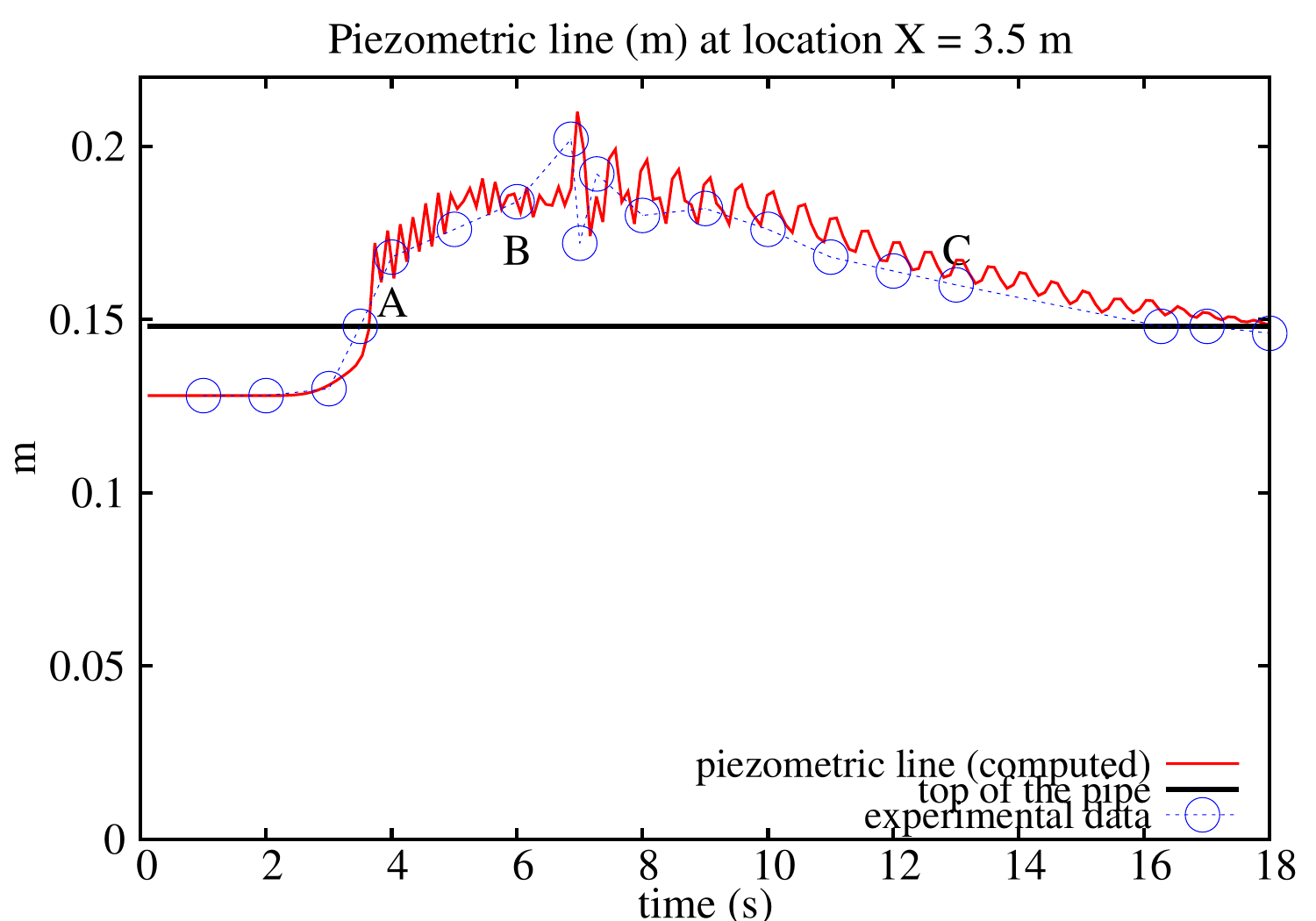}
}
\subfigure[Velocity of the transition point.\label{speed}]
{
\includegraphics[height=7cm,angle=-90]{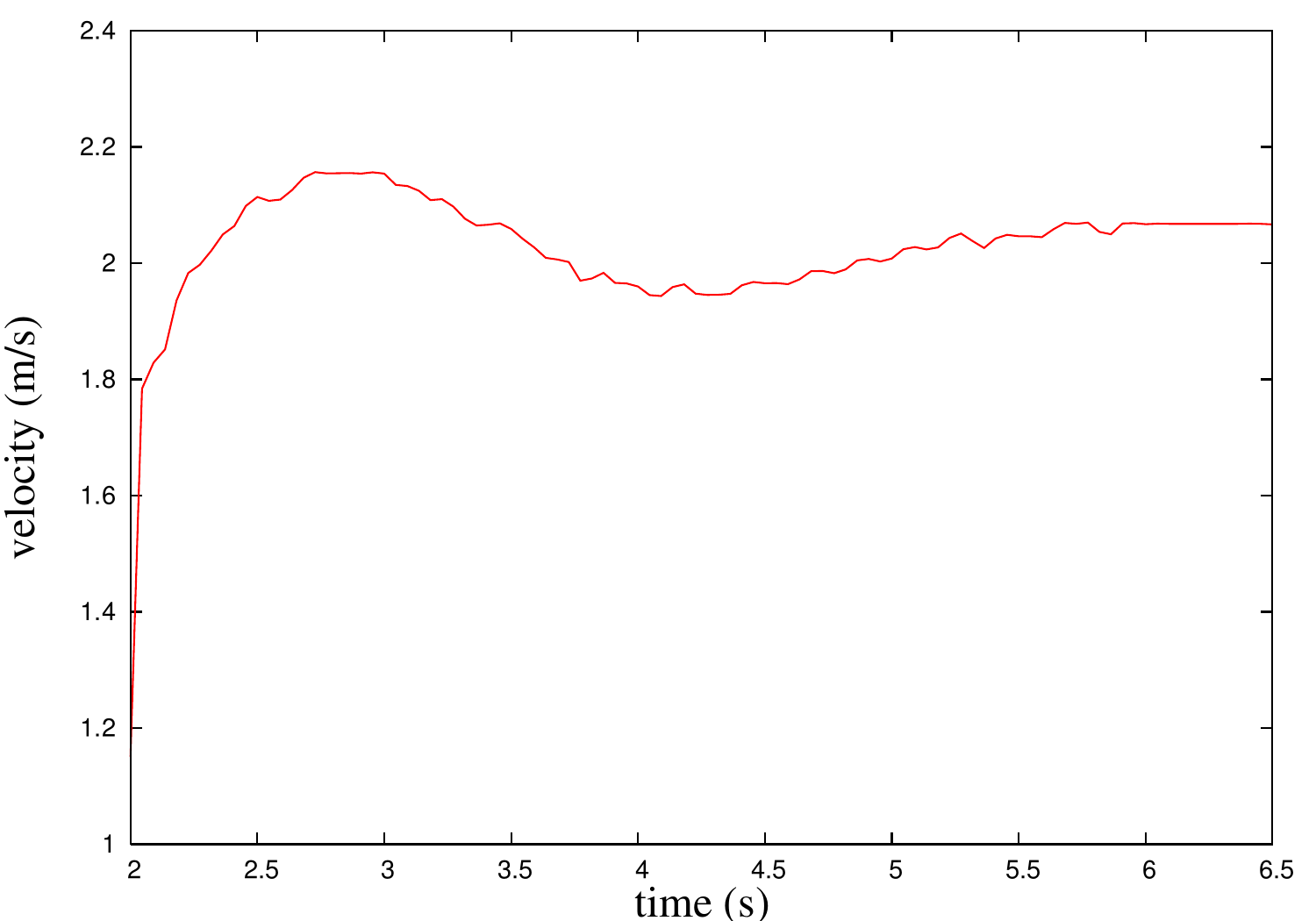}
}
\caption{Numerical results for  Wiggert's test.}
\end{center}
\end{figure}
%%%%%%%%%%%%%%%%%%%%%%%%%%%%%
\subsection{Numerical validation for a pressurized flow in an uniform pipe}\label{belier}
%%%%%%%%%%%%%%%%%%%%%%%%%%%%%

We present now numerical results of a water hammer test.
The pipe of circular cross-section of $2 \: m^2$ and thickness $20 \: cm$ is $2000\: m $ long. The altitude of the upstream end of the
pipe is $250 \: m$ and the angle is $5^{\circ}$. The Young modulus is $23 \, 10^9 \:Pa$ since the pipe
is supposed to be built in concrete: thus the sonic speed is equal to   $c=1414.2 \: m/s$. The total upstream head is $300 \:m$. \\
Other parameters are: 
$$
\begin{array}{lcl}
\hbox{Discretisation points}& : & 1000,\\ 
\hbox{Delta x }(m)& : &2,\\
\hbox{CFL }& : & 1,\\
\hbox{Simulation time } (s) & : &100,\\
\hbox{Sound speed } (ms^{-1})& : &1414.2.
\end{array}
$$

To ensure that the model and the kinetic numerical method that we propose describe precisely flows in closed uniform water pipes,
we present a validation of it by comparing numerical results of the proposed model with
the ones obtained by solving Allievi equations by the method of characteristics with the so-called \verb+belier+
code used by the engineers of Electricit\'e de France, Centre d'Ing\'enierie Hydraulique, Chamb\'ery, \cite{W93}.

A first simulation of the water hammer test is done for a fast cut-off of the downstream discharge for a pipe whose  
Strickler coeeficient is $K_{s} = 90$:
the initial downstream discharge is $10 \: m^3/s$ and we cut the flow in $5 \: s$. 
In figure \ref{belier1}, we present a comparison between the results
obtained by our kinetic scheme scheme and the ones obtained by the \verb+belier+ code at the middle of the pipe:
the behavior of the piezometric line  and the discharge at the middle of the pipe. One can observe that
the results for the proposed model and the numerical kinetic scheme are in very good agreement with the  solution of Allievi equations.

A second simulation of the water hammer test is done for the same  rapid cut-off of the downstream discharge but for a frictionless pipe
whose  Strickler coefficient is $K_{s} = 215,63 \: 10^{6}$.
In figure \ref{belier1bis}, we present a comparison between the results
obtained by our kinetic scheme scheme and the ones obtained by the \verb+belier+ code at the middle of the pipe:
the behavior of the piezometric line  and the discharge at the middle of the pipe. One can observe again that
the results for the proposed model and the numerical kinetic scheme are in very good agreement with the  solution of Allievi equations.
\begin{figure}[H]
\centering
\includegraphics[height=7cm,angle=-90]{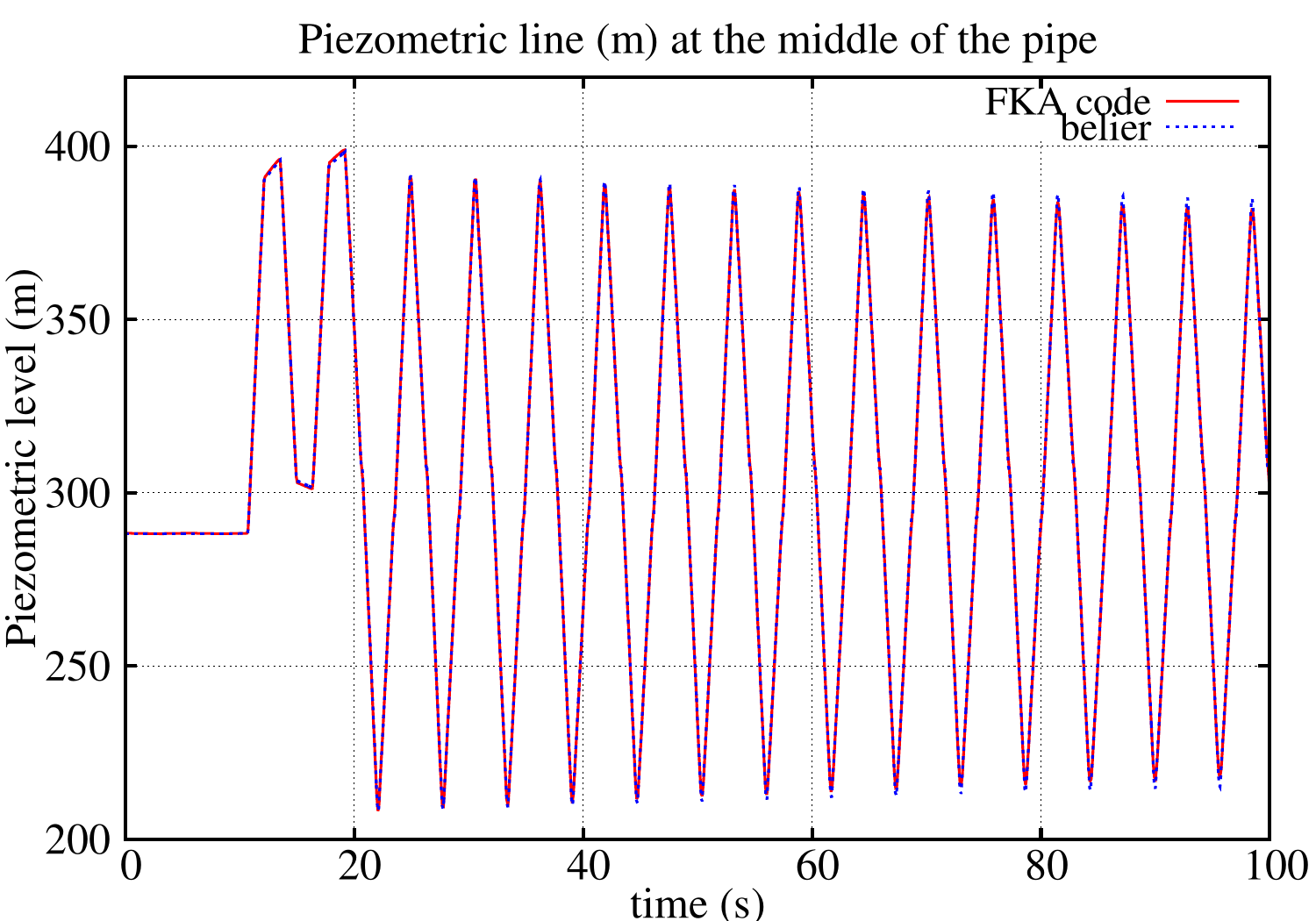}
\includegraphics[height=7cm,angle=-90]{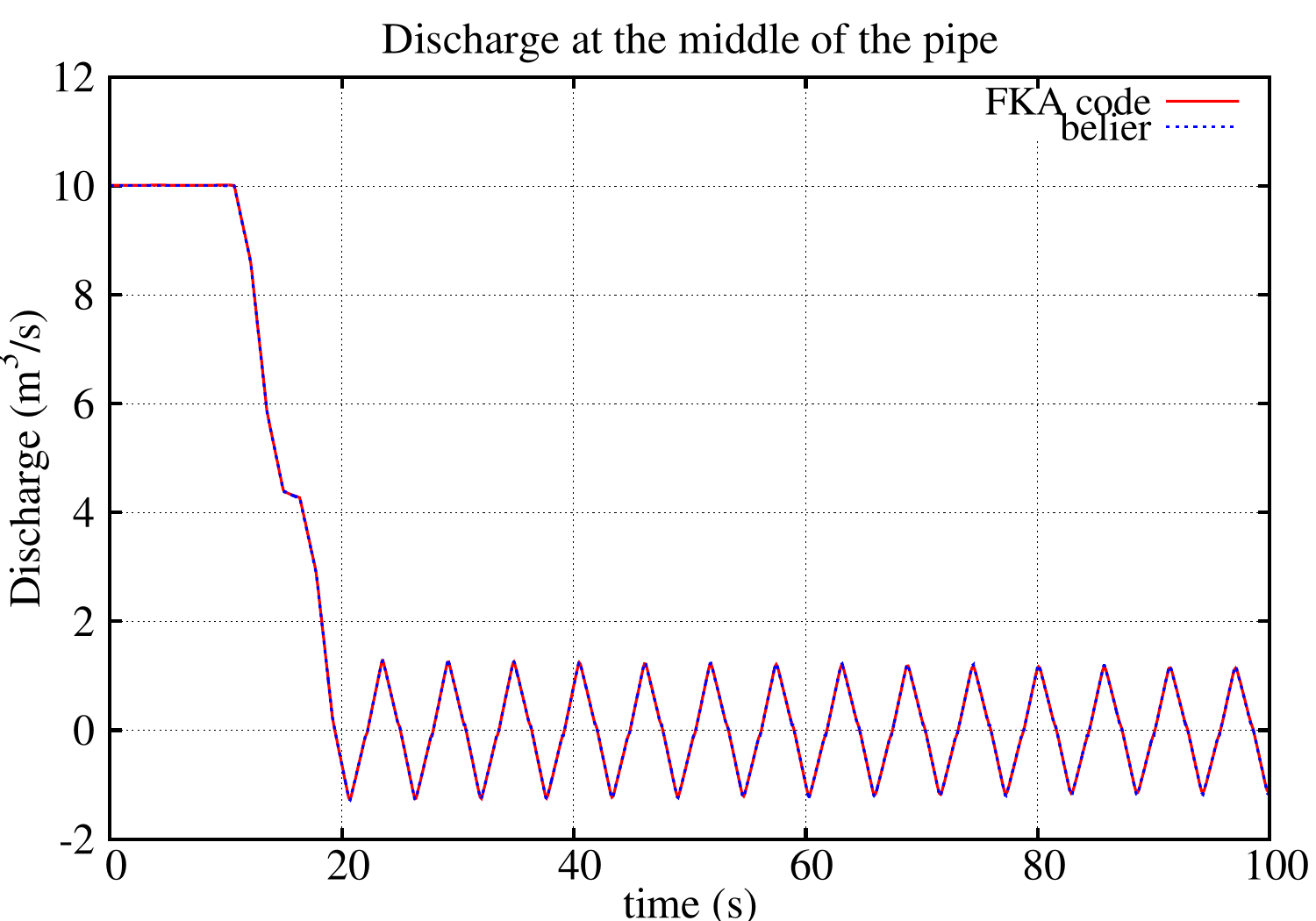}
\caption{Piezometric line (left) and discharge (right) at middle of the pipe\label{belier1}.}
\end{figure}
\begin{figure}[H]
\centering
\includegraphics[height=7cm,angle=-90]{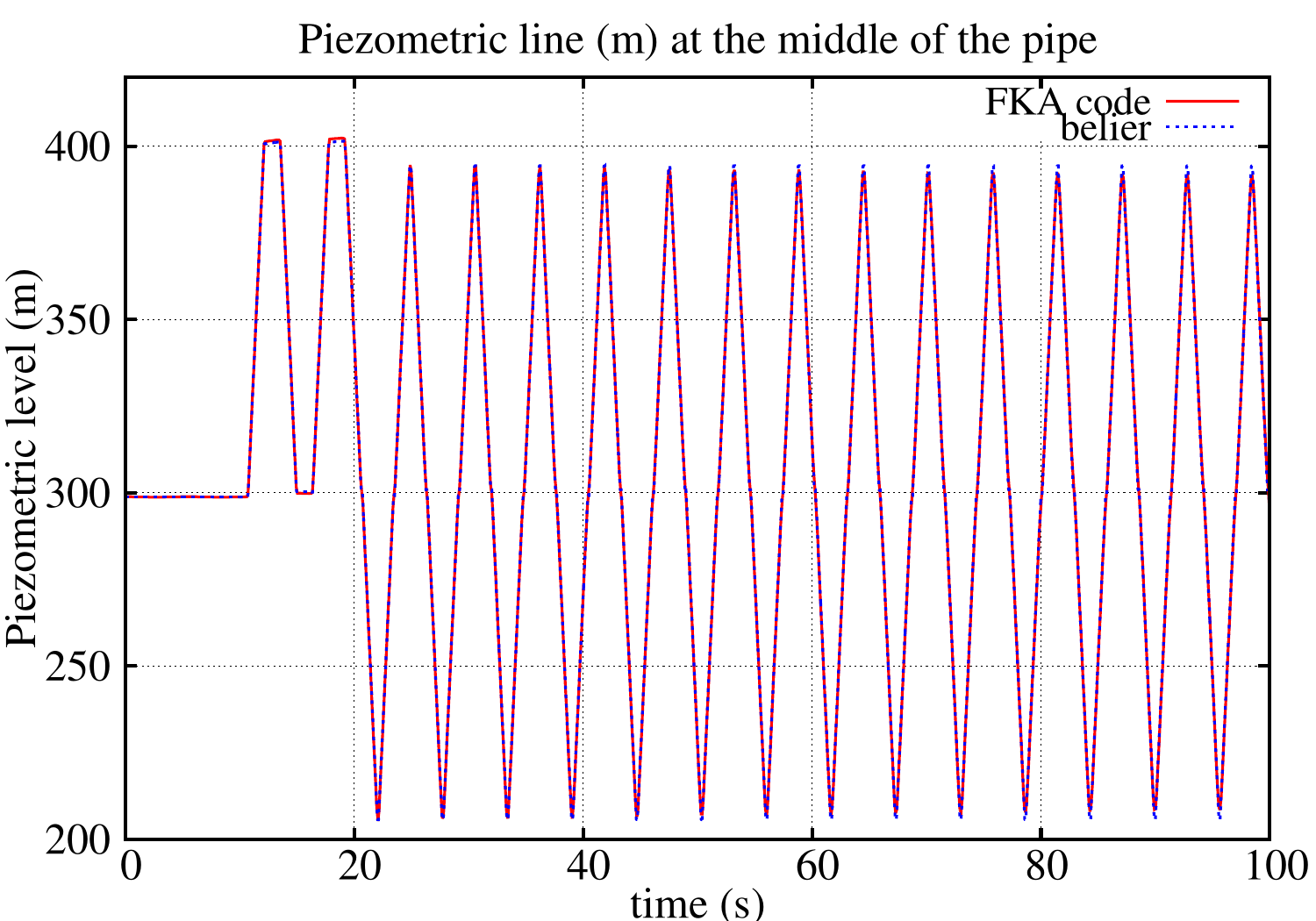}
\includegraphics[height=7cm,angle=-90]{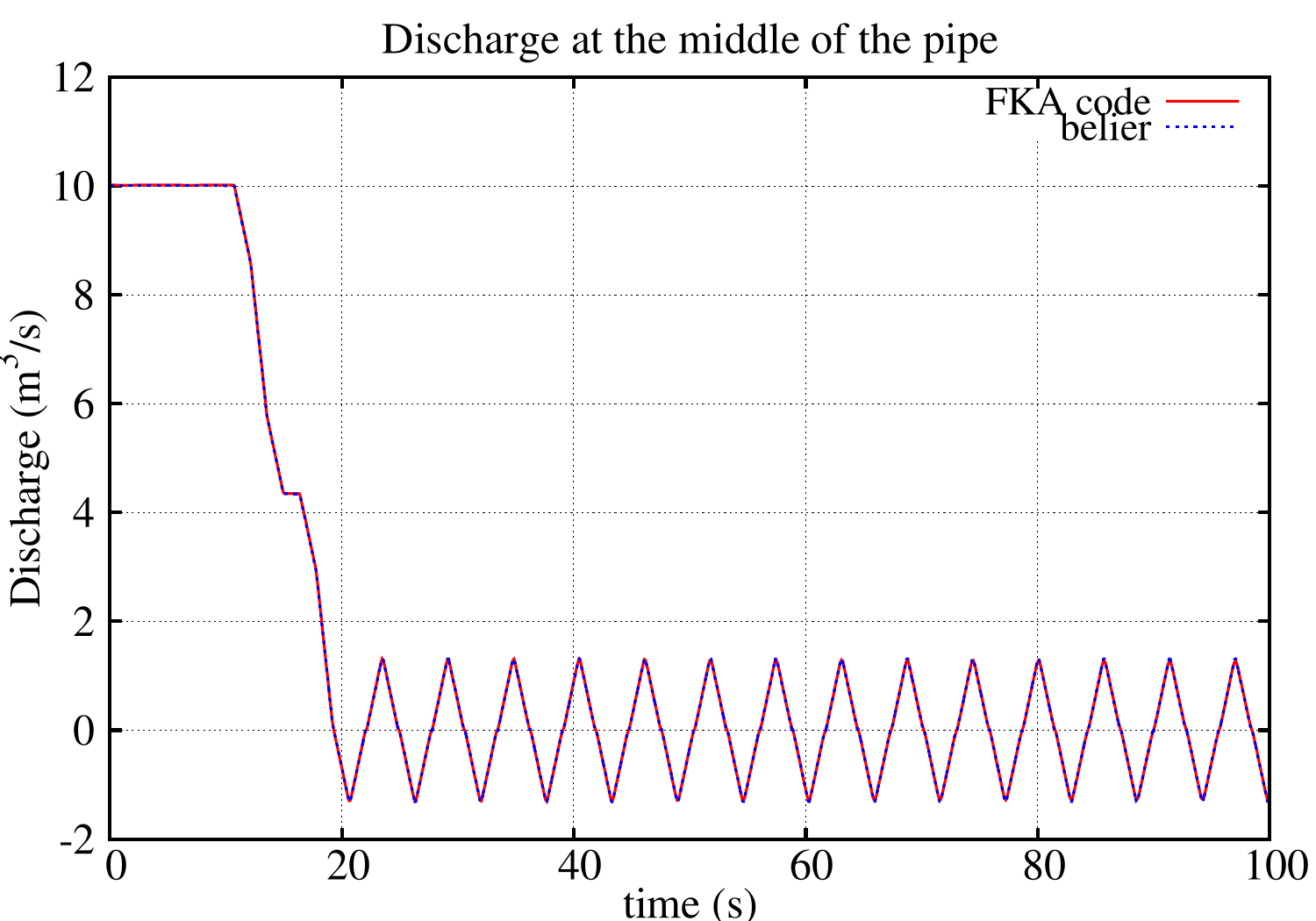}
\caption{Piezometric line (left) and discharge (right) at middle of the frictionless pipe\label{belier1bis}.}
\end{figure}
%%%%%%%%%%%%%%%%%%%%%%%%%%%%%%%%%%%%%%%%
\subsection{Numerical computation of  steady states}\label{steady}
%%%%%%%%%%%%%%%%%%%%%%%%%%%%%%%%%%%%%%%%
\subsubsection{Free surface steady state}
Our purpose in the following test cases is to study the convergence in time and space towards a 
free surface steady state in a varying rectangular channel.  
We first test the ability  of the presented scheme on two transcritical steady state test cases in a purely free surface flows  for which
analytic solutions are available.  For each test case, the numerical spatial order is computed. 

For all the free surface steady states, one has $\partial_t A=\partial_t u = \partial_t Q = 0 $. 
Thus the mass-conservation equation gives $Q = Q_{ex} = Q_{0}$ and we get from  Equation \eqref{ThmPFSEquationForU}: 
$$\partial_x Z = \partial_x\left( \mathcal{H}(A)\cos\theta + \frac{Q_0^2}{2 g A^2}\right) + K(x,A) \dsp\frac{Q_{0}|Q_{0}|}{A}.$$ 
Once the wet area $A$ and the discharge are given, one can compute the corresponding topography. Thus, for instance, following MacDonald \emph{et al. } \cite{Baines} in case of varying rectangular channel with: 
$$A(x) = B(x) h(x),\quad P(x,h) = B(x) + h\ ,$$
the bed level $Z$ is given by:
$$\partial_x Z = \left(\frac{Q_{0}^2}{g B(x)^2 h_{ex}(x)^3}-1\right) h'_{ex}(x)
-\frac{Q_{0}^2 n^2 (2 h_{ex}(x)+B(x))^{4/3}}{(B(x) h_{ex}(x))^{10/3}}+\frac{Q_{0}^2 B'(x)}{g B(x)^3 h_{ex}(x)^2}\ ,$$ 
where $h_{ex}$ is the given height, $B$ is the width of the channel, $P$ is the wet perimeter and $n$ is the Manning coefficient. 

In this configuration, we reproduce two transcritical test cases with an hydraulic jump: 
the first one is subcritical to supercritical and the second is  supercritical to subcritical. 
For each problem, we consider a channel of length $L=1000 \:m$ of varying width 
$$
B(x) = 10-64 \left(\left(x/L\right)^2-2\left(x/L\right)^3+\left(x/L\right)^4 \right).
$$ displayed on figure \ref{VR3B}
\begin{figure}[H]
\begin{center}
\includegraphics[height=5.5cm]{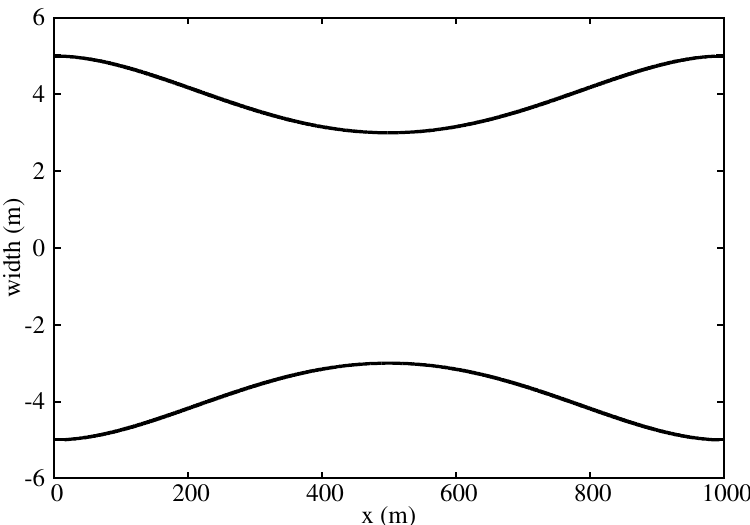}
\caption[Optional caption for list of figures]{Width profile.}
\label{VR3B}
\end{center}
\end{figure}
The discharge used  for steady state is $Q_{0}=20 \:m^3/s$ and $n= 0.02$. 
The analytical solution is given by 
\begin{equation}\label{hex}
h_{ex}(x)= \left\{
\begin{array}{lll}
-1/40+\frac{1}{1+2\left(\frac{2x-L}{2L}\right)^2} & \textrm{ if } & 0\leqslant x \leqslant 500,\\
h_r(x)  & \textrm{ if } & 500< x \leqslant 1000,\\
\end{array}
\right.
\end{equation}
where $h_r(x)= \displaystyle \sum_{i=1}^3 a_i \exp\left(-30i \frac{2x-L}{2L}\right)+a_4 \exp\left(\frac{4x-L}{4L}\right)$ is the solution on the right hand side of the hydraulic jump and the coefficients $a_i$ are given following the test problem.
\begin{rque}
As emphasized in \cite{Baines}, there is no analytical expression of $Z$. 
Therefore, we have constructed $Z$ using cubic spline interpolation.
\end{rque}

\paragraph{Subcritical to supercritical test case .\newline}
In this test problem, the analytical solution is subcritical at inflow and changes, via a hydraulic jump localized at $x=500 \:m$, to supercritical.  
The analytical solution is given by formula \eqref{hex} where :
$$a_1 = 0.769035 \,,\, a_2 = -0.755596 \,,\, a_3=0.106813 \mbox{ and }a_4=1.125000 \ .$$  
The analytical solution as well as the bed profil are shown in figure \ref{VR3Piezo}. The height at the upstream boundary is $0.641667 \:m$ and the height at downstream end is $1.125\:m$.

To compute the convergence of the numerical solution toward the steady state solution, the upstream total head is kept constant equal to $0.641667 \: m$ and the downstream water level to $1.125\:m$.
Starting from a still water steady state with $h_{x=0} = 0.641667 $, we compute the  numerical flow for all time from $t=0 \:s$ until the stationary state is reached. Then,  the $L^1$ norm of the difference between $(h,Q)$ computed by the numerical kinetic scheme, for different mesh sizes of the uniform discretisation $\Delta x$, and the analytic solution $(h_{ex}, Q_{ex})$ at final time. 

Other parameters are: 
$$
\begin{array}{lcl}
\hbox{CFL }& : & 0.95,\\
\hbox{Simulation time } (s) & : &5000.\\
\Delta x = L/N& : & N=\left(100+m 100\right)_{m = 0,19} \hbox{ and } 20000.
\end{array}
$$
We present, in figure \ref{VR3}, the piezometric line (see figure \ref{VR3Piezo}) and the discharge (see figure \ref{VR3Discharge}) of the flow along the pipe when the steady state is reached for four different mesh sizes $\Delta x = 10$, $\Delta x = 5$, $\Delta x = 1$ and $\Delta x = 0.05$. 
In figure \ref{VR3Piezo} and \ref{VR3PiezoZoom}, the four curves representing the piezometric line, compared to the analytical one,   
are  close to the analytical solution  and the hydraulic jump location are very well captured even if, for large $\Delta x$,
the numerical solution is smooth around this point. 
For the discharge, we observe that the convergence toward $Q_{0} = 20 \:m^3/s$ is also close, since for $\Delta x = 100$, 
the error is of order 0.1 and decreases as $\Delta x$.  
Indeed, in figure \ref{VR3Norm}, we have computed the $L^1$ norm. The obtained numerical order is almost  equal to 1 in $h$ as well as for $Q$.
%%%%%%%%%%%%%%%%%%%
\begin{figure}[H]
\begin{center}
\subfigure[Piezometric level.\label{VR3Piezo}]
{
\includegraphics[height=5.5cm]{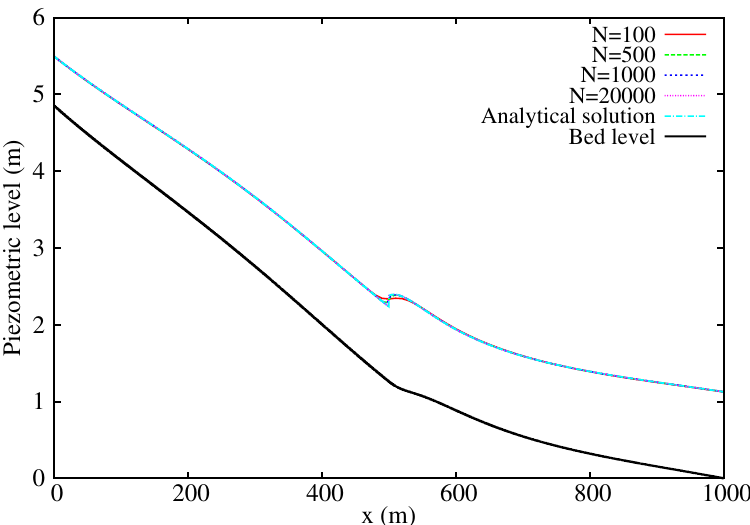}
}
\subfigure[Around the hydraulic jump.\label{VR3PiezoZoom}]
{
\includegraphics[height=5.5cm]{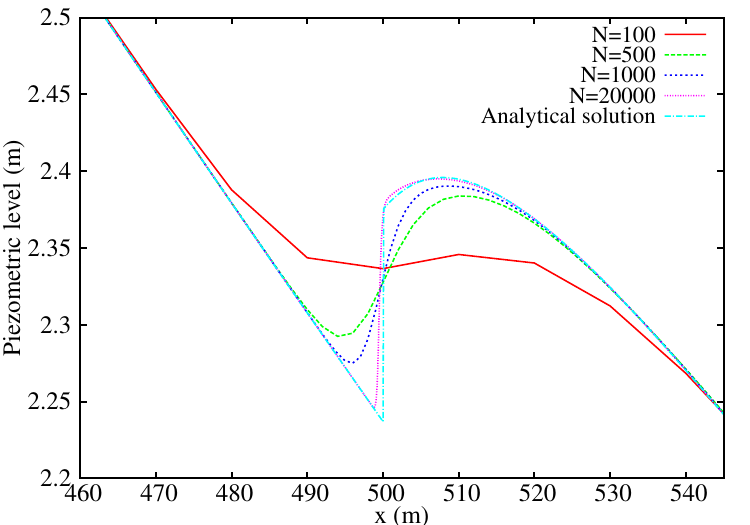}
}
\subfigure[Discharge.\label{VR3Discharge}]
{
\includegraphics[height=5.5cm]{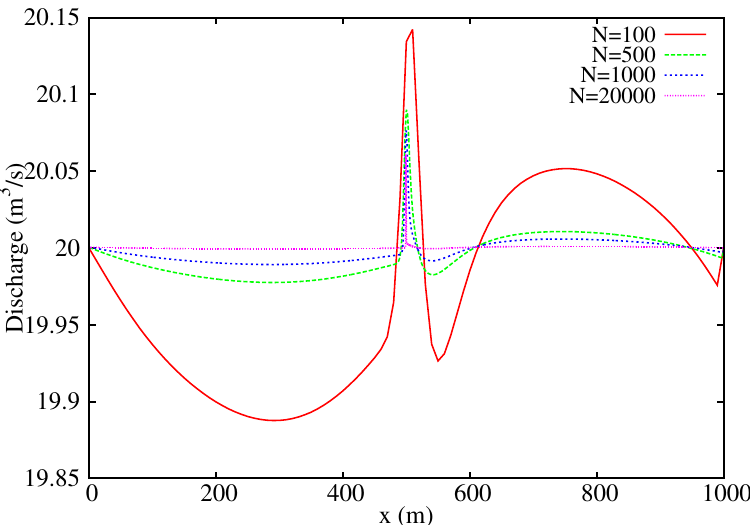}
}
\subfigure[$L^1$ discretisation error.\label{VR3Norm}]
{
\includegraphics[height=5.5cm]{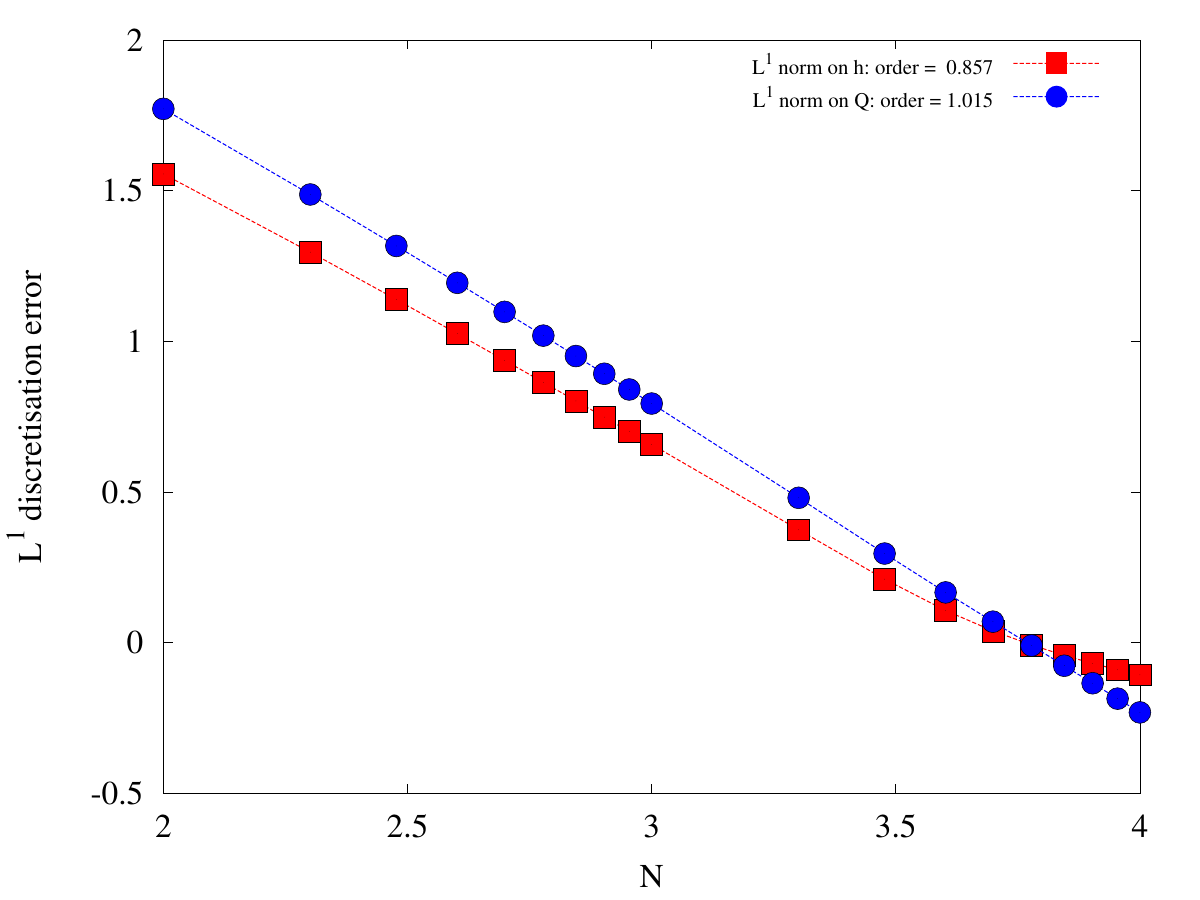}
}
\caption[Optional caption for list of figures]{Subcritical to supercritical test case.}
\label{VR3}
\end{center}
\end{figure}

\paragraph{Supercritical to subcritical test case.\newline}
In this test problem, the analytical solution is supercritical at inflow and changes, via a hydraulic jump localized at $x=500 \:m$, to subcritical.  
The analytical solution here is given by formula \eqref{hex} where 
$$a_1=-0.230680 \,,\, a_2=0.248267 \,,\, a_3=-0.228271 \mbox{ and } a_4=1.500000 \ .$$
The analytical solution as well as the bed profil are shown in figure \ref{VR4Piezo}. The height at the upstream boundary is $0.641667 \:m$ and the height at downstream end is $1.5\:m$. 

We proceed as done before to compute the convergence of  the numerical solution  toward the steady state.  Results on the piezometric line, the discharge and the numerical order are displayed on figure \ref{VR4}. We have used the same parameters as in the previous section. 
As one can observe, the same conclusion holds.
%%%%%%%%%%%%%%%%%%%
\begin{figure}[H]
\begin{center}
\subfigure[Piezometric level.\label{VR4Piezo}]
{
\includegraphics[height=5.5cm]{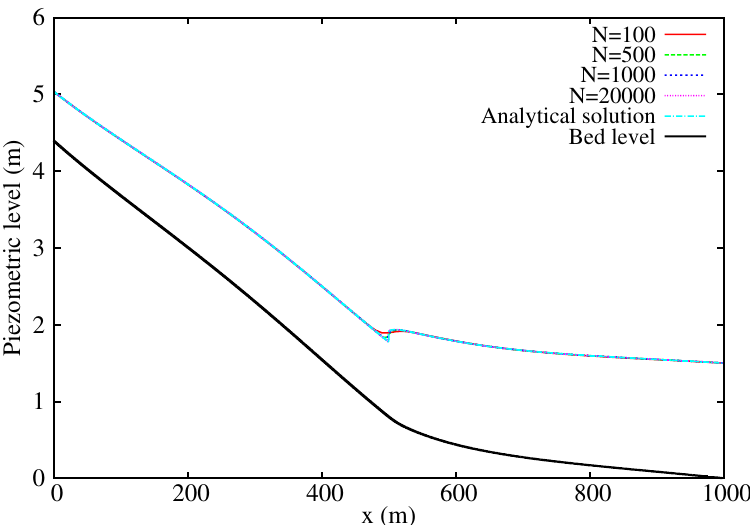}
}
\subfigure[Around the hydraulic jump.\label{VR4PiezoZoom}]
{
\includegraphics[height=5.5cm]{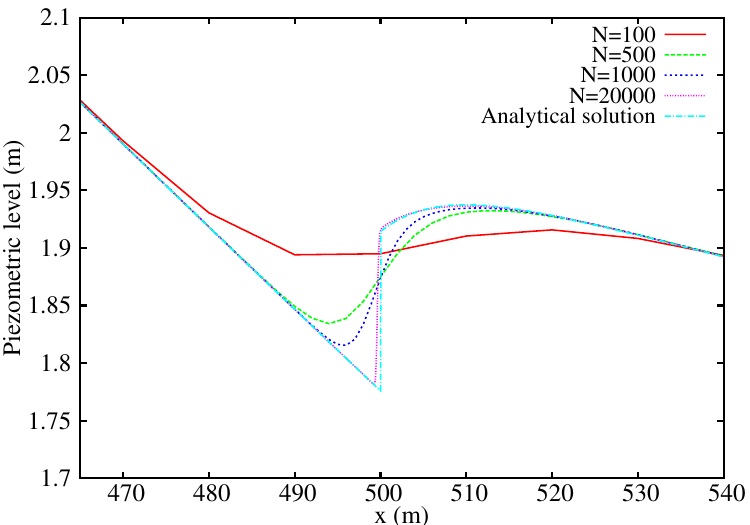}
}
\subfigure[Discharge.\label{VR4Discharge}]
{
\includegraphics[height=5.5cm]{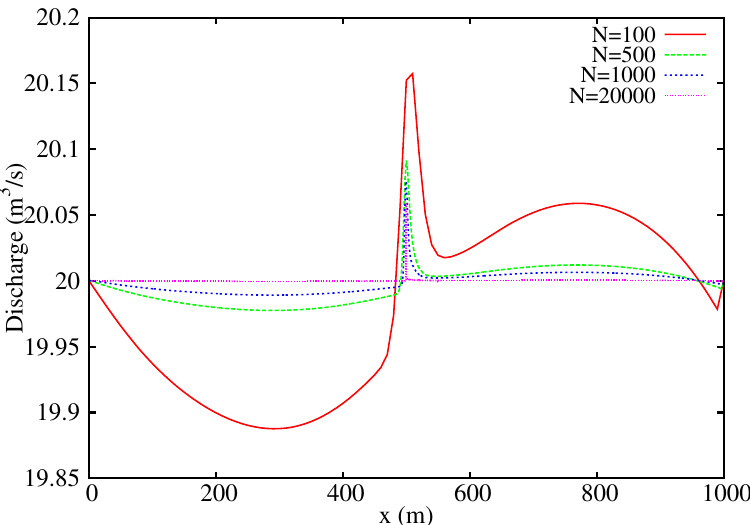}
}
\subfigure[$L^1$ discretisation error.\label{VR4Norm}]
{
\includegraphics[height=5.5cm]{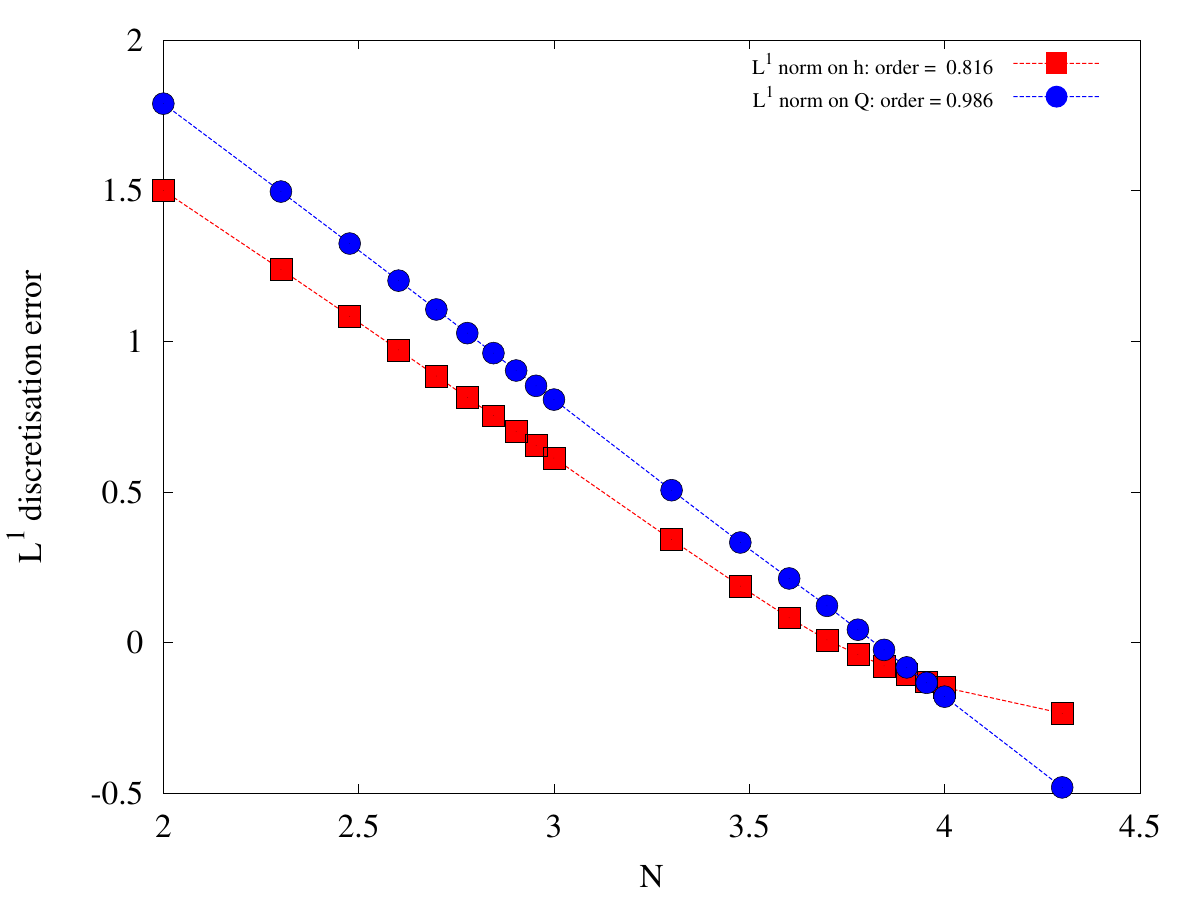}
}
\caption[Optional caption for list of figures]{Supercrtitical to subcritical test case.}
\label{VR4}
\end{center}
\end{figure}

\subsubsection{Mixed steady state}
In order to obtain a qualitative behavior of the scheme and to compute a ``numerical'' order of the discretisation error of the kinetic numerical scheme,
we present now a numerical experiment where the steady state is mixed. 
The pipe is a circular pipe of diameter $3 \:m$ and  $100 \: m$ long with slope 0.001 and Strickler coefficient is $K_{s} = 63.7$.
The altitude of the upstream end of the pipe is $100 \: m$.
The upstream total head is kept constant equal to $104 \: m$ whereas the downstream water level varies (see figure \ref{ordreaval}). 
We have compute the ``exact'' numerical flow for all time from $t=0 \:s$ until the stationary state is reached at time $t= 100 \:s$, 
by the VFRoe method presented in \cite{BEG09} with a uniform  discretisation  of 8000  mesh points.
We have then computed the $L^1$ norm of the difference between the piezometric line computed by the numerical kinetic scheme for different mesh 
sizes of the uniform discretisation, $\Delta x$, and the ``exact'' numerical solution at time $t = 20  \:s$ and $t= 100 \:s$. 

Other parameters are: 
$$
\begin{array}{lcl}
\hbox{CFL }& : & 0.9,\\
\hbox{Simulation time } (s) & : &100,\\
\hbox{Sound speed } (ms^{-1})& : &40.
\end{array}
$$
We present, in figure \ref{t1}, the piezometric line and the speed of the flow along the pipe at time $t = 20  \:s$ for three different mesh sizes
(in fact we prefer to talk of the number of mesh points). The three curves representing the piezometric line are very close 
{\it whereas} the coarse mesh  does not capture at all  the speed along the pipe.

We present, in figure \ref{t2}, the piezometric line and the speed of the flow along the pipe at time $t = 100  \:s$. 
The three curves representing the piezometric line {\it as well as}  the speed  along the pipe are very close. 
One can see that the stationary speed is not constant along the pipe.

The numerical order at time $t = 20  \:s$, represented in figure \ref{ordret1} for different mesh sizes, 
and the numerical order at time $t = 100  \:s$, represented in figure \ref{ordret2}, are almost equal to 1,
which was expected since a kinetic finite volume scheme is known to be of order 1.
% , see \cite{PS01}.
\begin{figure}[H]
\centering
\includegraphics[height = 8cm]{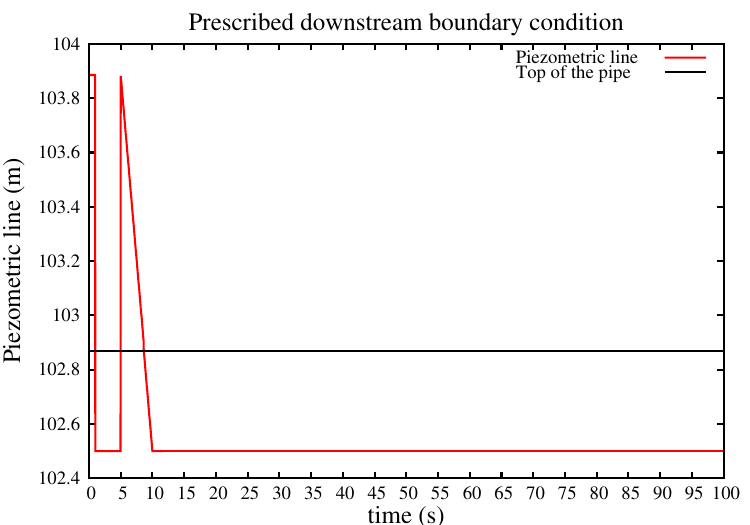}
\caption{Piezometric line  at the downstream end of the pipe.}\label{ordreaval}
\end{figure}

%%%%%%%%%%%%%%%%%%%
% t = 20 s
%%%%%%%%%%%%%%%%%%%
\begin{figure}[H]
\begin{center}
\subfigure[Piezometric line.]
{
\includegraphics[height=5.5cm]{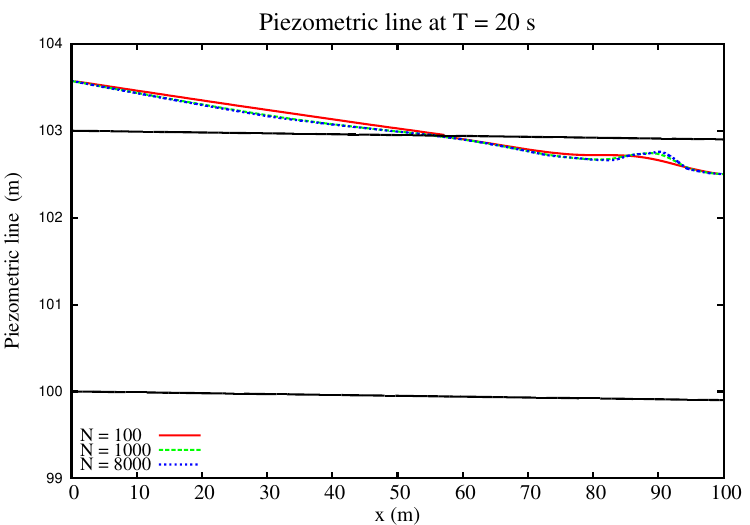}
}
\subfigure[Speed along the pipe.]
{
\includegraphics[height=5.5cm]{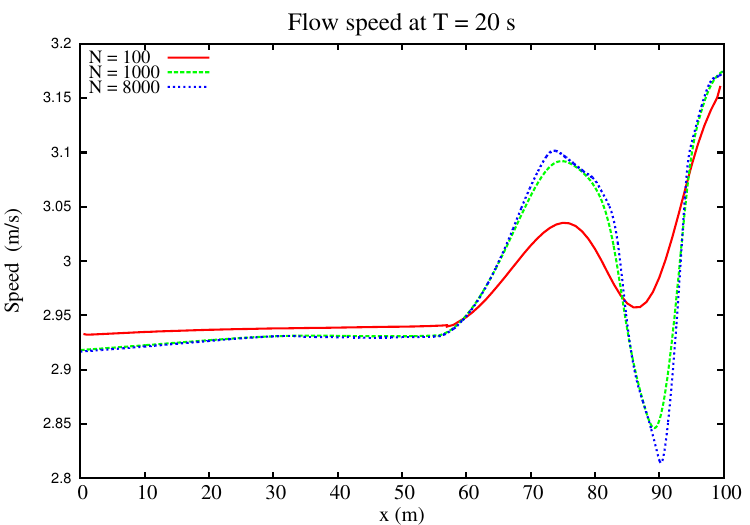}
}
\caption[Optional caption for list of figures]{Piezometric line and speed along the pipe at time $t = 20  \:s$. }
\label{t1}
\end{center}
\end{figure}

%%%%%%%%%%%%%%%%%%%
% t = 100 s
%%%%%%%%%%%%%%%%%%%
\begin{figure}[H]
\begin{center}
\subfigure[ Piezometric line.]
{
\includegraphics[height=5.5cm]{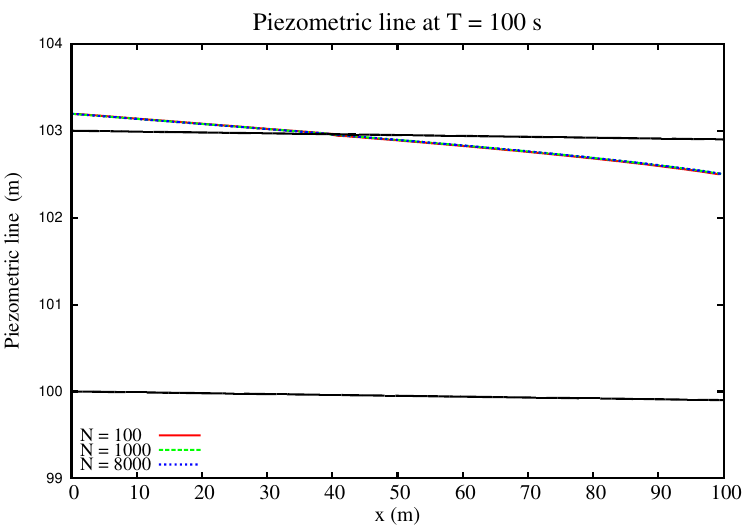}
}
\subfigure[Speed along the pipe.]
{
\includegraphics[height=5.5cm]{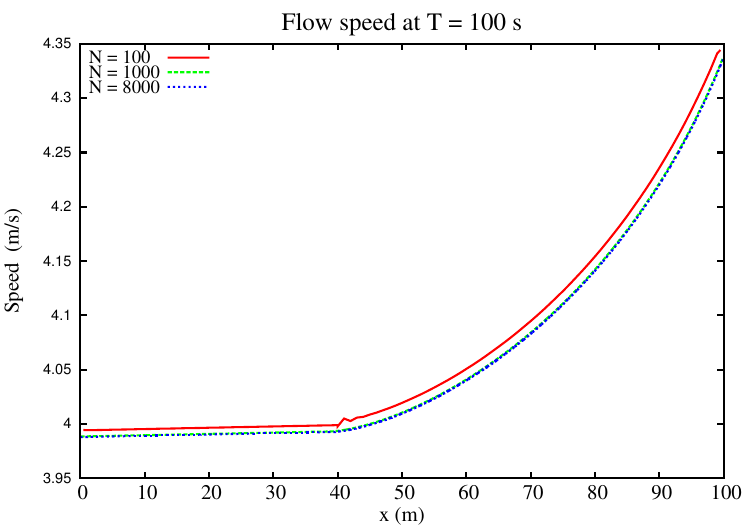}
}
\caption[Optional caption for list of figures]{Piezometric line and speed along the pipe at time $t = 100  \:s$. }
\label{t2}
\end{center}
\end{figure}

%%%%%%%%%%%%%%%%%%%
% Ordre
%%%%%%%%%%%%%%%%%%%
\begin{figure}[H]
\begin{center}
\subfigure[  $t = 20  \:  s$.\label{ordret1}]
{
\includegraphics[height=6cm]{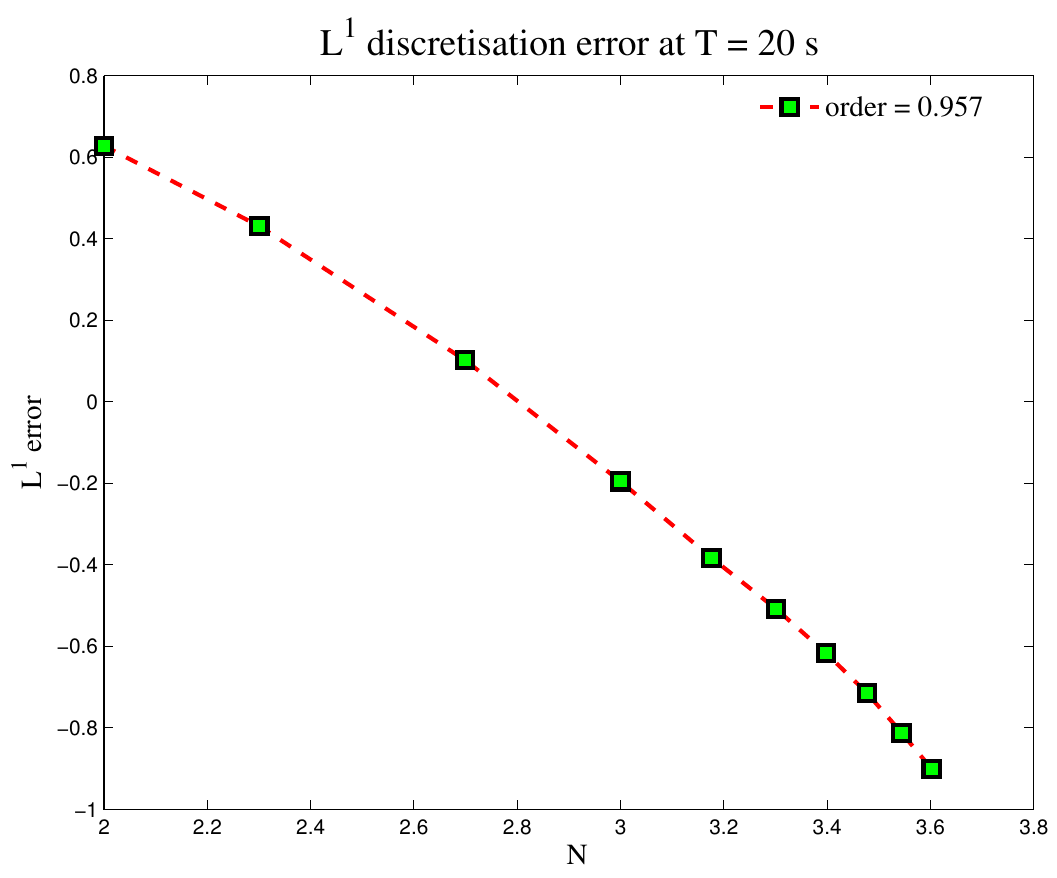}
}
\subfigure[ $t = 100 \: s$.\label{ordret2}]
{
\includegraphics[height=6cm]{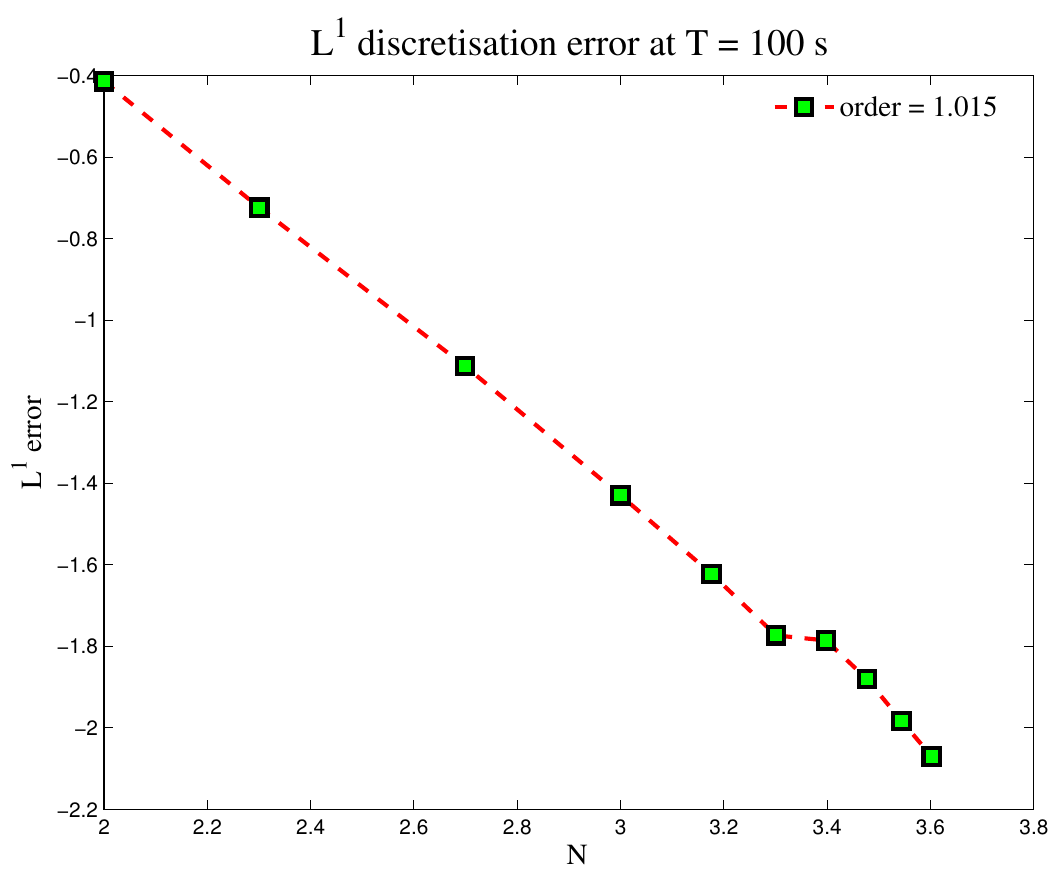}
}
\caption[Optional caption for list of figures]{$L^{1}$ discretisation error versus $N$ in $10$-logarithmic scale at time 
$t=20  \:s$ and time $t=100  \:s$. }
\label{Testordre}
\end{center}
\end{figure}

Finally, although we do not know if the two numerical schemes that we proposed satisfy the conservative in cell entropy (see Equation \eqref{ThmPFSEntropy}),
every numerical results presented have a very good qualitative behavior.
%%%%%%%%%%%%%%%%%%%%%%%%%%%%%
\subsection{Numerical validation for drying and flooding flow}\label{dry}
%%%%%%%%%%%%%%%%%%%%%%%%%%%%%
We present now numerical results for a flow that will be drying and flooding.
The frictionless pipe is constituted by a pipe of circular cross-section of diameter $2 \:m$ and  $50 \: m$ long with slope 0.003 and another pipe
of circular cross-section of diameter $2 \:m$  and  $100 \: m$ long with slope 0.05 . The altitude of the upstream end of the
pipe is $100\:  m$.
The upstream and downstream discharge is kept to 0. \\
Other parameters are: 
$$
\begin{array}{lcl}
\hbox{First pipe discretisation points}& : & 100,\\ 
\hbox{Delta x }(m)& : &0.5,\\
\hbox{Second pipe discretisation points}& : & 200,\\ 
\hbox{Delta x }(m)& : &0.5,\\
\hbox{CFL }& : & 0.9,\\
\hbox{Simulation time } (s) & : &500,\\
\hbox{Sound speed } (ms^{-1})& : &10.
\end{array}
$$
The initial state is a flow of constant height ($1.8 \: m$) on half the first pipe and a dry zone on the rest of the pipe, see figure \ref{flaque0}.
We present the flow at time $T= 6 \: s$, see figure \ref{flaque1}, where a drying zone is present, 
at time $T= 80 \: s$, see figure \ref{flaque2}, when the flow has reached the downstream end and is
partially pressurized and the flow at the final time $T = 500 \:  s$, see figure \ref{flaque3} where all the water is in the second pipe.
This non physical test shows that the kinetic numerical scheme treats ``naturally'' the flooding zone and almost the drying zone (up the
rounding error of the computer). The water height is exactly equal to $0$,
in the initial condition and at the final time for the dry zones.
\begin{figure}[H]
\centering
\includegraphics[height=7cm,angle=-90]{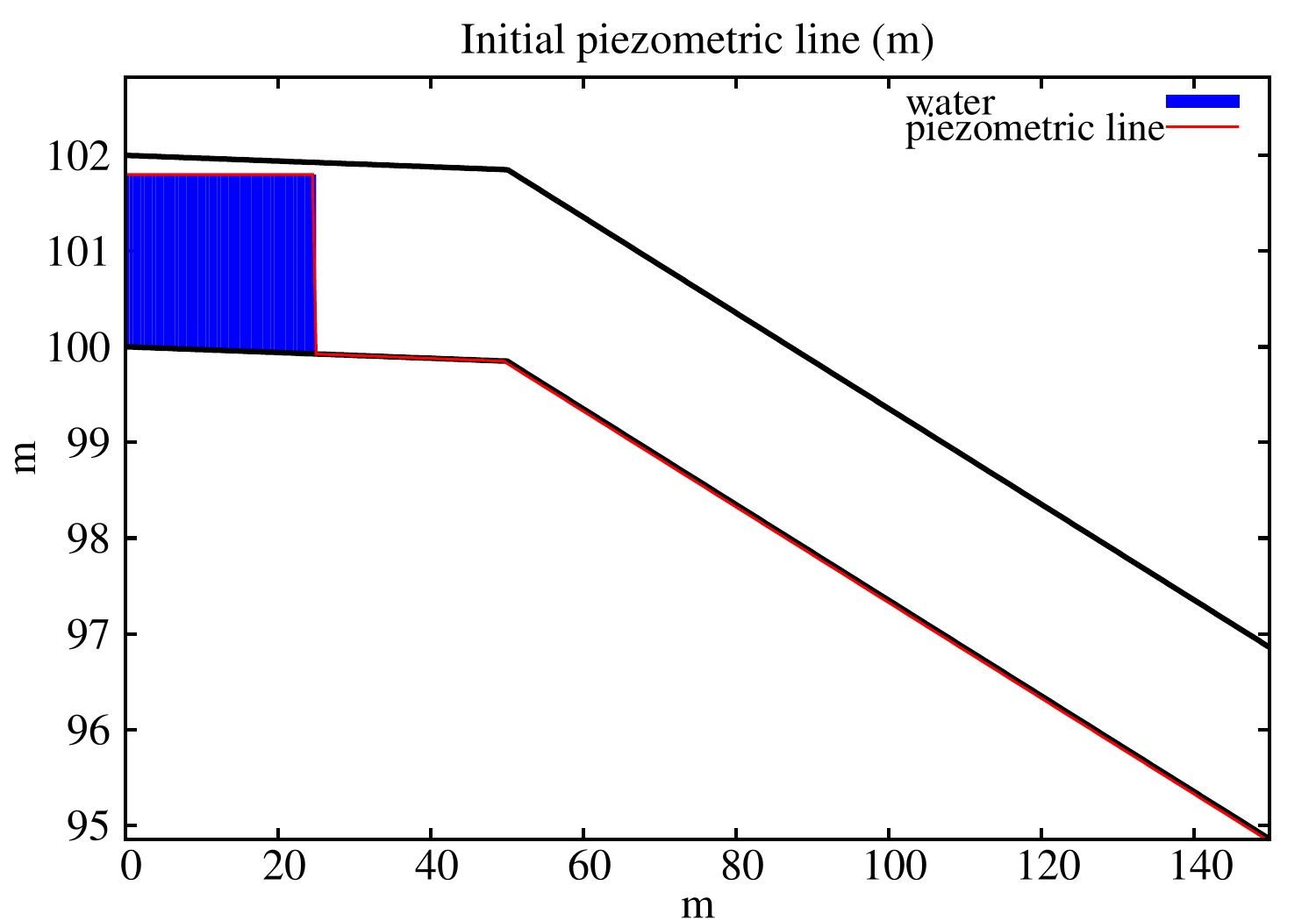}
\includegraphics[height=7cm,angle=-90]{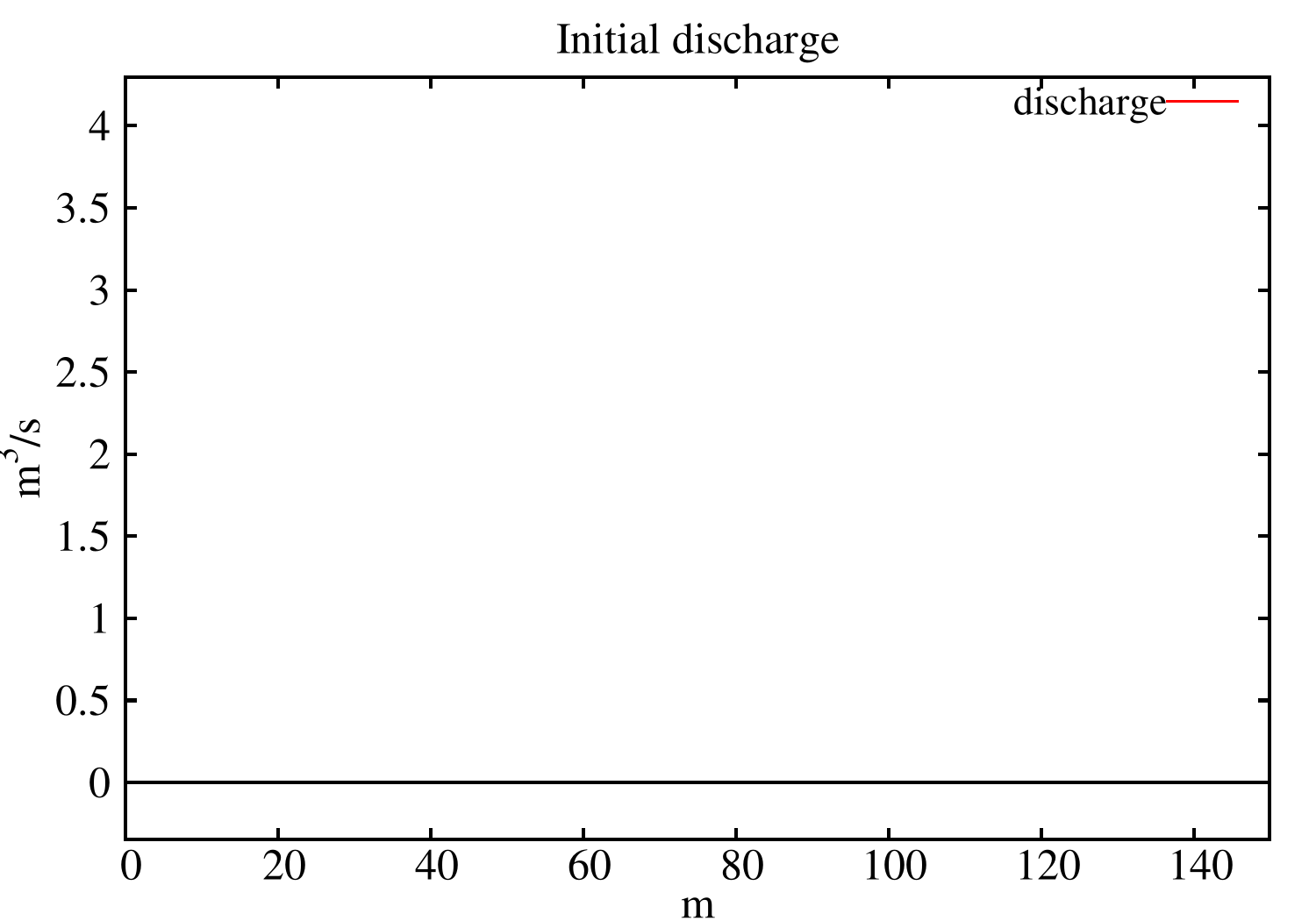}
\caption{Piezometric line (left) and discharge (right) at initial condition\label{flaque0}.}
\end{figure}

\begin{figure}[H]
\centering
\includegraphics[height=7cm,angle=-90]{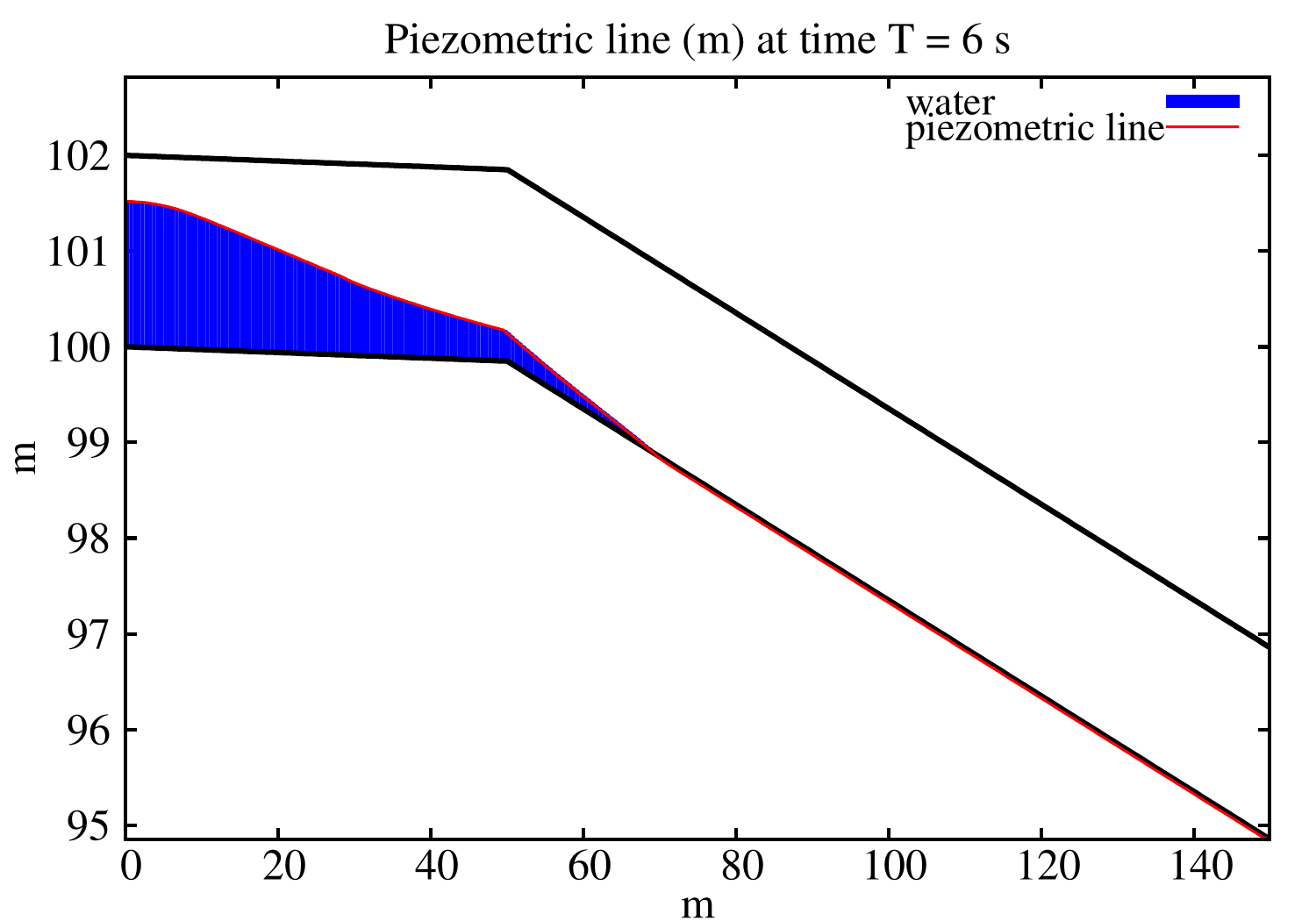}
\includegraphics[height=7cm,angle=-90]{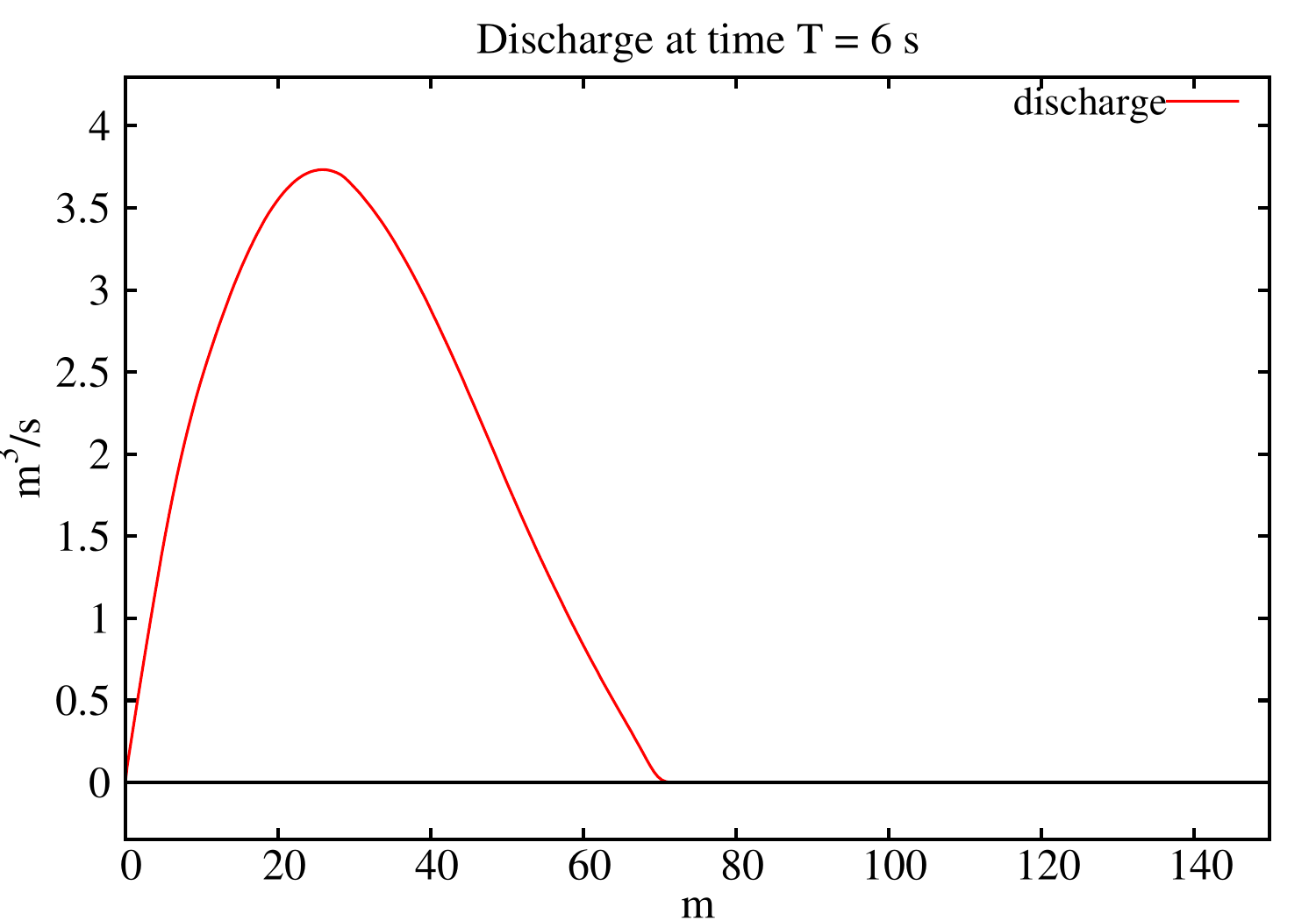}
\caption{Piezometric line (left) and discharge (right) at time $T = 6 \: s$\label{flaque1}.}
\end{figure}

\begin{figure}[H]
\centering
\includegraphics[height=7cm,angle=-90]{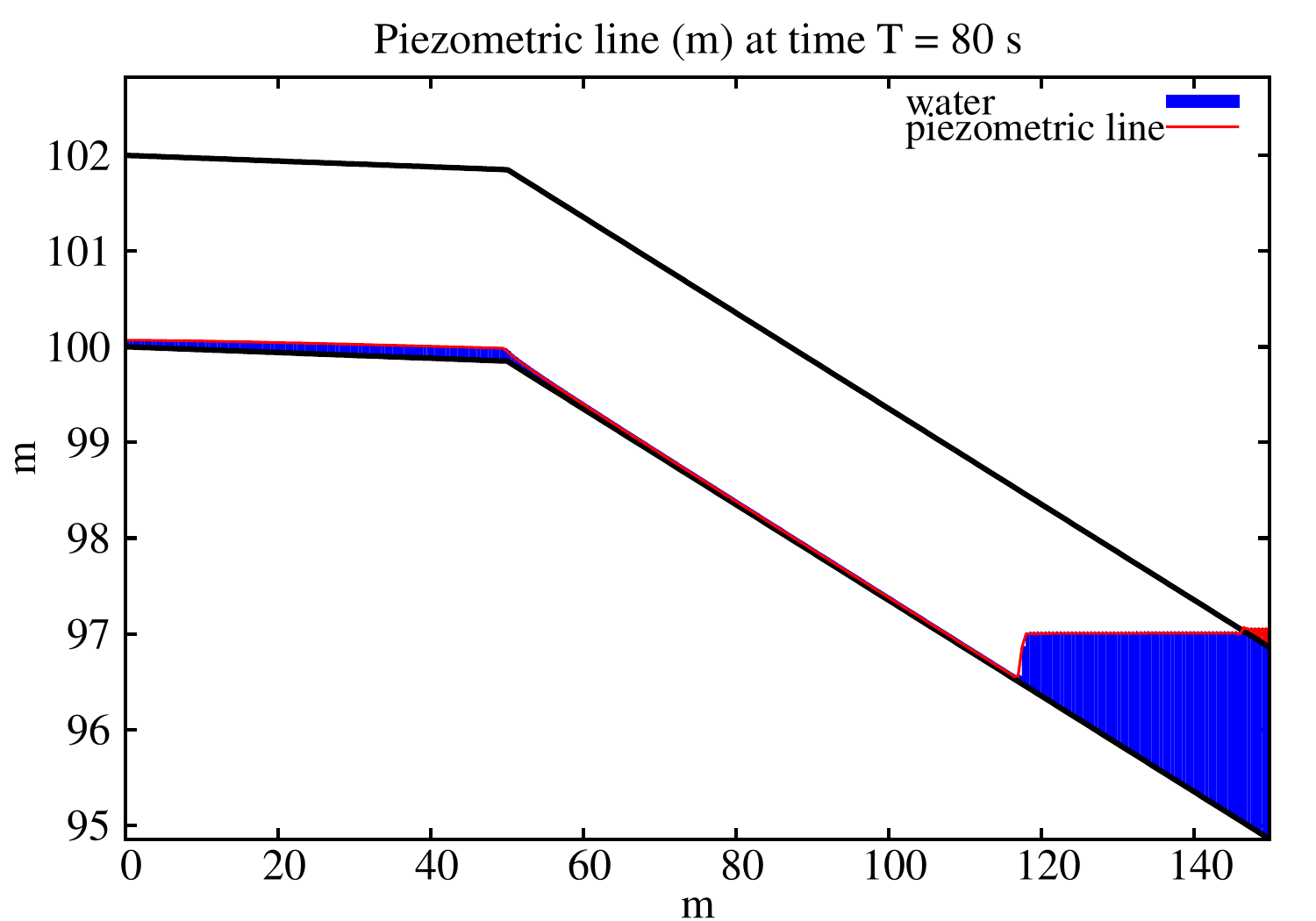}
\includegraphics[height=7cm,angle=-90]{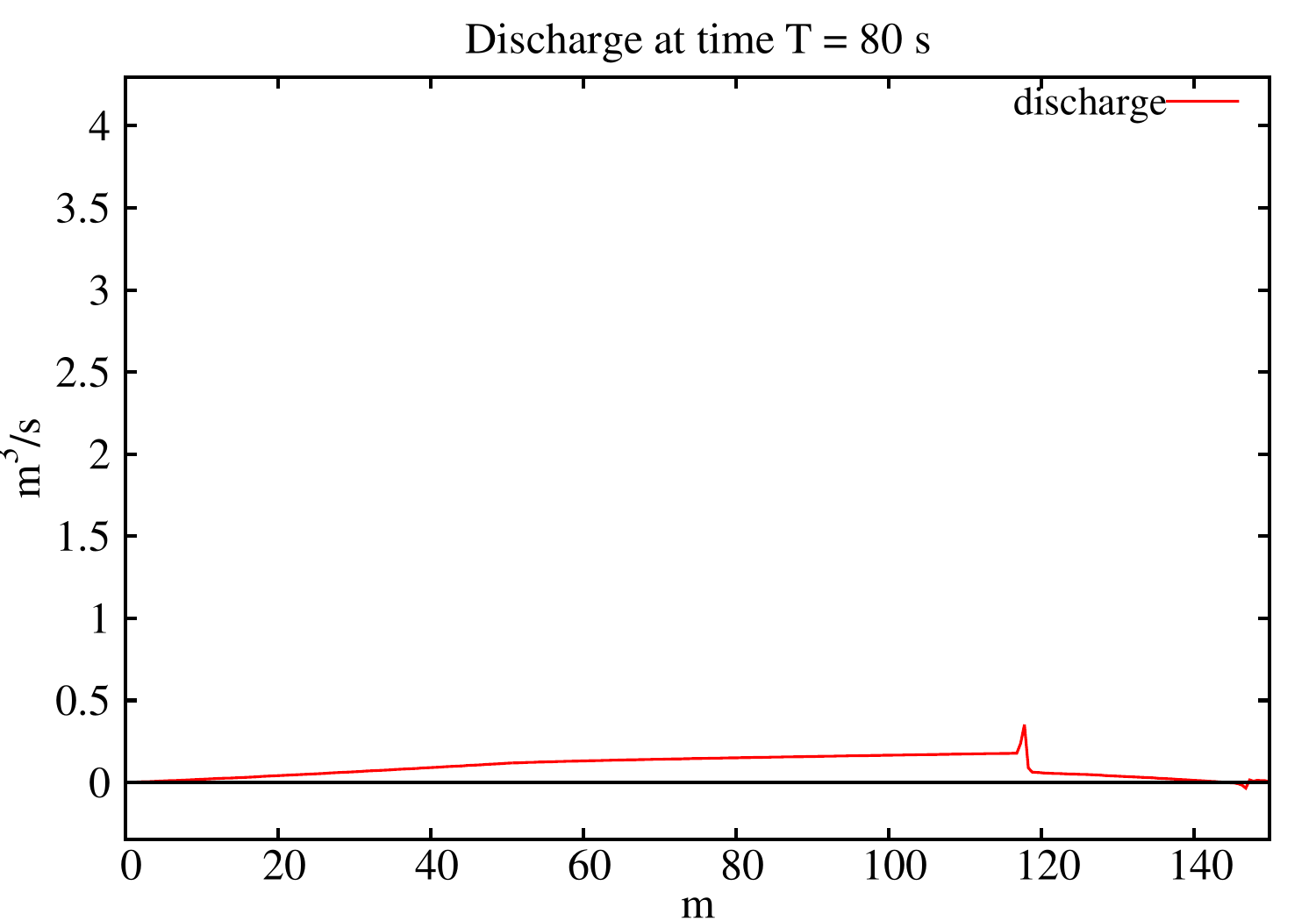}
\caption{Piezometric line (left) and discharge (right) at time $T = 80 \: s$\label{flaque2}.}
\end{figure}

\begin{figure}[H]
\centering
\includegraphics[height=7cm,angle=-90]{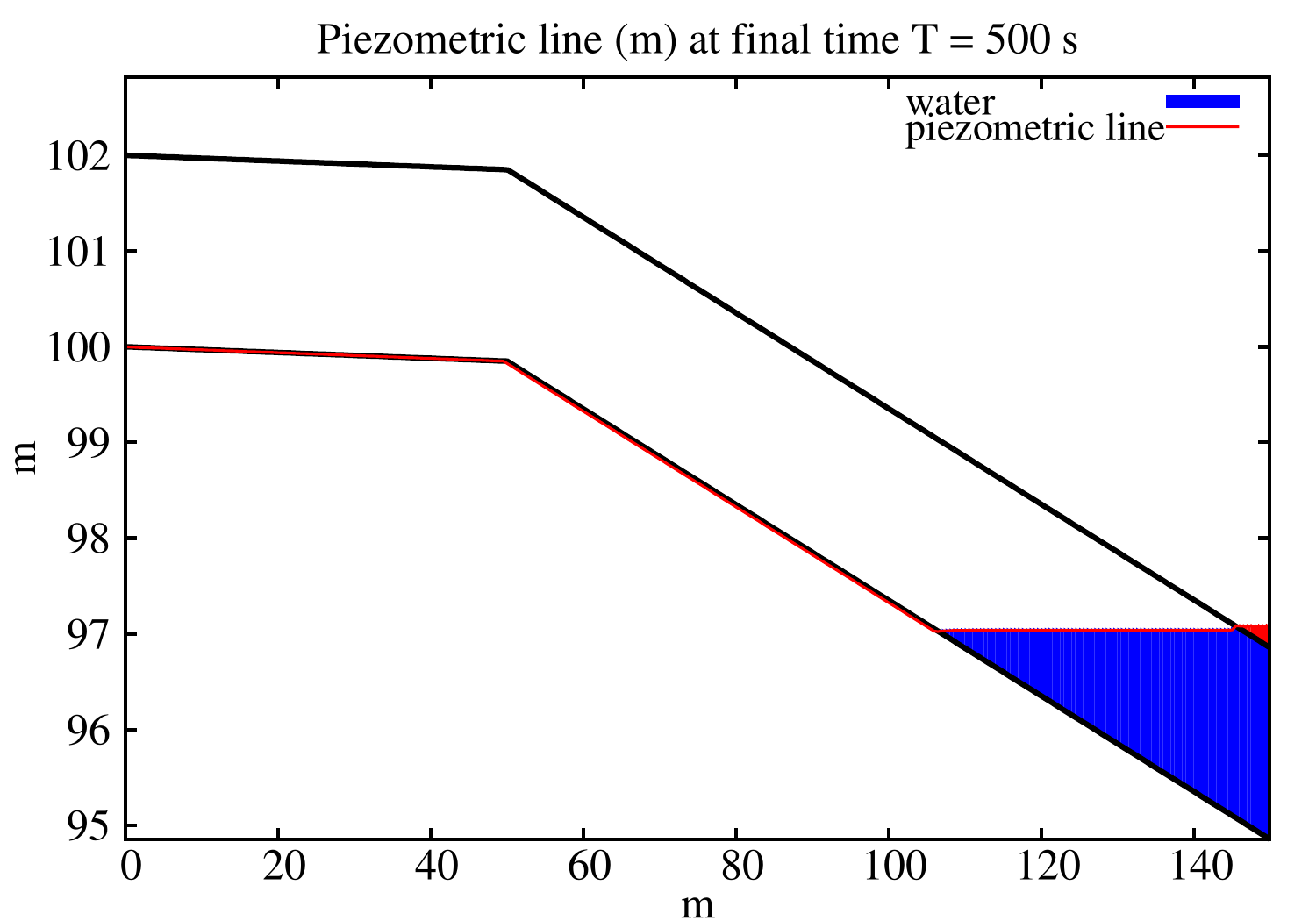}
\includegraphics[height=7cm,angle=-90]{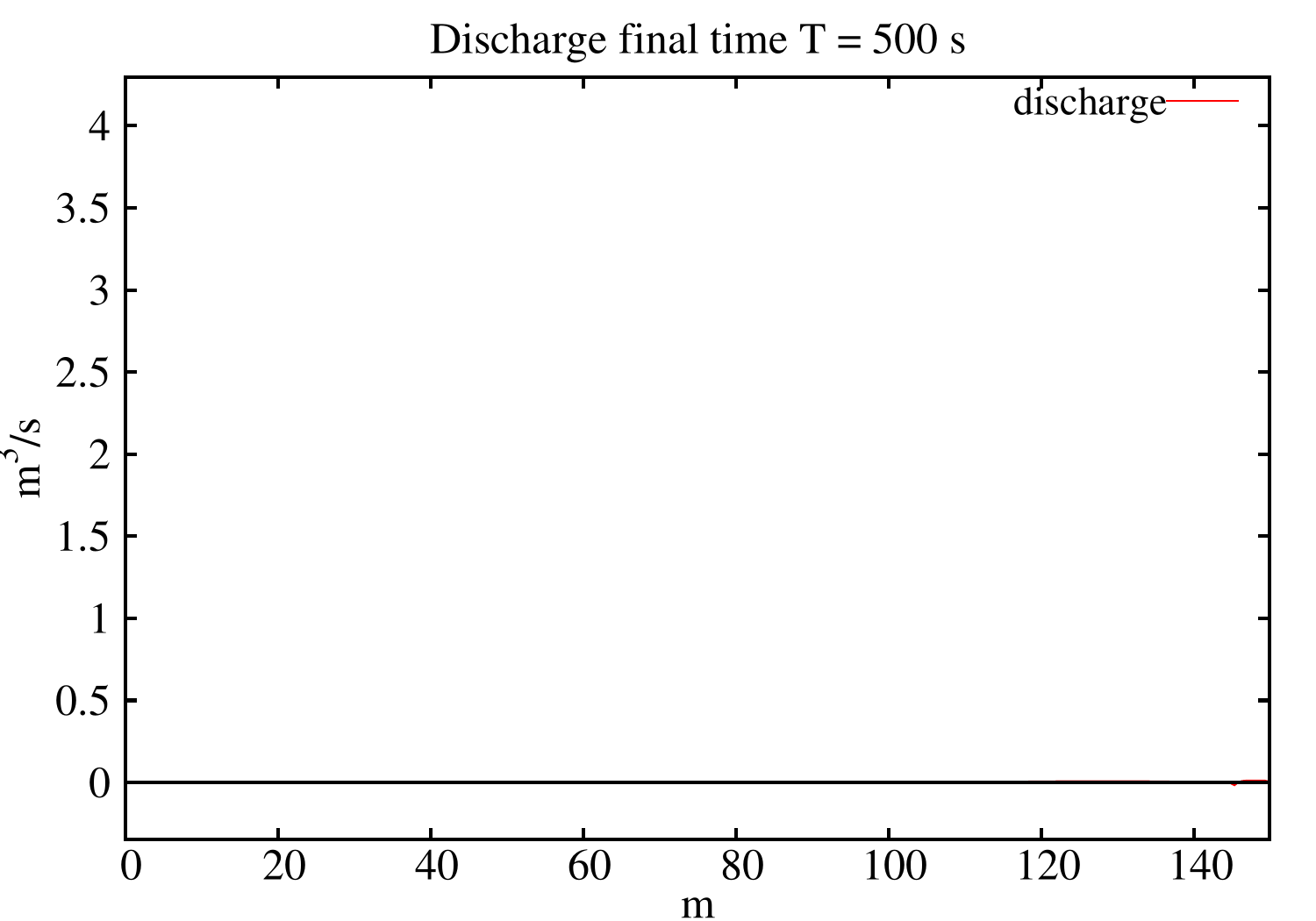}
\caption{Piezometric line (left) and discharge (right) at final time $T = 500 \:  s$\label{flaque3}.}
\end{figure}
%%%%%%%%%%%%%%%%%%%%%%%%%%%%%
\subsection{``Ghost waves approach'' versus ``\FKAl''}\label{versus}
%%%%%%%%%%%%%%%%%%%%%%%%%%%%%
We want to compare numerically the two approaches on a violent water hammer ``numerical'' test for a non uniform frictionless closed water pipes.

To this end, the numerical experiment is performed in the case of an expanding  $5 \:m$ long closed
circular water pipe  with $0$ slope. 
The upstream diameter is $2\;m$ and the downstream diameter is $3.2\;m$. The altitude of the main pipe axis is set to $Z = 1 \: m$. 

At the upstream boundary condition, the piezometric line (increasing linearly from  $1\:m$ to $3.2\:m$ in $5\;s$) is prescribed while 
the downstream discharge is kept constant equal to $0\:m^3/s$ (see  figure \ref{Init1}). 
The simulation starts from  a still water free surface  steady state where the height of the upstream is 
$1\:m$  (see figure \ref{Init1}) and the discharge  is null. 

Other parameters are 
$$
\begin{array}{lcl}
\hbox{Discretisation points}& : & 100,\\ 
\hbox{Delta x }(m)& : &0.05,\\
\hbox{CFL }& : & 0.8,\\
\hbox{Simulation time } (s) & : &5,\\
\hbox{Sound speed } (ms^{-1})& : &20.
\end{array}
$$
Let us mention that we have already used this numerical test case to compare the kinetic scheme using the ``ghost waves approach'' with 
the VFRoe scheme presented in \cite{BEG09}, see \cite[Figure 3]{BEG11_2}.

This numerical test intends to reproduce a ``sharp''  water hammer experiment inducing large oscillations of the piezometric level 
and the discharge as  showed in figures \ref{pointtransition1} and \ref{pointtransition2}. From a numerical point of view, 
it is a ``hard'' numerical test.
In order to validate numerically this approach and due to the lack of experimental data in the case of variable cross section pipes, we compare
the result of the presented numerical scheme with those obtained by the upwinded VFRoe scheme \cite{BEG09}. 
Results are represented in figures \ref{pointtransition1} and \ref{pointtransition2} where  
we have plotted the piezometric line, especially the transition point at different times $t=1.6 \: s$, $t=1.7 \: s $, $t=1.8 \: s $, $t=1.9 \: s $. 
In figures \ref{pointtransition1} and  \ref{pointtransition2}, the left side to the transition point corresponds to a free surface state 
and the right one to a pressurized except at $t=1.9 \: s$ where we can observe two transition points due to the pressurized state propagating from the downstream end. 
The behavior of the two methods are in a good agreement and particularly  
with respect to the localization of transition points. This short and ``sharp'' water hammer test allows us to validate numerically the two approaches for capturing the
transition between free surface and pressurized flow.
\begin{figure}[H]
\begin{center}
\subfigure[ Initial state.]
{
\includegraphics[height=5.5cm]{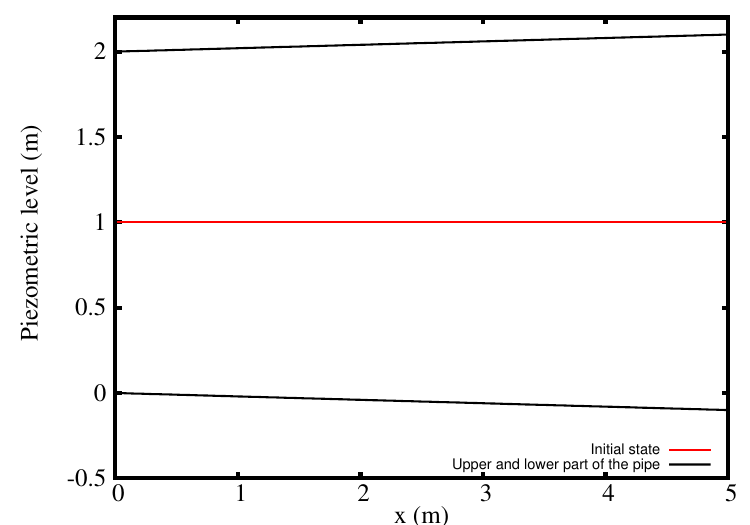}
}
\subfigure[ Prescribed upstream boundary condition.]
{
\includegraphics[height=5.5cm]{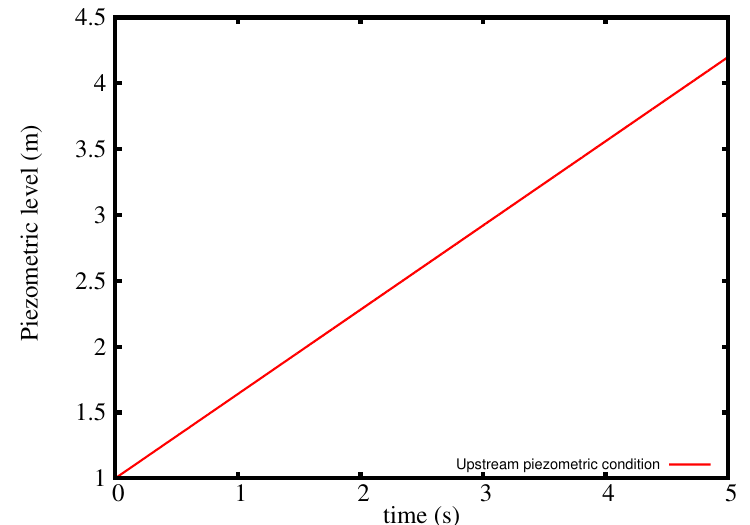}
}
\caption[Optional caption for list of figures]{Initial state and boundary conditions.}
\label{Init1}
\end{center}
\end{figure}
\begin{figure}[H]
\begin{center}
\subfigure[  $t = 1.6\; s$.]
{
\includegraphics[height=5.5cm]{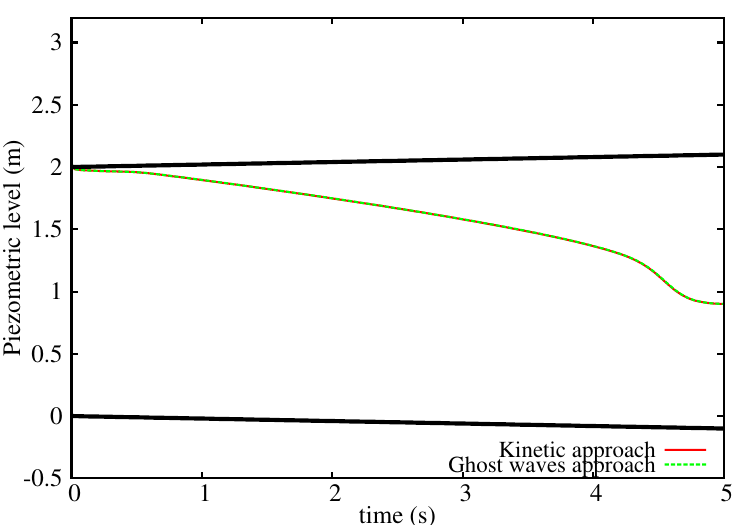}
}
\subfigure[ $t = 1.7\; s$.]
{
\includegraphics[height=5.5cm]{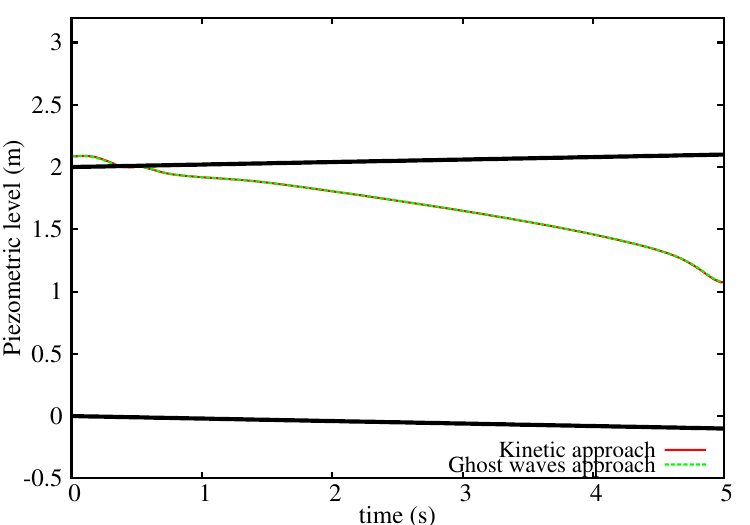}
}
\caption[Optional caption for list of figures]{Water hammer test case in non uniform closed water pipe.}
\label{pointtransition1}
\end{center}
\end{figure}

\begin{figure}[H]
\begin{center}
\subfigure[  $t = 1.8\; s$.]
{
\includegraphics[height=5.5cm]{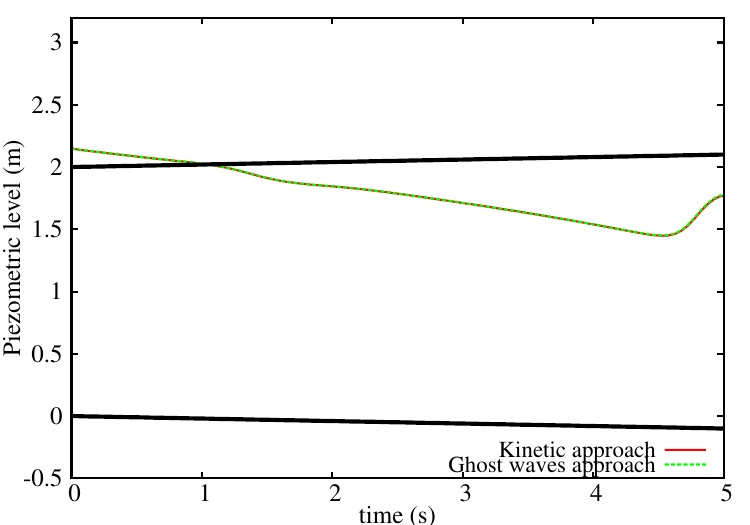}
}
\subfigure[  $t = 1.9\; s$.]
{
\includegraphics[height=5.5cm]{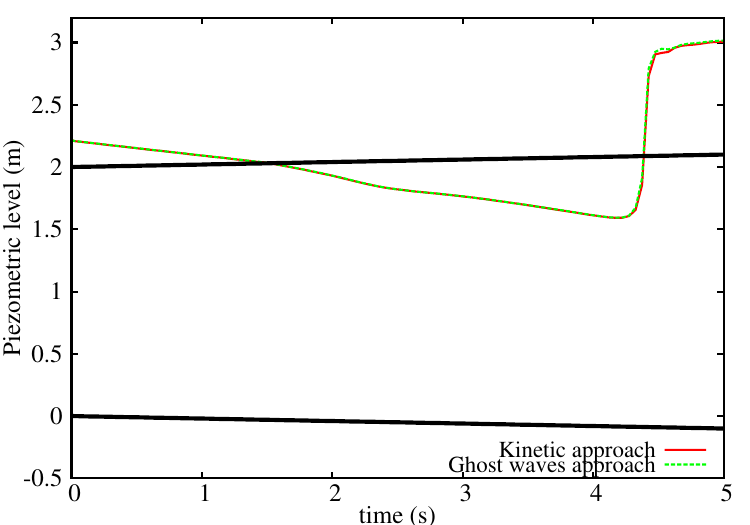}
}
\caption[Optional caption for list of figures]{Water hammer test case in non uniform closed water pipe.}
\label{pointtransition2}
\end{center}
\end{figure}
\section{Conclusion and perspectives}
%%%%%%%%%%%%%%%%%%%%%%%%%%%%%
We have proposed in this work a new manner to extend the numerical kinetic scheme with reflections build by Perthame and Simeoni \cite{PS01},
to closed water pipes with varying sections and not only to rectangular closed water pipes.  This scheme is wet area conservative and under a CFL condition
preserves the positivity of the wet area. 

As a well known feature of general kinetic schemes, we are able to ``naturally'' deal with flows where a flooding zone may be present. 
This key property was not solved by the previous VFRoe scheme that we proposed in \cite{BEG09} without introducing a cut-off function for the wetted
area which may causes a loss of conservativity.

The \textbf{PFS} model is numerically solved by a  kinetic  scheme with reflections  using the interfacial upwind of all the source terms into  the
numerical fluxes. 

As mentioned in \cite{BG07,BEG09} this numerical method reproduces correctly  laboratory tests for uniform pipes (Wiggert's test case) and 
can deal with multiple transition points between the two types of flows. The code to code comparison for pressurized flows in uniform pipes
has proved the robustness of the method. But due to the lack of experimental data for 
drying and flooding flows, we have only shown the behavior of the  piezometric line which seems reasonable (at less no major difference was observed).
For non uniform pipes, the two numerical schemes are in a very good agreement even though we  are not in possession of experimental data.

We are at the present time interested in the construction of a class of  ``in cell entropy satisfying'' schemes consistent with the numerical approximations of 
hyperbolic systems with source terms.  

The next step is to take into account the air entrainment which may have non negligible effects on the behavior of the
piezometric head. A first approach has been derived in the case of perfect fluid and perfect gas modeled s a bilayer model based on the
\textbf{PFS} model \cite{BEG13}.

\section*{Acknowledgements}
This work is supported by the ``Agence Nationale de la Recherche'' referenced by  ANR-08-BLAN-0301-01 and 
the  second author was  supported by  the ERC Advanced Grant FP7-246775 NUMERIWAVES. This work was finalized while the third author was visiting 
BCAM--Basque Center for Applied Mathematics, Derio, Spain, and partially supported by  the ERC Advanced Grant FP7-246775 NUMERIWAVES. The third author wishes to thank Enrique Zuazua for his kind hospitality.

Moreover, the  authors wish to thank the referees  for their  remarks and the careful reading of the numerical scheme presented in this paper.
%%%%%%%%%%%%%%%%%%%%%%%%%%%%%
\bibliographystyle{siam}

%\bibliography{cinemix}
\end{document}